\documentclass[reqno, 12pt]{amsart}

\usepackage{amsmath}
\usepackage{amsfonts}
\usepackage{amssymb}
\usepackage{amsthm}
\usepackage{mathrsfs}
\usepackage{bbm}
\usepackage{graphicx, color}
\usepackage{tikz}
\usetikzlibrary{cd, intersections, calc, decorations.pathmorphing, arrows, decorations.pathreplacing}

\usepackage{mathtools}

\usepackage{enumerate}

\definecolor{refkey}{rgb}{0.9451,0.2706,0.4941}
\definecolor{labelkey}{rgb}{0.9451,0.2706,0.4941}

\definecolor{mygreen}{rgb}{0,0.7,0.3}
\definecolor{myblue}{rgb}{0,0.50,1.20}
\definecolor{myorange}{rgb}{1,0.38,0}

\usepackage{subfiles}
\usepackage[colorlinks, linkcolor=blue, citecolor=red, urlcolor=cyan, pagebackref]{hyperref}

\numberwithin{equation}{section}

\usepackage{cleveref}
\crefname{thm}{Theorem}{Theorems}
\crefname{cor}{Corollary}{Corollaries}
\crefname{lem}{Lemma}{Lemmas}
\crefname{sublem}{Sublemma}{Sublemmas}
\crefname{prop}{Proposition}{Propositions}
\crefname{dfn}{Definition}{Definitions}
\crefname{defi}{Definition}{Definitions}
\crefname{ex}{Example}{Examples}
\crefname{claim}{Claim}{Claims}
\crefname{conj}{Conjecture}{Conjectures}
\crefname{conv}{Notation}{Notations}
\crefname{rem}{Remark}{Remarks}
\crefname{rmk}{Remark}{Remarks}
\crefname{prob}{Problem}{Problems}
\crefname{figure}{Figure}{Figures}
\crefname{table}{Table}{Tables}
\crefname{section}{Section}{Sections}
\crefname{subsection}{Section}{Sections}
\crefname{appendix}{Appendix}{Appendices}

\newtheorem{thm}{Theorem}[section]
\newtheorem{prop}[thm]{Proposition}
\newtheorem{cor}[thm]{Corollary}
\newtheorem{lem}[thm]{Lemma}

\theoremstyle{definition}
\newtheorem{dfn}[thm]{Definition}
\newtheorem{defi}[thm]{Definition}
\newtheorem{ex}[thm]{Example}

\newtheorem{conj}[thm]{Conjecture}
\newtheorem{conv}[thm]{Notation}
\theoremstyle{remark}
\newtheorem{rmk}[thm]{Remark}
\newtheorem{rem}[thm]{Remark}
\newtheorem{prob}[thm]{Problem}

\newcommand{\corrected}
{\textcolor{red}{(corrected)}}
\newcommand{\CORRECTED}
{}

\newcommand{\too}{\longrightarrow}

\newcommand{\Too}{\Longrightarrow}

\newcommand{\oot}{\longleftarrow}

\newcommand*{\chom}{\mathcal{H}\kern -.5pt om}

\newcommand{\bZ}{\mathbb{Z}}
\newcommand{\bQ}{\mathbb{Q}}
\newcommand{\bR}{\mathbb{R}}
\newcommand{\bC}{\mathbb{C}}

\newcommand{\bT}{\mathbb{T}}

\newcommand{\bP}{\mathbb{P}}

\newcommand{\bA}{\mathbb{A}}
\newcommand{\bE}{\mathbb{E}}
\newcommand{\bH}{\mathbb{H}}

\newcommand{\bs}{{\mathsf{s}}}

\newcommand{\spe}{|M_\partial|}

\newcommand{\ve}{b}
\newcommand{\bdX}{\boldsymbol{X}}

\newcommand{\bdA}{\boldsymbol{A}}

\newcommand{\A}{\mathcal{A}}
\newcommand{\cA}{\mathcal{A}}

\newcommand{\cC}{\mathcal{C}}

\newcommand{\cE}{\mathcal{E}}
\newcommand{\cF}{\mathcal{F}}

\newcommand{\cN}{\mathcal{N}}

\newcommand{\cR}{\mathcal{R}}

\newcommand{\cU}{\mathcal{U}}

\newcommand{\X}{\mathcal{X}}
\newcommand{\cX}{\mathcal{X}}

\newcommand{\fF}{\mathfrak{F}}
\newcommand{\fS}{\mathfrak{S}}

\newcommand{\sfa}{\mathsf{a}}
\newcommand{\sfx}{\mathsf{x}}

\newcommand{\tri}{\triangle}
\newcommand{\sgn}{\mathrm{sgn}}
\newcommand{\trop}{\mathrm{trop}}
\newcommand{\pos}{\mathbb{R}_{>0}}
\newcommand{\stab}{\mathrm{stab}}

\newcommand{\ML}{\mathcal{ML}}

\newcommand{\MF}{\mathcal{MF}}

\newcommand{\dML}{\widetilde{\mathcal{ML}}}
\newcommand{\eML}{\widehat{\mathcal{ML}}}
\newcommand{\dMF}{\widetilde{\mathcal{MF}}}

\newcommand{\Tri}{\mathrm{Tri}}
\newcommand{\bTri}{\bT \mathrm{ri}}
\newcommand{\tr}{\mathsf{T}}

\newcommand{\bExch}{\bE \mathrm{xch}}
\newcommand{\uf}{\mathrm{uf}}
\newcommand{\f}{\mathrm{f}}

\newcommand{\Teich}{Teichm\"uller}

\newcommand{\hL}{\widehat{L}}

\newcommand{\bep}{\boldsymbol{\epsilon}}
\newcommand{\bsfx}{\boldsymbol{\sfx}}
\newcommand{\bsfa}{\boldsymbol{\sfa}}

\DeclareMathOperator{\im}{\mathrm{im}}

\DeclareMathOperator{\interior}{\mathrm{int}}

\newcommand{\oline}[1]{\overline{#1}}

\makeatletter
\newcommand{\oset}[3][0ex]{%
  \mathrel{\mathop{#3}\limits^{
    \vbox to#1{\kern-2\ex@
    \hbox{$\scriptstyle#2$}\vss}}}}
\makeatother
\newcommand{\overbar}[1]{\oset{#1}{-\!\!\!-\!\!\!-}}


\makeatletter
\newcommand{\osetnear}[3][0ex]{%
  \mathrel{\mathop{#3}\limits^{
    \vbox to#1{\kern-.3\ex@
    \hbox{$\scriptstyle#2$}\vss}}}}
\makeatother
\newcommand{\overbarnear}[1]{\osetnear{#1}{-\!\!\!-\!\!\!-}}

\newcommand\qarrow[2]{\draw[->,shorten >=4pt,shorten <=4pt] (#1) -- (#2) [thick];} 

\tikzset{
  mid arrow/.style={postaction={decorate,decoration={
        markings,
        mark=at position .5 with {\arrow[#1]{stealth}}
      }}},
}

\tikzset{
  symbol/.style={
    draw=none,
    every to/.append style={
      edge node={node [sloped, allow upside down, auto=false]{$#1$}}}
  }
}

\tikzset{
    squigarrow/.style={-{Classical TikZ Rightarrow[length=4pt]}, decorate, decoration={snake, amplitude=1.8pt, pre length=2pt, post length=3pt}}
}

\title[A characterization of pA via tropical cluster transformations]{A characterization of pseudo-Anosov mapping classes via tropical cluster transformations}

\author[Tsukasa Ishibashi]{Tsukasa Ishibashi}
\address{Tsukasa Ishibashi, Mathematical Institute, Tohoku University, 
6-3 Aoba, Aramaki, Aoba-ku, Sendai, Miyagi 980-8578, Japan.}
\email{tsukasa.ishibashi.a6@tohoku.ac.jp}

\author[Shunsuke Kano]{Shunsuke Kano}
\address{Shunsuke Kano, Research Alliance Center for Mathematical Sciences, Tohoku University
6-3 Aoba, Aramaki, Aoba-ku, Sendai, Miyagi 980-8578, Japan.}
\email{s.kano@tohoku.ac.jp}

\date{\today}

\begin{document}

\begin{abstract}
We give a characterization of generic pseudo-Anosov mapping classes purely in terms of their expressions in the shear coordinates, thus giving an answer to a problem raised by Papadopoulos--Penner \cite{PP93}. 
This characterization has a cluster algebraic generalization called the \emph{sign stability} introduced in \cite{IK19}. 
By combining with the results in \cite{IK19}, we see that the algebraic entropies of the cluster $\A$- and $\X$-transformations induced by a generic pseudo-Anosov mapping class both coincide with its topological entropy.
\end{abstract}

\maketitle

\setcounter{tocdepth}{1} 
\tableofcontents

\section{Introduction}



The \emph{mapping class group} of an oriented closed surface $\Sigma$ is the group of isotopy classes of orientation-preserving diffeomorphisms on $\Sigma$. 
It naturally acts on the \Teich\ space $T(\Sigma)$ of $\Sigma$ properly discontinuously, and the quotient orbifold is the celebrated moduli space of Riemann surfaces. 
This action extends continuously to the Thurston compactification of the \Teich\ space, which is obtained by attaching the space of \emph{projective measured geodesic laminations} at infinity. 
In terms of certain fixed point properties of this action, each mapping class is classified into three types: periodic, reducible, and pseudo-Anosov \cite{Th88}. 
This result is called the \emph{Nielsen--Thurston classification}, which gives a beautiful generalization of the elliptic-parabolic-hyperbolic trichotomy for elements of the modular group $PSL_2(\bZ)$. 

For a punctured surface $\Sigma$ (\emph{i.e.}, a closed surface equipped with a finite set of marked points), the \emph{decorated \Teich\ theory} \cite{Pen87} provides a combinatorial tool for study of the \Teich--Thurston theory mentioned above. More precisely, there are two variants $\widetilde{T}(\Sigma)$ and $\widehat{T}(\Sigma)$ of the \Teich\ space called the \emph{decorated} and \emph{enhanced \Teich\ spaces}, repsectively. These spaces have distinguished charts associated with ideal triangulations, and the action of the mapping class group has a rational expression in terms of these coordinate systems. A similar construction applies for the spaces $\dML(F)$ and $\eML(F)$ of 
\emph{decorated} and \emph{enhanced measured geodesic laminations} (or equivalently, \emph{measured foliations}) with respect to a hyperbolic structure $F$ on $\Sigma$, which have piecewise-linear (PL for short) coordinate systems associated with ideal triangulations \cite{PP93,BKS}. The actions of the mapping class group on these spaces of laminations are expressed as PL maps in these coordinates, which are the \emph{tropical analogues} of the expressions of the actions on the corresponding \Teich\ spaces. 
The relations between the spaces mentioned above are summarized in the following diagram:
\[
\begin{tikzcd}
\widetilde{T}(\Sigma) \ar[d, twoheadrightarrow]\\
T(\Sigma) \ar[r, hook] & \widehat{T}(\Sigma),
\end{tikzcd}
\qquad
\begin{tikzcd}
\widetilde{\ML}(F) \ar[d, twoheadrightarrow]\\
\ML_0(F) \ar[r, hook] & \widehat{\ML}(F).
\end{tikzcd}
\]
where $\ML_0(F)$ denotes the space of measured geodesic laminations with compact support. 
In the last section in their paper \cite{PP93}, Papadopoulos and Penner raised the following problem:

\begin{prob}[Papadopoulos--Penner \cite{PP93}]\label{prob:PP question}
Characterize the Nielsen--Thurston type of a mapping class in terms of its rational or PL coordinate expressions.
\end{prob}

\subsection{Pseudo-Anosov property and the sign stability}
While the original problem of Papadopoulos--Penner arose in the setting of the PL coordinate system on $\dML(F)$ (the \emph{tropical $\lambda$-length coordinates}), we attack the problem by using the PL coordinate system on $\eML(F)$ (the \emph{shear coordinates}). Specifically, we focus on the following structures on $\eML(F)$. 

For each labeled ideal triangulation $(\tri,\ell)$ of a punctured surface $\Sigma$, which is an ideal triangulation $\tri$ of $\Sigma$ together with a bijection $\ell: I \xrightarrow{\sim} \tri$, we have a PL coordinate system
\begin{align*}
    \bsfx^{(\tri,\ell)}=(\sfx_i^{(\tri,\ell)})_i: \eML(F) \xrightarrow{\sim} \bR^{I} 
\end{align*}
formed by the shear coordinates (\cref{thm:lam_shear}). In this coordinate system, the natural action of a mapping class $\phi \in MC(\Sigma)$ on $\eML(F)$ has a coordinate expression
\begin{align*}
    \phi^{(\tri,\ell)}:= \bsfx^{(\tri,\ell)} \circ \phi \circ (\bsfx^{(\tri,\ell)})^{-1}: \bR^{I} \to \bR^{I},
\end{align*}
which is a PL map. We can study the \emph{domains of linearity} of this PL map (namely, the maximal domains on which the action is linear) as follows. Take a sequence $\gamma: (\tri,\ell) \to \phi^{-1}(\tri,\ell)$ of flips and permutations of labels, which we call a \emph{representation path} of $\phi$ (\cref{d:rep_path}). It gives a factorization of the PL map $\phi^{(\tri,\ell)}$ into certain elementary transformations (called the \emph{tropical cluster $\X$-transformations} \eqref{eq:sign x-cluster}), which tells us that the PL map $\phi^{(\tri,\ell)}$ fails to be linear at the points where one of the shear coordinates associated with the flipped edges in $\gamma$ vanishes. In other words, the choice of $\gamma$ gives a decomposition
\begin{align*}
    \eML(F) = \bigcup_{\bep
    \in \{+,-\}^h} \cC_\gamma^{\bep},
\end{align*}
where $h$ is the number of flips in $\gamma$, and $\cC_\gamma^{\bep}$ is the cone obtained as the closure of the points $L \in \eML(F)$ whose shear coordinate of the $i$-th flipped edge in $\gamma$ has the sign $\epsilon_i$ for $i=0,\dots,h-1$ (see \eqref{eq:sign_cone}). The restriction of $\phi^{(\tri,\ell)}$ to each cone $\cC_\gamma^{\bep}$ is a linear map. 
In view of these structures, we make the following definition:

\begin{dfn}[\cref{d:sign stability} and \cref{d:uniform stability}]\label{introdef:sign_stability}
Fix an $\bR_{>0}$-invariant subset $\Omega \subset \eML(F)$. 
\begin{enumerate}
    \item We say that a sequence $\gamma: (\tri,\ell) \to \phi^{-1}(\tri,\ell)$ is \emph{sign-stable} on $\Omega$ if there exists a sign $\bep_\gamma^{\stab}$ such that every $\phi$-orbit of points in $\Omega$ eventually enters the cone $\cC_\gamma^{\bep_\gamma^{\stab}}$. 
    \item We say that a mapping class $\phi \in MC(\Sigma)$ is \emph{uniformly sign-stable} if every representation path $\gamma: (\tri,\ell) \to \phi^{-1}(\tri,\ell)$ is sign-stable on $\Omega=\eML_\bQ(F)$. 
\end{enumerate}
\end{dfn}
The condition (1) implies that the presentation matrices\footnote{Here, the presentation matrix of $\phi^{(\tri,\ell)}$ at a differentiable point $L$ means the presentation matrix of the tangent map $(d\phi^{(\tri,\ell)})_L$ with respect to the shear coordinates.}
of the action $\phi^{(\tri,\ell)}$ along each orbit `stabilizes' to a common one, denoted by $E_{\phi,\Omega}^{(\tri,\ell)}$. If the domain is chosen to be a `canonical' one $\Omega=\Omega_\tri^{\mathrm{can}} \subset \eML_\bQ(F)$, the spectral radius $\lambda_\phi^\tri \geq 1$ of $E_{\phi,\Omega}^{(\tri,\ell)}$ is called the \emph{cluster stretch factor}. 

Observe that the concepts of the sign stability and the cluster stretch factor are formulated purely in terms of the shear coordinates. The condition (1) mimics the convergence of \emph{RLS sequences} of trian track splittings for pseudo-Anosov mapping classes \cite{PP87}, and the definition of the cluster stretch factor resembles the description of the pseudo-Anosov stretch factor as the spectral radius of the associated \emph{transition matrix} \cite{PH,BH95}.
The following is our main theorem:

\begin{thm}[\cref{t:pA_SS_NS}]\label{introthm:pA_SS_NS}
Let $\Sigma$ be a punctured surface.
For a mapping class $\phi \in MC(\Sigma)$, the following conditions are equivalent:
\begin{enumerate}
    \item $\phi$ is generic pseudo-Anosov.
    \item $\phi$ is uniformly sign-stable.
    \item $\phi$ has North-South dynamics on $\bP \eML(F)$, and the attracting and repelling points are $\cX$-filling.
\end{enumerate}
In this case, the cluster stretch factor $\lambda_\phi^\tri$ coincides with the pseudo-Anosov stretch factor of $\phi$ for any ideal triangulation $\tri$.
\end{thm}
In particular, the equivalence $(1) \Longleftrightarrow (2)$ gives a characterization of the generic pseudo-Anosov mapping classes purely in terms of the shear coordinates, thus giving an answer to \cref{prob:PP question}. The condition (3) also gives a dynamical characterization. 

\begin{rem}
It turns out that the domains of linearity of the PL action on $\dML(F)$ is also characterized by the signs of shear coordinates under the projection $\dML(F) \to \eML(F)$. Moreover, an analogue of $(1) \Longleftrightarrow (2)$ for $\dML(F)$ does not hold. 
This is the reason why we choose to use the shear coordinates rather than the tropical $\lambda$-lengths. 
Non-generic pseudo-Anosov mapping classes are discussed in \cite[Part II]{Kan23}. 
\end{rem}

\subsubsection*{Dehn twists}
We also study the sign stability of Dehn twists, which are of reducible class. We see that a Dehn twist satisfying a mild condition admits a representation path that is sign-stable on a certain domain $\Omega$, and its cluster stretch factor is $1$ (\cref{lem:Dehn_SS}).
On the other hand, it also has a representation path which is not sign-stable on the same domain (\cref{ex:Dehn_SS}). This gives a counterexample to the naive conjecture \cite[Conjecture 1.3]{IK19} which asserted the independence of sign stability on representation paths.

\subsection{Sign stability for mutation loops on cluster varieties}
The concept of sign stability has been introduced in \cite{IK19} to generalize the pseudo-Anosov property of mapping classes for \emph{mutation loops} on cluster varieties \cite{FG09}. Our \cref{introthm:pA_SS_NS} confirms this principle. Let us explain the correspondence. 

The coordinate structures of the \Teich\ and lamination spaces mentioned above fit in with the language of \emph{cluster algebra/variety}, which is introduced by Fomin--Zelevinsky \cite{FZ-CA1} and Fock--Goncharov \cite{FG03}. In terms of the cluster variety, there exists a seed pattern $\bs_\Sigma$ associated with a punctured surface $\Sigma$ formed by the ideal (or tagged) triangulations, and the \Teich\ and lamination spaces are uniformly reconstructed from the associated \emph{cluster varieties} $\A_\Sigma:=\A_{\bs_\Sigma}$ and $\X_\Sigma:=\X_{\bs_\Sigma}$, as follows \cite{FG07,FST}:
\begin{align}
    \begin{aligned}\label{eq:TT_cluster}
    \A_\Sigma(\pos) &\cong \widetilde{T}(\Sigma),& \X_\Sigma(\pos) &\cong \widehat{T}(\Sigma), \\
    \A_\Sigma(\bR^\trop) &\cong \dML(\Sigma),& \X_\Sigma(\bR^\trop) &\cong \eML(\Sigma).
    \end{aligned}
\end{align}
Here the \emph{positive structures} of the cluster varieties allow us to consider their semifield-valued points, and we have used the semifields $\pos=(\pos,+,\cdot)$ and $\bR^\trop=(\bR,\min,+)$. 
The natural actions of the mapping class group $MC(\Sigma)$ on these spaces coincides with the one through an embedding $MC(\Sigma) \to \Gamma_{\bs_\Sigma}$ into the \emph{cluster modular group}, which is a group consisting of \emph{mutation loops} and acting on the cluster varieties. 

The concepts of sign stability given in \cref{introdef:sign_stability} make sense for a general mutation loop, where we replace $\eML(F)$ with $\X_\bs(\bR^\trop)$ for a general seed pattern $\bs$, and $\gamma$ with a sequence of mutations and permutations that represents a mutation loop $\phi \in \Gamma_\bs$. Examples of sign-stable mutation loops that do not arise from the surface topology are studied in \cite{IK19}. We also prove a slight generalization of \cref{introthm:pA_SS_NS} for a general element of $\Gamma_{\bs_\Sigma}$ in \cref{sec:Appendix}.

\subsubsection*{Relation to a prior work \cite{Ish19}}
Based on the correspondence \eqref{eq:TT_cluster}, the first-named author studied an analogue of the Nielsen--Thurston classfication for cluster modular groups in \cite{Ish19} by introducing three types for mutation loops: periodic, cluster-reducible, cluster-pseudo-Anosov. The basic idea is to distinguish them according to certain fixed point properties of their actions on the \emph{tropical compactification} \cite{FG16,Le16} of the spaces $\A_\bs(\pos)$ and $\X_\bs(\pos)$, similarly to the original Nielsen--Thurston's spirit. 
However, there was a fault that the cluster-pseudo-Anosov property is weaker than the original pseudo-Anosov property (\cite[Example 3.13]{Ish19}). Moreover, an analogue of the stretch factor was also missing. Our concept of sign stability improves these points. We give a comparison between the uniform sign stability and the cluster-pseudo-Anosov property in \cref{prop:uniform_SS_cluster_pA} in a from applicable to general mutation loops.

\subsection{The topological and algebraic entropies}
As an application of our study of sign stability, we compute the \emph{algebraic entropy} \cite{BV99} $\cE^a_\phi$ (resp. $\cE^x_\phi$) of the action of a mapping class $\phi$ on the cluster variety $\A_\Sigma$ (resp. $\X_\Sigma$), based on the main result of \cite{IK19}. Moreover, we compare them with the \emph{topological entropy} \cite{AKM65} $\cE^{\mathrm{top}}_\phi$.

In fact, the existence of at least one sign-stable representation path on $\Omega_\tri^{\mathrm{can}}$ is enough to compute the algebraic entropies, as demonstrated in \cite{IK19}. Therefore, we can also compute the algebraic entropy of a Dehn twist as well as pseudo-Anosov mapping clasees. 


\begin{thm}[\cref{cor:cluster_Dehn_entropy,cor:pA_entropy}]\label{introthm:entropy_1}
We have the following comparison results:
\begin{enumerate}
    \item For a simple closed curve $C$ on a marked surface $\Sigma$ satisfying the condition \eqref{eq:condition_Dehn}, we have
    \[ \cE^a_{T_C} = \cE^x_{T_C} = \cE^{\mathrm{top}}_{T_C} = 0. \]
    \item For a punctured surface $\Sigma$ and a generic pseudo-Anosov mapping class $\phi \in MC(\Sigma)$, we have
     \[ \cE^a_\phi = \cE^x_\phi = \cE^{\mathrm{top}}_{\phi} = \log\lambda_\phi. \]
     Here $\lambda_\phi >1$ is the cluster stretch factor (=pseudo-Anosov stretch factor) of $\phi$. 
\end{enumerate}
\end{thm}
Thus we get a complete agreement of the two kinds of entropies associated with these mapping classes. 
Since the cluster varieties $\A_\Sigma$ and $\X_\Sigma$ are birationally isomorphic to certain extensions of the $PSL_2(\bC)$-character variety \cite{FG03}, this result also computes the algebraic entropies of the natural actions on these spaces. 
It should be compared with the work of Hadari \cite{Hadari}, where the algebraic entropies of the actions of a mapping class on various character varieties are studied.

\subsection*{Organization of the paper}

In \cref{sec:lamination}, we recall the notion of measured geodesic laminations and their shear coordinates, and mapping class group action on them.
We also recall the Nielsen--Thurston classification for mapping classes. 

In \cref{sec:sign_stability}, after reviewing the combinatorial description of the mapping class group action, we introduce basic concepts around the sign stability. It basically follows \cite{IK19} but includes the treatment of permutations of labels. In \cref{subsec:cluster-pA}, we compare the uniform sign stability with the cluster-pseudo-Anosov property. 
%
%
%
We study the sign stability of Dehn twists in \cref{sec:Dehn twist}. This is partially intended to be a mild introduction to the sign stability. 

In \cref{sec:pA_SS_NS}, we investigate the uniform sign stability of pseudo-Anosov mapping classes on a punctured surface and prove the main theorem (\cref{introthm:pA_SS_NS}). 
%
%
In \cref{subsec:entropy}, we compute the algebraic entropy of mutation loops as an application of previous sections and \cite{IK19}, giving a proof of \cref{introthm:entropy_1}.

In \cref{sec:Appendix}, we study the sign stability of the signed mapping classes, which are the general elements of the cluster modular group $\Gamma_\Sigma$ associated with the seed pattern $\bs_\Sigma$.
The conclusion is that the statements parallel to the original case remain valid in this generalized setting.


\smallskip

\subsection*{Acknowledgements}
The authors are grateful to Hidetoshi Masai for informing us of the modern results on the North-South dynamics on the lamination spaces.
They also wish to thank Takeru Asaka for pointing out an error in the proof of \cref{p:shear_non-vanishing} (2) in the first version.
T. I. would like to express his gratitude to his former supervisor Nariya Kawazumi for his continuous guidance and encouragement in the earlier stage of this work. 
S. K. would like to thank Yoshihiko Mitsumatsu. The strategy of the proof of \cref{lem:asympt to U} is inspired by him. 
S. K. is also deeply grateful to his supervisor Yuji Terashima for his advice and giving him a lot of knowledge.

T. I. has been supported by JSPS KAKENHI Grant Number 18J13304, JP20K22304 and JP24K16914, and the Program for Leading Graduate Schools, MEXT, Japan.


\section{Measured geodesic laminations on a marked surface}\label{sec:lamination}
In this section, we recall basic concepts around measured geodesic laminations on a hyperbolic surface with geodesic boundary, following \cite{BB}.
The space $\eML(F)$ of measured geodesic laminations on a hyperbolic surface $F$ with suitable boundary conditions will be identified with the tropical cluster variety $\cX_\Sigma(\bR^\trop)$ associated with the corresponding marked surface $\Sigma$.


A marked surface $\Sigma$ is a compact oriented surface with a fixed non-empty finite set $M$ of \emph{marked points} on it. A marked point is called a \emph{puncture} if it lies in the interior of $\Sigma$, and a \emph{special point} otherwise. 
Let $P=P(\Sigma)$ (resp. $M_\partial=M_\partial(\Sigma)$) denote the set of punctures (resp. special points), so that $M=P \cup M_\partial$. 

A marked surface is called a \emph{punctured surface} if it has empty boundary (and hence $M_\partial=\emptyset$). 
We denote by $g$ the genus of $\Sigma$, $h$ the number of punctures, and $b$ the number of boundary components in the sequel. We always assume the following conditions:
\begin{enumerate}
    \item[(S1)] Each boundary component (if exists) has at least one marked point.
    \item[(S2)] $2(2g-2+h+b)+\spe >0$.
    \item[(S3)] If $g=0$ and $b=0$, then $h \geq 4$.
    \item[(S4)] $\Sigma$ is not a once-punctured monogon.
\end{enumerate}

An \emph{ideal arc} in $\Sigma$ is an isotopy class of a curve in $\Sigma$ with endpoints in $P \cup M_\partial$ which has no self-intersections except possibly at its endpoints, not contractible to one point, and  
not isotopic to a boundary segment connecting two consecutive special points.
An {\it ideal triangulation} $\tri$ of $\Sigma$ is a collection of ideal arcs such that
\begin{itemize}
\item each pair of ideal arcs in $\tri$ can intersect only at their endpoints;
\item each region complementary to $\tri$ is a triangle whose edges (resp. vertices) are ideal arcs or boundary segments (resp. marked points).
\end{itemize}
The conditions (S1) and (S2) ensure the existence of such an ideal triangulation $\tri$, and in particular the number $2(2g-2+h+b)+\spe >0$ in (S2) gives the number of triangles in $\tri$. The number of arcs is given by $3(2g-2+h+b)+\spe >0$.
To simplify the exposition, we will only consider ideal triangulations without self-folded triangles shown below, except for \cref{rem:self-folded,lem:ML_Casimir}.
\begin{align*}
\begin{tikzpicture}
\draw[blue](0,0) -- (0,1.8);
\draw[blue] (0,0) ..controls (45:0.5) and (1.5,2.5).. (0,2.5);
\draw[blue] (0,0) ..controls (135:0.5) and (-1.5,2.5).. (0,2.5);
\fill(0,0) circle(2pt);
\fill(0,1.8) circle(2pt);
\node[blue] at (-0.3,1.2) {$\alpha$};
\node[blue] at (0.9,1) {$\beta$};
\end{tikzpicture}
\end{align*}
We call $\alpha$ the \emph{self-folded edge}. 
The condition (S4) ensures that a triangulation without self-folded triangulations always exists, and these triangulations can be connected by flips inside this class \cite{FST}.   


\subsection{Measured geodesic laminations}\label{sec:geodesic_laminations}
We recall the \emph{measured geodesic laminations}. See \cite{PH} for basic notions for the case of surfaces without marked points. For the marked surface case, we follow the formulation given by Benedetti--Bonsante \cite{BB}. Their model is suited for the cut-and-paste operations of marked surfaces, which are useful in Section \ref{sec:pA_SS_NS}
and the accompanying paper \cite{IK2}.

Let $\Sigma$ be a marked surface.
Fix a complete hyperbolic structure $h$ on $\Sigma \setminus M$ such that each puncture corresponds to a funnel and each special point corresponds to a spike of a crown. See \cref{fig:Sigma^h}. 
Let $F$ be the hyperbolic surface obtained from $(\Sigma \setminus M,h)$ by truncating each funnel at the outer side of the shortest closed geodesic. 
For $p \in P$, let $\partial_p$ denote the resulting boundary component. 
We fix 
an orientation-preserving homeomorphism
\begin{align*}
    f : \Sigma \setminus M \to F^\circ:= F \setminus \bigcup_{p \in P} \partial_p,
\end{align*}
which maps a representative of each ideal arc to a geodesic so that 
\begin{itemize}
    \item each end incident to $m \in M_\partial$ is mapped to a complete geodesic ray, and
    \item each end incident to $p \in P$ is mapped to a geodesic perpendicular to $\partial F$. 
\end{itemize}
We will refer to $(F,f)$ as a marked hyperbolic surface with geodesic boundary \emph{of type $\Sigma$}.

\begin{figure}[h]
    \centering
    \begin{tikzpicture}[scale=.7]
    \draw  (-4.5,0) ellipse (3 and 2);
    \draw (-5.5,-0.5) .. controls (-5.5,-1.35) and (-3.5,-1.35) .. (-3.5,-0.5);
    \draw (-5.4,-0.8) .. controls (-5.4,-0.2) and (-3.6,-0.2) .. (-3.6,-0.8);
    \draw  (-5.5,0.65) ellipse (0.5 and 0.5);
    \node [fill, circle, inner sep=1.3pt] at (-3.5,0.65) {};
    \node [fill, circle, inner sep=1.3pt] at (-5.05,0.85) {};
    \node [fill, circle, inner sep=1.3pt] at (-5.95,0.85) {};
    \node [fill, circle, inner sep=1.3pt] at (-5.5,0.15) {};
    \node at (-3.1,0.6) {$p$};
    \draw[red] (-5.5,0.15) .. controls (-5.2,-0.35) and (-3.5,0) .. (-3.5,0.65);
    \node [red] at (-4.4,0.25) {$\alpha$};

    \draw  (3,0) ellipse (3 and 2);
    \draw [white, ultra thick](1.3,1.65) .. controls (1.6,1.8) and (2,1.9) .. (2.3,1.95);
    \draw [white, ultra thick](3.75,1.95) .. controls (4,1.9) and (4.3,1.8) .. (4.55,1.7);
    \draw (2,-0.5) .. controls (2,-1.35) and (4,-1.35) .. (4,-0.5);
    \draw (2.1,-0.8) .. controls (2.1,-0.2) and (3.9,-0.2) .. (3.9,-0.8);
    \draw (0.85,0.7) .. controls (1.35,0.95) and (1.35,2.3) .. (1.35,3.15) .. controls (1.35,2.3) and (1.8,1.8) .. (1.8,2.65) .. controls (1.8,1.8) and (2.25,2.3) .. (2.25,3.15) .. controls (2.25,2.3) and (1.35,2.3) .. (1.35,3.15);
    \draw (2.75,0.7) .. controls (2.25,0.95) and (2.25,2.3) .. (2.25,3.15);
    \draw (3.2,0.7) .. controls (3.7,0.95) and (3.75,1.95) .. (3.75,2.5);
    \draw (5.05,0.7) .. controls (4.55,0.95) and (4.5,1.95) .. (4.5,2.5);
    \draw [dashed](3.75,2.5) .. controls (3.75,2.75) and (3.75,2.9) .. (3.6,3.2);
    \draw [dashed](4.5,2.5) .. controls (4.5,2.75) and (4.5,2.9) .. (4.65,3.2);

    \draw  (10.5,0) ellipse (3 and 2);
    \draw [white, ultra thick](8.8,1.65) .. controls (9.1,1.8) and (9.5,1.9) .. (9.8,1.95);
    \draw [white, ultra thick](11.25,1.95) .. controls (11.5,1.9) and (11.8,1.8) .. (12.05,1.7);
    \draw (9.5,-0.5) .. controls (9.5,-1.35) and (11.5,-1.35) .. (11.5,-0.5);
    \draw (9.6,-0.8) .. controls (9.6,-0.2) and (11.4,-0.2) .. (11.4,-0.8);
    \draw (8.35,0.7) .. controls (8.85,0.95) and (8.85,2.3) .. (8.85,3.15) .. controls (8.85,2.3) and (9.3,1.8) .. (9.3,2.65) .. controls (9.3,1.8) and (9.75,2.3) .. (9.75,3.15) .. controls (9.75,2.3) and (8.85,2.3) .. (8.85,3.15);
    \draw (10.25,0.7) .. controls (9.75,0.95) and (9.75,2.3) .. (9.75,3.15);
    \draw (10.7,0.7) .. controls (11.2,0.95) and (11.25,1.95) .. (11.25,2.5);
    \draw (12.55,0.7) .. controls (12.05,0.95) and (12.05,1.95) .. (12.05,2.5);
    \draw(11.65,2.5) ellipse (0.4 and 0.2);
    \node at (12.5,2.5) {$\partial_p$};
    
    \node at (-4.5,-2.5) {$\Sigma$};
    \node at (3,-2.5) {$(\Sigma \setminus M, h)$};
    \node at (10.5,-2.5) {$F$};
    \node [red] at (12,.3) {$f(\alpha)$};
    \draw [red](9.3,2.65) .. controls (9.3,0.5) and (10,0.25) .. (10.5,0.25) .. controls (11,0.25) and (11.5,0.6) .. (11.6,1.35) .. controls (11.65,1.7) and (11.65,1.9) .. (11.65,2.3);
    \end{tikzpicture}
    \caption{Left: a topological surface $\Sigma$ and an ideal arc $\alpha$ (red arc); center: the surface $\Sigma \setminus M$ with a hyperbolic metric $h$; right: the hyperbolic surface $F$ and the arc $f(\alpha)$.}
    \label{fig:Sigma^h}
\end{figure}

A \emph{geodesic lamination} on $F$ is
a closed subset $G \subset F$ which is a disjoint union of complete simple geodesics.
A \emph{transverse measure} $\mu$ on $G$ is an assignment of a  number $\mu(\alpha) \geq 0$ to each transverse arc $\alpha \subset F$ to $G$ satisfying the following conditions:
\begin{enumerate}
    \item $\mu(\alpha) >0$ if and only if $\alpha \cap G \neq \emptyset$.
    \item If $\alpha_1$ and $\alpha_2$ are isotopic through transverse arcs, then $\mu(\alpha_1) = \mu(\alpha_2)$.
    \item $\mu(\alpha)$ is $\sigma$-additive in the sense that $\mu(\alpha) = \sum_{i=1}^\infty \mu(\alpha_i)$ for any countable partition $\alpha=\bigcup_{i=1}^\infty \alpha_i$.
\end{enumerate}
Here a transverse arc to $G$ is a simple arc $\alpha$ such that $\partial \alpha \subset \interior F \setminus G$ and transverse to each leaf of $G$. Such a pair $(G,\mu)$ is called a \emph{measured geodesic lamination} on $F$.
For a measured geodesic lamination $(G, \mu)$, we call $G$ its \emph{support}.
Let $\eML(F)$ denote the space of measured geodesic laminations on $F$.

\begin{ex}\label{ex:MGL}
Let $\{\gamma_i\}$ be a finite collection of mutually-disjoint complete simple geodesics on $F$ having no accumulation points in $F^\circ$. Examples of $\gamma_i$ include closed geodesics and bi-infinite geodesics with ends either entering spikes or spiralling into $\partial F$ (\cref{fig:spiralling leaves}). Let $G:= \bigcup_i \gamma_i \subset F$, and consider the transverse measure $\delta_G$ counting the number of intersections with $G$. Then $(G,\delta_G)$ is a measured geodesic lamination.  
\end{ex}

Here are some basic results on the structure of a measured geodesic lamination:

\begin{lem}[{\cite[Lemma 2.1]{BKS}}]\label{l:spiralling leaves}
Let $G$ be a geodesic lamination on $F$. 
Then for each boundary component $\partial_p \subset \partial F$ for $p \in P$, there exists an $\epsilon$-neighborhood $U_p$ such that every leaf of $G$ intersecting $U_p$ must spiral around $\partial_p$ (\cref{fig:spiralling leaves}). Moreover, the leaves in $U_p \cap G$ are locally isolated.
\end{lem}

We say that a geodesic lamination $G$ is \emph{compact} if it is disjoint from a neighborhood of the boundary of $F$. We denote $\ML_0(F) \subset \eML(F)$ the subset of compact measured geodesic laminations.

\begin{lem}[{\cite[Lemma 2.2]{BKS}, \cite[Corollary 1.7.3]{PH}}]\label{l:structure_of_lamination}
The support of any measured geodesic lamination $(G,\mu)$ on $F$ splits as the union of sub-laminations
\begin{align*}
    G = B \cup S \cup G_1 \cup \dots \cup G_k,
\end{align*}
where $B$ is the union of leaves that enter any neighborhood of the closed boundary or the spike of $F$, $S$ is a finite union of closed geodesics, and $G_j$ is a compact lamination whose leaves are bi-infinite geodesics which are dense in $G_j$ for $j=1,\dots,k$.
\end{lem}

\begin{rmk}
The case of $F$ with closed boundary components is treated in \cite{BKS} and \cite{PH}.
In the case where $F$ has a crown-shaped boundary, it can be reduced to the previous case by considering the hyperbolic surface $DF$ with the type of the double of the surface $\Sigma$ along the boundary. (See also \cref{subsec:proof_1}.)
\end{rmk}

The compact part $S \cup G_1 \cup \dots \cup G_k$ is classically well-studied. For example, see \cite{PH,CB}. 

Note that each leaf $g$ in the non-compact part $B$ determines a properly embedded arc in the surface $\overline{F \setminus \bigcup_{p \in P} U_p}$. We call its homotopy class the \emph{homotopy class of $g$}, which is independent of the neighborhoods $U_p$ in \cref{l:spiralling leaves}. 
We define the lamination signature 
\begin{align*}
    \sigma_G: P \to \{+,0,-\}
\end{align*}
of $G$ by $\sigma_G(p):=+$ (resp. $\sigma_G(p):=-$) if some leaves of $G$ spiral positively (resp. negatively) as in the top (resp. bottom) in \cref{fig:spiralling leaves} around the boundary component $\partial_p$, and $\sigma_G(p):=0$ otherwise. 

\begin{figure}
    \centering
    \begin{tikzpicture}[scale=.9]
    \draw  (-0.5,-8) ellipse (0.4 and 1.35);
    \draw (-0.5,-6.65) .. controls (-2,-6.65) and (-2.5,-6.65) .. (-4,-6.15);
    \draw (-0.5,-9.35) .. controls (-2,-9.35) and (-2.75,-9.4) .. (-4,-9.85);
    \draw[red] (-3,-6.45) .. controls (-2.5,-6.5) and (-2.55,-9.4) .. (-2.25,-9.4);
    \draw[red] (-1.85,-6.62) .. controls (-1.55,-6.65) and (-1.6,-9.35) .. (-1.2,-9.35);
    \draw [red](-1.3,-6.65) .. controls (-1.4,-6.65) and (-1.2,-9.35) .. (-0.8,-9.35);
    \draw [red](-0.9,-6.65) .. controls (-1.15,-6.65) and (-1.1,-9.35) .. (-0.65,-9.35);
    \draw [red](-0.7,-6.65) .. controls (-1.1,-6.65) and (-1.05,-9.35) .. (-0.57,-9.35);
    \draw [red, dashed](-3,-6.45) .. controls (-3.5,-6.5) and (-3.5,-7) .. (-4,-7.5);
    \draw [red, dashed](-1.85,-6.65) .. controls (-2.2,-6.65) and (-1.95,-9.4) .. (-2.25,-9.4);
    \node at (0.45,-8) {$\partial_p$};
    
    \fill [gray!30]  (5,-6.5) rectangle (5.2,-9.5);
    \draw  (5,-8) node (v2) {} ellipse (1.5 and 1.5);
    \draw (5,-6.5) -- (5,-9.5) node (v1) {};
    \draw  plot[smooth, tension=.7] coordinates {(v1)};
    \draw (5,-9.5);
    \draw (5,-9.5);
    \node at (5,-6.1) {$\widetilde{\partial_p}$};
    \draw [red] (4.2,-6.75) .. controls (4.65,-7.35) and (5,-8.5) .. (5,-9.5);
    \node [red] at (-3,-6.05) {$G$};
    \node [red] at (3.95,-6.4) {$\widetilde{G}$};
    \draw [dashed](v2) -- (3.5,-8);
    \draw [dashed](5,-7.6) .. controls (4.5,-7.6) and (4.05,-7.5) .. (3.65,-7.25);
    \draw [dashed](5,-8.4) .. controls (4.5,-8.4) and (4.05,-8.5) .. (3.65,-8.75);
    \draw [dashed](5,-7.25) .. controls (4.55,-7.25) and (4.15,-7.05) .. (4,-6.9);
    \draw [dashed](5,-8.75) .. controls (4.55,-8.75) and (4.15,-8.9) .. (4,-9.1);
    \draw [dashed](5,-6.95) .. controls (4.75,-6.95) and (4.55,-6.85) .. (4.4,-6.6);
    \draw [dashed](5,-9.05) .. controls (4.75,-9.05) and (4.5,-9.15) .. (4.35,-9.35);
    \draw [dashed](5,-6.7) .. controls (4.9,-6.7) and (4.7,-6.65) .. (4.7,-6.55);
    \draw [dashed](5,-9.3) .. controls (4.9,-9.3) and (4.7,-9.35) .. (4.7,-9.45);
    \node at (1.7,-9.75) {$\sigma_G(p) = -$};
    
    \begin{scope}[yscale=-1, shift={(-1,0)}]
    \draw  (0.5,3) ellipse (0.4 and 1.35);
    \draw  (6,3) node (v2) {} ellipse (1.5 and 1.5);
    \draw (0.5,4.35) .. controls (-1,4.35) and (-1.5,4.35) .. (-3,4.85);
    \draw (0.5,1.65) .. controls (-1,1.65) and (-1.75,1.6) .. (-3,1.15);
    \draw[red] (-2,4.55) .. controls (-1.5,4.5) and (-1.55,1.6) .. (-1.25,1.6);
    \draw[red] (-0.85,4.38) .. controls (-0.55,4.35) and (-0.6,1.65) .. (-0.2,1.65);
    \draw [red](-0.3,4.35) .. controls (-0.4,4.35) and (-0.2,1.65) .. (0.2,1.65);
    \draw [red](0.1,4.35) .. controls (-0.15,4.35) and (-0.1,1.65) .. (0.35,1.65);
    \draw [red](0.3,4.35) .. controls (-0.1,4.35) and (-0.05,1.65) .. (0.43,1.65);
    \draw [red, dashed](-2,4.55) .. controls (-2.5,4.5) and (-2.5,4) .. (-3,3.5);
    \draw [red, dashed](-0.85,4.35) .. controls (-1.2,4.35) and (-0.95,1.6) .. (-1.25,1.6);
    \node at (1.45,3) {$\partial_p$};
    
    \fill [gray!30]  (6,4.5) rectangle (6.2,1.5);
    \draw  (6,3) ellipse (1.5 and 1.5);
    \draw (6,4.5) -- (6,1.5) node (v1) {};
    \draw  plot[smooth, tension=0.7] coordinates {(v1)};
    \draw (6,1.5);
    \draw (6,1.5);
    \node at (6,1.1) {$\widetilde{\partial_p}$};
    \draw [red] (5.2,4.25) .. controls (5.65,3.65) and (6,2.5) .. (6,1.5);
    \node [red] at (-1.2,1.15) {$G$};
    \draw [dashed](v2) -- (4.5,3);
    \draw [dashed](6,3.4) .. controls (5.5,3.4) and (5.05,3.5) .. (4.65,3.75);
    \draw [dashed](6,2.6) .. controls (5.5,2.6) and (5.05,2.5) .. (4.65,2.25);
    \draw [dashed](6,3.75) .. controls (5.55,3.75) and (5.15,3.95) .. (5,4.1);
    \draw [dashed](6,2.25) .. controls (5.55,2.25) and (5.15,2.1) .. (5,1.9);
    \draw [dashed](6,4.05) .. controls (5.75,4.05) and (5.55,4.15) .. (5.4,4.4);
    \draw [dashed](6,1.95) .. controls (5.75,1.95) and (5.5,1.85) .. (5.35,1.65);
    \draw [dashed](6,4.3) .. controls (5.9,4.3) and (5.7,4.35) .. (5.7,4.45);
    \draw [dashed](6,1.7) .. controls (5.9,1.7) and (5.7,1.65) .. (5.7,1.55);
    \node [red] at (4.95,4.6) {$\widetilde{G}$};
    \end{scope}
    \node at (1.7,-4.8) {$\sigma_G(p) = +$};
    \end{tikzpicture}
    \caption{Left: Two spiralling patterns around $\partial_p$. Right: Lifts in the universal cover $\widetilde{F} = \bH^2$. Each region separated by dashed lines is a fundamental domain.\CORRECTED}
    \label{fig:spiralling leaves}
\end{figure}

\begin{lem}[{\cite[Lemma 2.3]{BKS}, \cite[Proposition 3.18]{BB}}]\label{l:determined by signature}
The support of the non-compact part $B$ of a measured geodesic lamination $(G,\mu)$ on $F$ is uniquely determined by its homotopy class and the signature $\sigma_G$. 
\end{lem}

\begin{rem}
The lemma above tells us that the non-compact part of any measured geodesic lamination can be determined only by the homotopy classes of the curves on the topological surface $\Sigma$ determined by $g$ and the signature $\sigma_G$. 
This is basically the formulation by Fock--Goncharov \cite{FG07} of \emph{rational unbounded laminations}. In fact, the compact part can be also described as a topological object on $\Sigma$ when it is ``rational''. The rational unbounded laminations on $\Sigma$ in the sense of \cite{FG07} form a dense subset of $\eML(F)$. Conversely, $\eML(F)$ can be understood as the metric completion of the set of rational unbounded laminations with respect to the Euclidean metric induced by any shear coordinate system. 
\end{rem}

We are going to define global coordinates on $\eML(F)$, called the \emph{shear coordinates}.
Fix an ideal triangulation $\tri$ of $\Sigma$, and represent each ideal arc $\alpha \in \tri$ by perpendicular geodesics $f(\alpha)$.
Given $(G,\mu) \in \eML(F)$, let $(G',\mu')$ be the measured geodesic lamination obtained from $(G,\mu)$ by shifting each leaf of $G'$ incident to a spike to a geodesic transverse to the boundary following the orientation induced by $\Sigma$. See \cref{fig:shift_lam}.
\begin{figure}
    \centering
    \begin{tikzpicture}[scale=.8]
    \draw (-1.5,2) .. controls (-1.5,0.55) and (0.5,0.55) .. (0.5,2);
    \draw [white, ultra thick](-0.6,0.9)  -- (-0.4,0.9);
    \draw (-3,-0.5) .. controls (-2,0) and (-1.5,0.55) .. (-1.5,2) .. controls (-1.5,0.55) and (-0.5,0.05) .. (-0.5,1.5) .. controls (-0.5,0.05) and (0.5,0.55) .. (0.5,2) .. controls (0.5,0.55) and (1,0) .. (2,-0.5);
    \draw [red](-0.5,1.5) .. controls (-0.5,0.05) and (0,-0.1) .. (0.25,-0.25);
    \draw [red, dashed](0.25,-0.25) .. controls (0.5,-0.4) and (0.75,-0.45) .. (1,-0.5);
    
    \draw [very thick,-{Classical TikZ Rightarrow[length=4pt]},decorate,decoration={snake,amplitude=1.5pt,pre length=2pt,post length=3pt}](2.5,0.75) -- (4,0.75);
    
    \draw (6,2) .. controls (6,0.55) and (8,0.55) .. (8,2);
    \draw [white, ultra thick](6.9,0.9)  -- (7.1,0.9);
    \draw (4.5,-0.5) .. controls (5.5,0) and (6,0.55) .. (6,2) .. controls (6,0.55) and (7,0.05) .. (7,1.5) .. controls (7,0.05) and (8,0.55) .. (8,2) .. controls (8,0.55) and (8.5,0) .. (9.5,-0.5);
    \draw [red](6.55,0.65) .. controls (7,0.05) and (7.5,-0.15) .. (7.75,-0.25);
    \draw [red, dashed](7.75,-0.25) .. controls (8,-0.35) and (8.25,-0.45) .. (8.5,-0.5);
    \end{tikzpicture}
    \caption{The way to shifting around spikes.}
    \label{fig:shift_lam}
\end{figure}
Then, the resulting ``shifted'' measured geodesic lamination $(G',\mu')$ transversely intersects with geodesics $f(\alpha)$ for $\alpha \in \tri$.

Fix $\alpha \in \tri$, and we focus on the quadrilateral $\square_\alpha$ of $\tri$ which contains $\alpha$ as a diagonal. 

Then, each path-connected component of $G' \cap \square_\alpha$ is either surrounding a corner of $\square_\alpha$ or not.
Then we take a connected compact subarc $\gamma_\alpha \subset \alpha$ which intersects only the components of $G' \cap \square_\alpha$ not surrounding corners of $\square_\alpha$.

Under the configuration of \cref{fig:lam_in_quad}, define
\begin{align*}
\sfx_\alpha^{\tri}(G,\mu) :=
\begin{cases}
    \mu'(\gamma_\alpha) & \text{if $\square_\alpha \cap G'$ has north-west to south-east components}, \\
    -\mu'(\gamma_\alpha)& \text{if $\square_\alpha \cap G'$ has north-east to south-west components}\\
    0 & \mbox{otherwise}.
\end{cases}
\end{align*}

\begin{figure}[h]
    \centering
    \begin{tikzpicture}[scale=.8]
    \draw [blue](-1.5,3) coordinate (v1) -- (-4,0.5) node (v2) {} -- (-1.5,-2) node (v3) {} -- (1,0.5) node (v4) {} -- (v1);
    \draw[blue] (v1) -- (v3);
    \draw[fill=white]  (v1) ellipse (0.5 and 0.5);
    \draw[fill=white]  (v2) ellipse (0.5 and 0.5);
    \draw[fill=white]  (v3) ellipse (0.5 and 0.5);
    \draw[fill=white]  (v4) ellipse (0.5 and 0.5);
    \draw [red](-2.5,3) .. controls (-2.5,2.5) and (-2,2) .. (-1.5,2) .. controls (-1,2) and (-0.5,2.5) .. (-0.5,3);
    \draw [red](-4,1.5) .. controls (-3.5,1.5) and (-3,1) .. (-3,0.5) .. controls (-3,0) and (-3.5,-0.5) .. (-4,-0.5);
    \draw [red](-2.5,-2) .. controls (-2.5,-1.5) and (-2,-1) .. (-1.5,-1) .. controls (-1,-1) and (-0.5,-1.5) .. (-0.5,-2);
    \draw [red](1,-0.5) .. controls (0.5,-0.5) and (0,0) .. (0,0.5) .. controls (0,1) and (0.5,1.5) .. (1,1.5);
    \draw [red](-2.05,3) .. controls (-2.05,2.7) and (-1.8,2.45) .. (-1.5,2.45) .. controls (-1.2,2.45) and (-0.95,2.7) .. (-0.95,3);
    \draw [red](-4,1.05) .. controls (-3.7,1.05) and (-3.45,0.8) .. (-3.45,0.5) .. controls (-3.45,0.2) and (-3.7,-0.05) .. (-4,-0.05);
    \draw [red](-2.05,-2) .. controls (-2.05,-1.7) and (-1.8,-1.45) .. (-1.5,-1.45) .. controls (-1.2,-1.45) and (-0.95,-1.7) .. (-0.95,-2);
    \draw [red](1,-0.05) .. controls (0.7,-0.05) and (0.45,0.2) .. (0.45,0.5) .. controls (0.45,0.8) and (0.7,1.05) .. (1,1.05);
    \draw [red](-2.15,3) .. controls (-2.15,2.65) and (-1.85,2.35) .. (-1.5,2.35) .. controls (-1.15,2.35) and (-0.85,2.65) .. (-0.85,3);
    \draw [red](-4,1.15) .. controls (-3.65,1.15) and (-3.35,0.85) .. (-3.35,0.5) .. controls (-3.35,0.15) and (-3.65,-0.15) .. (-4,-0.15);
    \draw [red](-2.15,-2) .. controls (-2.15,-1.65) and (-1.85,-1.35) .. (-1.5,-1.35) .. controls (-1.15,-1.35) and (-0.85,-1.65) .. (-0.85,-2);
    \draw [red](1,-0.15) .. controls (0.65,-0.15) and (0.35,0.15) .. (0.35,0.5) .. controls (0.35,0.85) and (0.65,1.15) .. (1,1.15);
    \draw [red](-2.3,3) .. controls (-2.3,2.55) and (-1.95,2.2) .. (-1.5,2.2) .. controls (-1.05,2.2) and (-0.7,2.55) .. (-0.7,3);
    \draw [red](-4,1.3) .. controls (-3.55,1.3) and (-3.2,0.95) .. (-3.2,0.5) .. controls (-3.2,0.05) and (-3.55,-0.3) .. (-4,-0.3);
    \draw [red](-2.3,-2) .. controls (-2.3,-1.55) and (-1.95,-1.2) .. (-1.5,-1.2) .. controls (-1.05,-1.2) and (-0.7,-1.55) .. (-0.7,-2);
    \draw [red](1,-0.3) .. controls (0.55,-0.3) and (0.2,0.05) .. (0.2,0.5) .. controls (0.2,0.95) and (0.55,1.3) .. (1,1.3);
    \draw [red](-4,2) .. controls (-3.45,2) and (0,-1.45) .. (0,-2);
    \draw [red](-3,3) .. controls (-3,2.45) and (0.45,-1) .. (1,-1);
    \draw [red](-4,2.5) .. controls (-3.55,2.5) and (0.5,-1.55) .. (0.5,-2);
    \draw [red](-3.5,3) .. controls (-3.5,2.55) and (0.55,-1.5) .. (1,-1.5);
    \draw [red](-4,1.75) .. controls (-3.4,1.75) and (-0.25,-1.4) .. (-0.25,-2);
    \draw [red](-2.75,3) .. controls (-2.75,2.4) and (0.4,-0.75) .. (1,-0.75);
    \draw [red](-4,2.25) .. controls (-3.5,2.25) and (0.25,-1.5) .. (0.25,-2);
    \draw [red](-3.25,3) .. controls (-3.2,2.5) and (0.5,-1.25) .. (1,-1.25);
    \draw [red](-4,3) -- (1,-2);
    \draw [mygreen, very thick](-1.5,1.7) node[fill, circle, inner sep=1.5pt] {} -- (-1.5,-0.7) node[fill, circle, inner sep=1.5pt] {};
    \node [mygreen] at (-1,1.5) {$\gamma_\alpha$};
    
    \begin{scope}[xshift=5cm, xscale=-1]
    \draw [blue](-1.5,3) coordinate (v1) -- (-4,0.5) node (v2) {} -- (-1.5,-2) node (v3) {} -- (1,0.5) node (v4) {} -- (v1);
    \draw[blue] (v1) -- (v3);
    \draw[fill=white]  (v1) ellipse (0.5 and 0.5);
    \draw[fill=white]  (v2) ellipse (0.5 and 0.5);
    \draw[fill=white]  (v3) ellipse (0.5 and 0.5);
    \draw[fill=white]  (v4) ellipse (0.5 and 0.5);
    \draw [red](-2.5,3) .. controls (-2.5,2.5) and (-2,2) .. (-1.5,2) .. controls (-1,2) and (-0.5,2.5) .. (-0.5,3);
    \draw [red](-4,1.5) .. controls (-3.5,1.5) and (-3,1) .. (-3,0.5) .. controls (-3,0) and (-3.5,-0.5) .. (-4,-0.5);
    \draw [red](-2.5,-2) .. controls (-2.5,-1.5) and (-2,-1) .. (-1.5,-1) .. controls (-1,-1) and (-0.5,-1.5) .. (-0.5,-2);
    \draw [red](1,-0.5) .. controls (0.5,-0.5) and (0,0) .. (0,0.5) .. controls (0,1) and (0.5,1.5) .. (1,1.5);
    \draw [red](-2.05,3) .. controls (-2.05,2.7) and (-1.8,2.45) .. (-1.5,2.45) .. controls (-1.2,2.45) and (-0.95,2.7) .. (-0.95,3);
    \draw [red](-4,1.05) .. controls (-3.7,1.05) and (-3.45,0.8) .. (-3.45,0.5) .. controls (-3.45,0.2) and (-3.7,-0.05) .. (-4,-0.05);
    \draw [red](-2.05,-2) .. controls (-2.05,-1.7) and (-1.8,-1.45) .. (-1.5,-1.45) .. controls (-1.2,-1.45) and (-0.95,-1.7) .. (-0.95,-2);
    \draw [red](1,-0.05) .. controls (0.7,-0.05) and (0.45,0.2) .. (0.45,0.5) .. controls (0.45,0.8) and (0.7,1.05) .. (1,1.05);
    \draw [red](-2.15,3) .. controls (-2.15,2.65) and (-1.85,2.35) .. (-1.5,2.35) .. controls (-1.15,2.35) and (-0.85,2.65) .. (-0.85,3);
    \draw [red](-4,1.15) .. controls (-3.65,1.15) and (-3.35,0.85) .. (-3.35,0.5) .. controls (-3.35,0.15) and (-3.65,-0.15) .. (-4,-0.15);
    \draw [red](-2.15,-2) .. controls (-2.15,-1.65) and (-1.85,-1.35) .. (-1.5,-1.35) .. controls (-1.15,-1.35) and (-0.85,-1.65) .. (-0.85,-2);
    \draw [red](1,-0.15) .. controls (0.65,-0.15) and (0.35,0.15) .. (0.35,0.5) .. controls (0.35,0.85) and (0.65,1.15) .. (1,1.15);
    \draw [red](-2.3,3) .. controls (-2.3,2.55) and (-1.95,2.2) .. (-1.5,2.2) .. controls (-1.05,2.2) and (-0.7,2.55) .. (-0.7,3);
    \draw [red](-4,1.3) .. controls (-3.55,1.3) and (-3.2,0.95) .. (-3.2,0.5) .. controls (-3.2,0.05) and (-3.55,-0.3) .. (-4,-0.3);
    \draw [red](-2.3,-2) .. controls (-2.3,-1.55) and (-1.95,-1.2) .. (-1.5,-1.2) .. controls (-1.05,-1.2) and (-0.7,-1.55) .. (-0.7,-2);
    \draw [red](1,-0.3) .. controls (0.55,-0.3) and (0.2,0.05) .. (0.2,0.5) .. controls (0.2,0.95) and (0.55,1.3) .. (1,1.3);
    \draw [red](-4,2) .. controls (-3.45,2) and (0,-1.45) .. (0,-2);
    \draw [red](-3,3) .. controls (-3,2.45) and (0.45,-1) .. (1,-1);
    \draw [red](-4,2.5) .. controls (-3.55,2.5) and (0.5,-1.55) .. (0.5,-2);
    \draw [red](-3.5,3) .. controls (-3.5,2.55) and (0.55,-1.5) .. (1,-1.5);
    \draw [red](-4,1.75) .. controls (-3.4,1.75) and (-0.25,-1.4) .. (-0.25,-2);
    \draw [red](-2.75,3) .. controls (-2.75,2.4) and (0.4,-0.75) .. (1,-0.75);
    \draw [red](-4,2.25) .. controls (-3.5,2.25) and (0.25,-1.5) .. (0.25,-2);
    \draw [red](-3.25,3) .. controls (-3.2,2.5) and (0.5,-1.25) .. (1,-1.25);
    \draw [red](-4,3) -- (1,-2);
    \draw [mygreen, very thick](-1.5,1.7) node[fill, circle, inner sep=1.5pt] {} -- (-1.5,-0.7) node[fill, circle, inner sep=1.5pt] {};
    \node [mygreen] at (-1.1,1.5) {$\gamma_\alpha$};
    \end{scope}
    \end{tikzpicture}
    \caption{Left: $\sfx_\alpha^\tri > 0$. Right: $\sfx_\alpha^\tri < 0$.}
    \label{fig:lam_in_quad}
\end{figure}

\begin{thm}[\cite{FG07,BB}]\label{thm:lam_shear}
For any ideal triangulation $\tri$ of $\Sigma$, the map
\begin{align*}
\boldsymbol{\sfx}^{\tri}:\eML(F)\to \bR^{\tri},\quad
(G,\mu) \mapsto (\sfx_\alpha^{\tri}(G,\mu))_{\alpha \in \tri}
\end{align*}
is bijective.
Moreover, the coordinate transformation $\boldsymbol{\sfx}^{\tri'} \circ (\boldsymbol{\sfx}^{\tri})^{-1}$ associated with a flip is given as in \cref{f:tropical x-flip}.
More accurately, it coincides with the max-plus tropicalization of the cluster $\cX$-transformation.
\end{thm}

\begin{figure}[h]
\[\hspace{1.4cm}
\begin{tikzpicture}[scale=0.7]
\path(0,0) node [fill, circle, inner sep=1.6pt] (x1){};
\path(135:4) node [fill, circle, inner sep=1.6pt] (x2){};
\path(0,4*1.4142) node [fill, circle, inner sep=1.6pt] (x3){};
\path(45:4) node [fill, circle, inner sep=1.6pt] (x4){};
\draw[blue](x1) to node[midway,left,black]{$\sfx_4$} (x2) 
to node[midway,left,black]{$\sfx_1$} (x3) 
to node[midway,right,black]{$\sfx_2$} (x4) 
to node[midway,right,black]{$\sfx_3$} (x1) 
to node[midway,left,black]{$\sfx_\kappa$} (x3);

\draw[-implies, double distance=2pt](4,2*1.4142) to node[midway,above]{$f_{\kappa}$} (6,2*1.4142);

\begin{scope}[xshift=10cm]
\path(0,0) node [fill, circle, inner sep=1.6pt] (x1){};
\path(135:4) node [fill, circle, inner sep=1.6pt] (x2){};
\path(0,4*1.4142) node [fill, circle, inner sep=1.6pt] (x3){};
\path(45:4) node [fill, circle, inner sep=1.6pt] (x4){};
\draw[blue](x1) to node[midway,left,black]{\scalebox{0.8}{$\sfx_4-\max\{0,-\sfx_\kappa\}$}} (x2) 
to node[midway,left,black]{\scalebox{0.8}{$\sfx_1+\max\{0,\sfx_\kappa\}$}} (x3) 
to node[midway,right,black]{\scalebox{0.8}{$\sfx_2-\max\{0,-\sfx_\kappa\}$}} (x4) 
to node[midway,right,black]{\scalebox{0.8}{$\sfx_3+\max\{0,\sfx_\kappa\}$}} (x1);
\draw[blue] (x2) to node[midway,above,black]{$-\sfx_\kappa$} (x4);
\end{scope}
\node [blue] at (0.25,2.1) {$\kappa$};
\node [blue] at (10.5,2.5) {$\kappa'$};
\end{tikzpicture}
\]
\caption{The coordinate transformation for a flip.
}
\label{f:tropical x-flip}
\end{figure}

Therefore, $\eML(F)$ is a PL manifold with the PL atlas consisting of the maps $\boldsymbol{\sfx}^\tri$ as global coordinates.

\begin{rem}\label{rem:self-folded}
When $\tri$ contains self-folded triangles, we need to slightly modify the definition of shear coordinates to make them match with the cluster transformations (cf. \cite[Definition 9.2]{AB}). For a self-folded edge $\alpha$ of $\tri$, we can still find a quadrilateral $\square_\alpha$ having $\alpha$  as a diagonal in a suitable covering of $F$ (see \cref{fig:c_p}). Then we define $\widetilde{\sfx}_\alpha^\tri$ by lifting the measured geodesic lamination and applying the same rule as above, which is clearly invariant under covering transformations. Then we define 
\begin{align*}
    \sfx_\gamma^\tri :=\begin{cases}
         \widetilde{\sfx}_\alpha^\tri + \sfx_\beta^\tri & \text{if $\gamma=\alpha$ is self-folded with encircling arc $\beta$}, \\
         \widetilde{\sfx}_\gamma^{\tri} & \text{otherwise}.
    \end{cases}
\end{align*}
With this modification, the shear coordinates $\{\sfx^\tri\}_\tri$ are shown to be related with each other by tropical cluster $\X$-transformations by the same argument as in \cite[Proposition 9.8]{AB}. 
\end{rem}

For each $p \in P$, we take a loop $c_p$ in $F$ parallel to $\partial_p$ and define
\begin{align*}
    \theta^\mathrm{ML}: \eML(F) \to \bR^P, \quad
    (G, \mu) \mapsto (\sigma_G(p) \cdot \mu(c_p))_{p \in P}\ .
\end{align*}

It is clear that $\ML_0(F) \cong (\theta^\mathrm{ML})^{-1}(0)$ from the definition of $\ML_0(F)$.
Namely, we have the exact sequence
\begin{align}\label{eq:lam_exact_seq}
    0 \to \ML_0(F) \to \eML(F) \xrightarrow{\theta^\mathrm{ML}} \bR^P \to 0
\end{align}
in the PL category.


\begin{lem}\label{lem:ML_Casimir}
For any ideal triangulation $\tri$ without self-folded triangles and $p \in P$, we have
\begin{align*}
    (\boldsymbol{\sfx}^\tri)^* \theta^\mathrm{ML}_p = \sum_{\substack{\alpha \in \tri\\ \text{incident to $p$}}} \sfx^\tri_\alpha,
\end{align*}
where $\alpha \in \tri$ is counted twice if both of its endpoints are incident to $p$. If $\tri$ contains a self-folded triangle around $p$, then $(\boldsymbol{\sfx}^\tri)^* \theta^\mathrm{ML}_p = \widetilde{\sfx}_\alpha^\tri=\sfx_\alpha^\tri - \sfx_\beta^\tri$ with the self-folded edge $\alpha$ incident to $p$ (recall \cref{rem:self-folded}).
\end{lem}

\begin{proof}
Let $\theta_p^\tri$ denote the right-hand side of the statement. 
Fix $p \in P$, and take an ideal triangulation $\tri_p$ of $\Sigma$ such that there is only one ideal arc $\alpha$ incident to $p$. Then $\alpha$ is a self-folded edge. 
Let $(G, \mu) \in \eML(F)$ and let $\gamma_\alpha$ be the subarc of $\alpha$ whose measure determines the shear coordinate of $(G,\mu)$ along $\alpha$.
Then, there is an arc $\delta_\alpha$ which connects the end points of $\gamma_\alpha$, its measure with respect to $\mu$ is zero and the union $\gamma_\alpha \cup \delta_\alpha$ is a loop parallel to $\partial_p$. See \cref{fig:c_p}.
It is easy to see that the sign of the shear coordinate $\widetilde{\sfx}^{\tri_p}_\alpha$ and the lamination signature $\sigma_G(p)$ are the same. Thus we get $(\boldsymbol{\sfx}^{\tri_p})^* \theta^\mathrm{ML}_p= \widetilde{\sfx}_\alpha^\tri = \theta_p^{\tri_p}$. 

The statement for a general ideal triangulation $\tri$ then follows from the fact that the expression in the right-hand side is invariant under the mutations. 
Indeed, $\theta_p^\tri$ is the Casimir function associated with $p$, and hence the tropical cluster $\X$-transformation from $\tri$ to $\tri'$ sends $\theta_p^\tri$ to $\theta_p^{\tri'}$. Thus the statement holds for general $\tri$. 
\end{proof}

\begin{figure}[ht]
    \centering
    \begin{tikzpicture}
\draw[blue] (0,0.5) node [draw, black, fill=white, circle, inner sep=2pt] {$\partial_p$} -- (0,-1.5) node (v) [draw, black, fill=white, circle, inner sep=5pt] {};
\draw[blue] (v) .. controls (-1.5,-0.5) and (-1.5,2) .. (0,2) .. controls (1.5,2) and (1.5,-0.5) .. (v);
\draw [red] (0.5,-1.55) .. controls (0.25,-1.05) and (-0.25,-1.05) .. (-0.5,-1.55);
\draw [red](0.75,-1.45) .. controls (0.4,-0.9) and (-0.4,-0.9) .. (-0.75,-1.45);
\draw [red](1,-1.3) .. controls (0.5,-0.7) and (-0.5,-0.7) .. (-1,-1.3);
\draw [red](-1.25,-1.1) .. controls (-0.5,-0.55) and (0.5,-0.75) .. (1.5,-0.3);
\draw [red](1,2.55) .. controls (1.2,1.35) and (1.1,-0.45) .. (0,-0.45) .. controls (-0.95,-0.45) and (-0.95,1.2) .. (0,1.2) .. controls (0.95,1.2) and (0.65,-0.55) .. (-0.35,0.1);
\draw [red](0.5,2.55) .. controls (0.8,1.35) and (1.05,-0.25) .. (0,-0.25) .. controls (-0.7,-0.25) and (-0.75,1) .. (0,1) .. controls (0.7,1) and (0.65,-0.25) .. (-0.25,0.1);
\draw [cyan!60, very thick](0,-0.55) .. controls (-1.1,-0.55) and (-1.1,1.35) .. (0,1.35) .. controls (0.95,1.35) and (0.85,-0.15) .. (0,-0.15);
\draw [mygreen, very thick](0,-0.55) node [draw, fill, circle, inner sep=1.3pt] {} -- (0,-0.15) node [draw, fill, circle, inner sep=1.3pt] {};
\node [cyan!60] at (0,1.6) {$\delta_\alpha$};
\draw [dashed, ->](1.7,0.3) .. controls (1.3,0.1) and (0.45,-0.15) .. (0.05,-0.35);
\node [mygreen] at (1.95,0.35) {$\gamma_\alpha$};
\node [blue] at (-1.1,1.6) {$\beta$};
\end{tikzpicture}
\qquad
\begin{tikzpicture}
\draw [blue](0,1.5) node [draw, fill, black, circle, inner sep=1.3pt] (v1) {} -- (0,-1.5) node [draw, fill, black, circle, inner sep=1.3pt] (v2) {} -- (-1.5,0) node [draw, fill, black, circle, inner sep=1.3pt] {} -- (v1) -- (1.5,0) node [draw, fill, black, circle, inner sep=1.3pt] {} -- (v2);
\node[blue] at (-1,1) {$\alpha$};
\node[blue] at (1,1) {$\alpha$};
\node[blue] at (0.15,0.25) {$\alpha$};
\node[blue] at (-1,-1) {$\beta$};
\node[blue] at (1,-1) {$\beta$};
\draw [red](-0.95,0.55) -- (0.85,-0.65);
\draw [red](-0.8,0.7) -- (1,-0.5);
\draw [red](-0.65,0.85) .. controls (-0.2,0.4) and (0.2,0.4) .. (1,0.5);
\draw [red](-0.5,1) .. controls (-0.15,0.65) and (0.15,0.6) .. (0.85,0.65);
\draw [red](-0.4,1.1) .. controls (-0.15,0.85) and (0.15,0.8) .. (0.65,0.85);
\node [red, rotate=10] at (0.1,1.25) {$\vdots$};
\draw [red](0.7,-0.8) .. controls (0,-0.3) and (-0.5,0) .. (-1,-0.5);
\draw [red](-0.85,-0.65) .. controls (-0.35,-0.3) and (0.25,-0.65) .. (0.55,-0.95);
\draw [red](-0.65,-0.85) .. controls (-0.2,-0.65) and (0.15,-0.85) .. (0.4,-1.1);
\draw [red](-0.45,-1.05) .. controls (-0.1,-0.95) and (0.1,-1.05) .. (0.25,-1.25);
\end{tikzpicture}
    \caption{Left: green line shows $\gamma_\alpha$ and cyan line shows $\delta_\alpha$; right: the quadrilateral whose diagonal is $\alpha$.}
    \label{fig:c_p}
\end{figure}

We are going to characterize several properties of measured geodesic laminations in terms of the shear coordinates.
A measured geodesic lamination $(G,\mu)$ is said to be \emph{filling} if its complementary regions are (once-punctured) polygons. 
A filling measured geodesic lamination is said to be \emph{generic} if its unpunctured complementary regions are trigon.
Then we have the following:

\begin{prop}\label{p:shear_non-vanishing}
Let $\Sigma$ be a punctured surface, and $(F,f)$ a marked hyperbolic surface of type $\Sigma$. 
For a measured geodesic lamination $(G, \mu) \in \eML(F)$ on $F$, we have the following.  
\begin{enumerate}
    \item If $(G, \mu)$ is generic filling, then it satisfies ${\sfx}_\alpha^\tri(G, \mu) \neq 0$ for any ideal triangulation $\tri$ and any arc $\alpha \in \tri$.
    \item If $(G,\mu)$ is filling but non-generic, then it satisfies ${\sfx}_\alpha^\tri(G,\mu) = 0$ for some ideal triangulation $\tri$ and an arc $\alpha \in \tri$.
    \item If the support $G$ consists of closed geodesics, then it satisfies $\sfx_\alpha^\tri(G, \mu) = 0$ for some ideal triangulation $\tri$ and an arc $\alpha \in \tri$.
\end{enumerate}
\end{prop}

\begin{proof} 
For a filling geodesic lamination $G$, let $\cF_G$ denote the singular foliation on $\Sigma$ obtained by extending $G$ so that each complementary $k$-gon  is filled by a foliation with one $k$-pronged singularity.

Recall the Euler--Poincar\'e formula  (see, for instance, \cite[Proposition 5.1]{FLP}) for a singular foliation $\cF$:
\begin{align*}
    2\chi(\Sigma)=\sum_{x \in \mathrm{Sing}\cF} (2-v(x)).
\end{align*}
Here, $\mathrm{Sing}\,\cF$ denotes the set of singular points of $\cF$, and $v(x)$ is the number of the separatorices at $x$.

(1): If $(G,\mu)$ is generic filling, then $v(x)=3$ for all $x \in\mathrm{Sing}\, \cF_G$. 
In particular 
\begin{align*}
    \# \mathrm{Sing}\,\cF_G =-2\chi(\Sigma),
\end{align*}
which is exactly the number of triangles in any ideal triangulation $\tri$ of $\Sigma$. Therefore we can take the representatives of the arcs of $f(\tri)$ so that each triangle contains exactly one singular point.
For each quadrilateral $\square_\alpha$, the genericity implies that there is no singular leaf (a \emph{saddle connection}) connecting the two singular points on the triangles in $\square_\alpha$. 
Thus, $G \cap \square_\alpha$ 
takes either configuration shown in the left or right of \cref{fig:lam_in_quad}, which implies $\sfx_\alpha^\tri(G,\mu)\neq 0$.

(2): Suppose that $(G,\mu)$ is non-generic.
We can resolve each singular point $x$ of $\cF_G$ with $v(x)>3$ to a collection of $v(x)-2$ generic singular points connected by saddle connections by applying a sequence of the inverse of Whitehead collapsing in an arbitrary direction. Let $\cF_G'$ be the resulting singular foliation. See \cref{fig:fol_sing_resol}.
Then we can take representatives of the arcs of $f(\tri)$
so that each of its triangle contains exactly one singular point of $\cF'_G$.
Then we have $\sfx_\alpha^\tri(G,\mu)=0$
exactly when the edge $f(\alpha)$ is dual to one of the saddle connections in $\cF'_G$, since there are only corner arcs in $G \cap \square_\alpha$ in such a situation.

\begin{figure}[h]
    \centering
    \begin{tikzpicture}
    \draw [very thick](-3.5,2) -- (-2.5,0.5) -- (-4,-0.5);
    \draw [very thick](-1.5,2) -- (-2.5,0.5) coordinate (v1) {} -- (-1,0);
    \draw [very thick](v1) -- (-2,-1);
    \draw [very thick](2.5,2) -- (3.05,0.5) coordinate (v2) {} -- (2,-0.5);
    \draw (-3.15,2) .. controls (-2.5,1) and (-2.5,1) .. (-1.85,2);
    \draw (-3.6,1.5) .. controls (-3,0.5) and (-3,0.5) .. (-3.9,0);
    \draw (-3.6,-0.6) .. controls (-2.6,0.1) and (-2.6,0.1) .. (-2.3,-1);
    \draw (-1.8,-0.85) .. controls (-2.15,0.15) and (-2.15,0.15) .. (-1.15,-0.2);
    \draw (-1.05,0.35) .. controls (-2,0.6) and (-2,0.6) .. (-1.4,1.6);
    \draw (-2.85,2) .. controls (-2.55,1.7) and (-2.4,1.7) .. (-2.15,2);
    \draw (-3.7,1) .. controls (-3.5,0.75) and (-3.55,0.55) .. (-3.85,0.45);
    \draw (-3.25,-0.7) .. controls (-2.95,-0.5) and (-2.75,-0.55) .. (-2.65,-0.9);
    \draw (-1.6,-0.7) .. controls (-1.7,-0.35) and (-1.6,-0.25) .. (-1.3,-0.35);
    \draw (-1.25,1.25) .. controls (-1.45,0.95) and (-1.4,0.8) .. (-1.1,0.7);
    
    \draw [squigarrow, very thick] (0,0.5) -- (1.5,0.5);
    
    \draw [very thick](5,-1) -- (4.5,0.15) coordinate (v3) {} -- (6,0);
    \draw [very thick](v2) -- (3.85,0.5) coordinate (v4) {} -- (v3);
    \draw [very thick](v4) -- (4.5,2);
    \draw (2.85,2) .. controls (3.2,0.65) and (3.75,0.65) .. (4.15,2);
    \draw (2.4,1.5) .. controls (2.75,0.55) and (2.75,0.55) .. (2.05,0);
    \draw (2.5,-0.6) .. controls (3.25,0.45) and (4.25,0.25) .. (4.55,-0.95);
    \draw (5.25,-0.9) .. controls (4.85,-0.1) and (4.85,-0.1) .. (5.85,-0.25);
    \draw (5.65,0.4) .. controls (4.4,0.45) and (4.15,0.5) .. (4.8,1.7);
    \draw (3.15,2) .. controls (3.4,1.5) and (3.6,1.5) .. (3.85,2);
    \draw (2.3,1) .. controls (2.4,0.75) and (2.35,0.55) .. (2.1,0.4);
    \draw (3.1,-0.7) .. controls (3.45,-0.3) and (3.9,-0.45) .. (4.05,-0.9);
    \draw (5.4,-0.8) .. controls (5.35,-0.5) and (5.45,-0.4) .. (5.7,-0.45);
    \draw (5,1.4) .. controls (4.7,0.9) and (4.85,0.7) .. (5.4,0.75);
    \node at (-4.5,1.5) {$\cF_G$};
    \node at (5.7,1.5) {$\cF'_G$};
    \end{tikzpicture}
    \caption{A resolution of a 5-pronged singularity into 3-pronged singularities.}
    \label{fig:fol_sing_resol}
\end{figure}

(3):
Let $(G, \mu) \in \ML_0(F)$ be a compact rational measured geodesic lamination 
and $\tri$ be an ideal triangulation of $\Sigma$.
We rescale $(G, \mu)$ by the lcm $u$ of the denominators of irreducible expressions of the shear coordinates $\sfx^\tri_\alpha(G, \mu)$, $\alpha \in \tri$.
Then, all the shear coordinates of the rescaled measured geodesic lamination $(G, u \mu)$ are integral.
We take $|\sfx_\alpha(G, u \mu)|$ points on $\alpha$ for each $\alpha \in \tri$.
If necessary adding the points on ideal arcs of $\tri$, we connect them like in the configuration of \cref{fig:lam_in_quad} according to the shear coordinate $\boldsymbol{\sfx}^\tri(G, u \mu)$ and we get multicurve $C$ of loops on $\Sigma$.
Then, the geodesic lamination $G_C$ isotopic to $f(C)$ with the transverse measure given by the geometric intersection number is nothing but $(G, u \mu)$ since these have the same shear coordinate.

Let us take a pants decomposition $\cR$ of $\Sigma$ which contains $C$. 
Then there is a component $P$ of $\Sigma \setminus \bigcup \cR$ which contains some punctures.
In particular, one can find an ideal arc $\alpha \subset P \subset \Sigma$ both of whose endpoints are the punctures in $P$. See \cref{fig:arc_in_pants}. Let $\tri$ be an ideal triangulation containing $\alpha$. 
Then we have $\sfx_\alpha(G, \mu) = 0$
since $C$ and $\alpha$ are disjoint.
\begin{figure}[h]
    \centering
    \begin{tikzpicture}
    \draw (4,-3) arc [start angle=-90, end angle=90, x radius=.4cm, y radius=1cm];
    \draw [dashed] (4,-1) coordinate (v1) {} arc [start angle=90, end angle=270, x radius=.4cm, y radius=1cm];
    \draw  (0,-2) ellipse (0.4 and 1);
    \draw (0,-3) .. controls (1.15,-3) and (2.85,-3) .. (4,-3);
    \draw (0,-1) -- (v1);
    \draw [blue] (2,-3) arc [start angle=-90, end angle=90, x radius=.4cm, y radius=1cm];
    \draw [dashed, blue] (2,-1) coordinate (v1) {} arc [start angle=90, end angle=270, x radius=.4cm, y radius=1cm];
    \node[fill, circle, inner sep=1.3] at (2.35,-1.5) {};
    \node [blue] at (2.7,-2) {$\alpha$};
    
    
    \node at (6.25,-2) {$\oot\ \, P\, \ \too$};

    \draw  (10,-2) ellipse (2 and 1.5);
    \node[fill, circle, inner sep=1.3] (v1) at (9,-2) {};
    \node[fill, circle, inner sep=1.3] (v2) at (11,-2) {};
    \draw[blue]  (v1) edge (v2);
    \node[blue] at (10,-1.73) {$\alpha$};
    \end{tikzpicture}
    \caption{Examples of the ideal arc $\alpha$.}
    \label{fig:arc_in_pants}
\end{figure}
\end{proof}

\begin{dfn}\label{def:X-filling}
We say that a measured geodesic lamination $(G,\mu) \in \eML(F)$ on a hyperbolic surface $F$ with geodesic boundary is \emph{$\X$-filling} if it satisfies $\sfx_\alpha^\tri(G,\mu) \neq 0$ for any ideal triangulation $\tri$ and any arc $\alpha \in \tri$.
\end{dfn}
By \cref{p:shear_non-vanishing} (1), a generic filling measured geodesic lamination is $\X$-filling.

\subsection{Nielsen--Thurston classification of mapping classes}\label{subsec:NT_classif}

\begin{dfn}
    Let $MC(\Sigma)$ be the \emph{mapping class group} of $\Sigma$, which consists of the isotopy classes of diffeomorphisms of $\Sigma$ which preserve the set $M$ of marked points. 
    If $\Sigma$ is a once-punctured torus, then we redefine $MC(\Sigma)$ as the quotient of that group by the hyperelliptic involution. 
    An element of $MC(\Sigma)$ is called a \emph{mapping class}. 
\end{dfn}

\begin{ex}
\begin{enumerate}
    \item For a simple closed curve $C \subset \Sigma$, we have the associated \emph{(right-hand) Dehn twist} $T_C \in MC(\Sigma)$ which takes another curve that intersects with $C$ once along $C$ in the right direction. See \cref{fig:example_Dehn} for an example.
    \item For a boundary component $\partial \subset \partial \Sigma$ with special points $m_1,\dots,m_k$ in this counter-clockwise order, there is a mapping class that rotates $\partial$ and takes $m_i$ to $m_{i+1}$ mod $k$. Its $k$-th power is the same as the Dehn twist along the curve $C$ isotopic to $\partial$. 
    \item For a simple arc $\alpha$ connecting two punctures $p_1,p_2$, there is a mapping class called the \emph{half-twist} that preserves $\alpha$ but interchanging $p_1$ and $p_2$ by the counter-clockwise rotation. Similarly, a braiding of two boundary components with the same number of marked points can be considered. 
\end{enumerate}
\end{ex}

Given a marked hyperbolic surface $(F,f)$ of type $\Sigma$, similarly we define the mapping class group $MC(F,\partial F)$, where the diffeomorphisms and their isotopies are required to preserve $\partial F$. 

The mapping class group $MC(F,\partial F)$ naturally acts on $\eML(F)$ by 
\begin{align*}
    \phi.(G,\mu):=(\phi_*(G), \phi_*(\mu))
\end{align*}
for $(G, \mu) \in \eML(F)$ and $\phi \in MC(F,\partial F)$. Here for $\phi=[\varphi] \in MC(F,\partial F)$, $\phi_*(G)$ denotes the geodesic lamination isotopic to $\varphi(G)$ with isotopy $h$, $\phi_\ast(\mu)$ is the push-forward of the measure $\mu$ by $h \circ \varphi$. 
The marking homeomophism induces an isomorphism $MC(\Sigma) \xrightarrow{\sim} MC(F,\partial F)$, $\phi \mapsto \phi^f:=f \phi f^{-1}$, through which $MC(\Sigma)$ acts on $\eML(F)$.  

\begin{rem}
For two marked hyperbolic surfaces $(F,f), (F',f')$ of type $\Sigma$, the spaces $\eML(F)$ and $\eML(F')$ are canonically isomorphic to each other so that the shear coordinates and the mapping class group action are compatible. Indeed, one can consider the bundle over the enhanced \Teich\ space whose fibers are the spaces $\eML(F)$, which admits an $MC(\Sigma)$-equivariant trivialization. See \cite[Appendix A]{AIK24}.
\end{rem}

\begin{prop}\label{prop:lamination_split_action}
Let
$\sigma_\Sigma: MC(\Sigma) \to \mathfrak{S}_P$
be the permutation action of the mapping class group on the punctures.
Then, for a marked hyperbolic surface $(F,f)$ of type $\Sigma$ we have the following commutative diagram
\begin{equation*}
    \begin{tikzcd}
    \eML(F) \ar[r,"\phi"] \ar[d,"\theta^\mathrm{ML}"'] & \eML(F) \ar[d,"\theta^\mathrm{ML}"] \\
    \bR^P \ar[r,"\sigma_\Sigma(\phi)"']& \bR^P
    \end{tikzcd}
\end{equation*}
for all $\phi \in MC(\Sigma)$.
\end{prop}

\paragraph{\textbf{Nielsen--Thurston classification}}
Let $\Sigma$ be a punctured surface.
Note that the group $\bR_{>0}$ acts on $\ML_0(F)$ by multiplying a common factor to the measures of transverse arcs. 
A mapping class $\phi \in MC(\Sigma)$ is said to be
\begin{itemize}
    \item \emph{periodic} if it is finite order;
    \item \emph{reducible} if it fixes a \emph{multicurve}, which is the isotopy class of a collection of mutually disjoint simple closed curves in $\Sigma$ neither peripheral nor contractible;
    \item \emph{pseudo-Anosov} (\lq\lq \emph{pA}" for short) if there exists a real number $\lambda >1$ and two filling compact measured geodesic laminations $(G_\pm,\mu_\pm)$ on $\Sigma$ which are
    transverse to each other
    and satisfying $\phi^{\pm 1}(G_\pm,\mu_\pm) = \lambda^{\pm 1}(G_\pm, \mu_\pm)$.
\end{itemize}
We call the pair $((G_+,\mu_+), (G_-,\mu_-))$ the \emph{pA pair} of $\phi$.
Also, a pA mapping class is said to be \emph{generic} if one of (or equivalently, both of) its pA pair is generic\footnote{This is equivalent to the condition that the meromorphic quadratic differential which gives the pA pair has only simple zeros. See, for instance, \cite{BS15}.}.
The numerical invariant $\lambda>1$ associated with a pA mapping class $\phi$ is called the \emph{(pseudo-Anosov) stretch factor}.

\begin{thm}[Nielsen--Thurston classification, e.g. \cite{FLP}]\label{t:NTclassification}
Each mapping class on a punctured surface is either periodic, reducible or pA. Moreover, a pA mapping class is neither periodic nor reducible. 
\end{thm}

Indeed, this classification theorem is based on Brouwer's fixed point theorem, applied to the action of a mapping class on the Thurston compactification of the \Teich\ space of $\Sigma$. The boundary is given by the space $\bP\ML_0(F):=(\ML_0(F) \setminus \{0\})/\bR_{>0}$ of projective measured geodesic laminations.

\begin{ex}
When $\Sigma$ is a once-punctured torus, the mapping class group is isomorphic to $PSL_2(\bZ)=SL_2(\bZ)/\{\pm 1\}$, where the quotient by $\{\pm 1\}$ corresponds to that by the hyperelliptic involution. In this case, the Nielsen--Thurston classification boils down to the well-known trichotomy elliptic/parabolic/hyperbolic for M\"obius transformations which act on the compactified upper half plane $\overline{\mathbb{H}^2}= \mathbb{H}^2 \cup S^1_\infty$. The space $\bP\ML_0(F)$ is identified with $S^1_\infty=\bR \cup \{\infty\}$, where filling laminations correspond to irrational numbers. 
\end{ex}
\section{Definition of sign stability}\label{sec:sign_stability}

\subsection{Combinatorial description of the mapping class group action}
For any marked surface $\Sigma$, let us use the index set
\begin{align*}
    I &= I(\Sigma) :=
    \{1, \dots, 3(2g-2+h+b)+\spe\}.
\end{align*} 
Recall that $|I|$ is the number of (internal) edges in any ideal triangulation of $\Sigma$.

\subsubsection{Combinatorics of ideal triangulations}

For an ideal triangulation $\tri$ {without self-folded triangles}, we define a skew-symmetric matrix $B^{\tri} = (b_{\alpha \beta}^\tri)_{\alpha, \beta \in \tri}$ by
\begin{align*}
    b_{\alpha \beta}^\tri = 
    \begin{cases}
        1 & \text{if $\alpha$ and $\beta$ form an angle with $\alpha$ following $\beta$ in the clockwise order,}\\
        -1  & \text{if the same holds with the counter-clockwise order,}\\
        0 & \text{otherwise.}
    \end{cases}
\end{align*}

The quiver representing the matrix $B^{\tri}$ locally looks like the following:
\begin{align*}
\tikz[scale=1.15,xscale=-1,>=latex]{
\path(0,0) node [fill, circle, inner sep=1.5pt] (x1){};
\path(135:2) node [fill, circle, inner sep=1.5pt] (x2){};
\path(0,2*1.4142) node [fill, circle, inner sep=1.5pt] (x3){};
\path(45:2) node [fill, circle, inner sep=1.5pt] (x4){};
\draw[blue](x1) to (x2) to (x3) to (x4) to (x1) to  (x3);
\color{mygreen}{
    \draw(0,1.4142) circle(2pt) coordinate(v0);
    \draw(135:2)++(45:1) circle(2pt) coordinate(v1);
    \draw(45:2)++(135:1) circle(2pt) coordinate(v2);
    \draw(45:1) circle(2pt) coordinate(v3);
    \draw(135:1) circle(2pt) coordinate(v4);
    \qarrow{v0}{v4}
    \qarrow{v4}{v1}
    \qarrow{v1}{v0}
    \qarrow{v0}{v2}
    \qarrow{v2}{v3}
    \qarrow{v3}{v0}
    }
}
\end{align*}

Let $\Tri(\Sigma)$ denote the graph whose vertices are ideal triangulations {without self-folded triangles} of $\Sigma$ and adjacent ideal triangulations are related by a flip. This is the graph version of the \emph{Ptolemy groupoid} \cite{Penner}, and is known to be connected \cite{FST}. 

A \emph{labeled triangulation} consists of an ideal triangulation $\tri$ and a bijection
$\ell: I \to \tri$, called a \emph{labeling}.
It allows us to promote the exchange matrix to be a matrix
$B^{(\tri, \ell)} = (b^{(\tri, \ell)}_{ij})_{i,j \in I}$ indexed by the set $I$, where $b^{(\tri, \ell)}_{ij} := b^\tri_{\ell(i), \ell(j)}$. Similarly, we get a coordinate system 
$\bsfx^{(\tri, \ell)}: \eML(F) \to \bR^{I}$ indexed by $I$ by setting $\sfx^{(\tri,\ell)}_i:= \sfx^\tri_{\ell(i)}$. This setting is useful to keep track of the mapping class group action correctly. 

Flips of ideal triangulations can be upgraded to \emph{labeled flips} 
\begin{align}\label{eq:labeled flip}
\mu_k: (\tri,\ell) \to (\tri',\ell')
\end{align}
for $k \in I$, where $\tri'$ is the ideal triangulation obtained by the flip along the arc $\ell(k)$, and $\ell'$ is defined by $\ell'(i):=\ell(i)$ for $i \neq k$ and by setting $\ell'(k)$ to be the new arc arising from the flip.

Let $\bTri(\Sigma)$ denote the graph whose vertices are labeled triangulations of $\Sigma$ and adjacent labeled triangulations are related by a labeled flip or the action of a transposition of labelings in $\fS_I$.
Here the action of $\fS_{I}$ on a labeling $\ell: I \to \tri$ is defined as $\sigma.\ell(i) := \ell(\sigma^{-1}(i))$ for $\sigma \in \fS_{I}$ and $i \in I$.
We will refer to an edge corresponding to a labeled flip (resp. the action of a transposition on a labeling) as a \emph{horizontal} (resp. \emph{vertical}) edge.

\subsubsection{Representation paths of mapping classes}
The mapping class group $MC(\Sigma)$ acts on the graph $\bTri(\Sigma)$ as $(\tri, \ell) \mapsto \phi(\tri, \ell) := (\phi(\tri), \phi \circ \ell)$.
We remark that this action preserves the edge labelings of $\bTri(\Sigma)$.

\begin{defi}\label{d:rep_path}
An \emph{representation path} of a mapping class $\phi \in MC(\Sigma)$ is an edge path on $\bTri(\Sigma)$ from a vertex $(\tri,\ell)$ to $\phi^{-1}(\tri, \ell)$.
Then, we also say that the path $\gamma$ represents the mapping class $\phi$.
\end{defi}

\begin{rem}
It is clear that $B^{(\tri, \ell)} = B^{\phi(\tri, \ell)}$ .
Furthermore, one can verify that an edge path $\gamma: (\tri, \ell) \to (\tri', \ell')$ represents some mapping class iff $B^{(\tri, \ell)} = B^{(\tri', \ell')}$.
(\emph{cf}. \cite[Section 2]{IK19} and \cite{FST}.)
In a general cluster algebra, the elements of the \emph{cluster modular group} are defined by these properties (see \cref{sec:cluster_modular_group}). 
\end{rem}

We denote by $\pi: \bTri(\Sigma) \to \Tri(\Sigma)$ the graph contraction that shrinks all the vertical edges followed by the identification of the horizontal edges with the same underlying pair of ideal triangulations, forgetting the labelings for the vertices and the edges. See \cref{fig:pent}.  
For an edge path $\gamma$ on $\bTri(\Sigma)$, let $\pi(\gamma)$ be its image. If $\gamma$ consists only of vertical edges, then $\pi(\gamma)$ is a constant path.  

\begin{conv}
\begin{itemize}
    \item For two edge paths $\gamma:(\tri_0,\ell_0) \to (\tri_1,\ell_1)$ and $\gamma':(\tri_1,\ell_1) \to (\tri_2,\ell_2)$, we denote their concatenation by $\gamma\ast \gamma':(\tri_0,\ell_0) \to (\tri_2,\ell_2)$. 
    \item For an edge path 
\begin{align*}
    \gamma: (\tri_0, \ell_0) \overbar{k_0} (\tri_1, \ell_1) \overbar{k_1} \cdots \overbar{k_{h-1}} (\tri_h, \ell_h)
\end{align*}
only with horizontal edges, 
we abbreviate it as $\gamma: (\tri_0, \ell_0) \xrightarrow{k_0, \dots, k_{h-1}} (\tri_h, \ell_h)$. Its contraction is written as 
\begin{align*}
    \pi(\gamma): \tri_0 \overbar{\kappa_0} \tri_1 \overbar{\kappa_1} \cdots \overbar{\kappa_{h-1}} \tri_h
\end{align*}
with abbreviation $\pi(\gamma): \tri_0 \xrightarrow{\kappa_0, \dots, \kappa_{h-1}} \tri_h$, 
where $\tri_i \overbar{\kappa_i} \tri_{i+1}$ is the underlying flip of $(\tri_i, \ell_i) \overbar{k_i} (\tri_{i+1}, \ell_{i+1})$ with $\kappa_i \in \tri_i$. 
Similarly, an edge path only with vertical edges
\begin{align*}
    \gamma: (\tri_0, \ell_0) \overbar{\sigma_0} (\tri_1, \ell_1) \overbar{\sigma_1} \cdots \overbar{\sigma_{n-1}} (\tri_n, \ell_n)
\end{align*}
 is abbreviated as $\gamma: (\tri_0, \ell_0) \xrightarrow{\sigma_{n-1} \cdots \sigma_0} (\tri_n, \ell_n)$.
 \item The cyclic permutation $k_1 \mapsto k_2 \mapsto \dots \mapsto k_n \mapsto k_1$ is denoted by the symbol $(k_1\ k_2\ \cdots\ k_n)$.
\end{itemize}
\end{conv}


\begin{rem}\label{rem:product_path}
If $\gamma:(\tri_0,\ell_0) \to \phi^{-1}(\tri_0,\ell_0)$ and $\delta:(\tri_0,\ell_0) \to \psi^{-1}(\tri_0,\ell_0)$ represent mapping classes $\phi,\psi \in MC(\Sigma)$, respectively, then the concatenation $\gamma\ast \phi^{-1}(\delta):(\tri_0,\ell_0) \to \phi^{-1}(\psi^{-1}(\tri_0,\ell_0))$ represents $\psi\phi$. 
\end{rem}



\begin{ex}\label{ex:pentagon}
Let us consider a disk with 5 punctures on its boundary.
Then $I = \{1,2\}$.
The graph $\bTri(\Sigma)$ is shown in \cref{fig:pent}.
Focus on the labeled triangulation $(\tri, \ell)$ on the upper-left vertex.
The edge path $(\tri, \ell) \overbar{2} (\tri_1, \ell_1) \overbar{(1\ 2)} (\tri_2, \ell_2)$ represents the mapping class $\phi$ corresponding to the $4\pi/5$ rotation.
Check that $(\tri_2, \ell_2) = \phi^{-1}(\tri, \ell)$ in \cref{fig:pent}.
The path $(\tri, \ell) \overbar{(1\ 2)} (\tri_3, \ell_3) \overbar{1} \phi^{-1}(\tri, \ell)$ also represents $\phi$.
Both of them project to the same path on $\Tri(\Sigma)$ of length $1$.
The square $\phi^2$ admits a horizontal representation path $(\tri,\ell) \xrightarrow{1,2} \phi^{-2}(\tri,\ell)$. 
\begin{figure}[h]
    \centering
    \begin{tikzpicture}
    \draw(4*0.34202,0.5*0.93969) coordinate(P+);
    \draw(-4*0.34202,0.5*0.93969) coordinate(P-);
    \begin{scope}[yshift=-2cm]
    \draw(4*0.34202,0.5*0.93969) coordinate(Q+);
    \draw(-4*0.34202,0.5*0.93969) coordinate(Q-);
    \end{scope}
    \draw(P+) arc[x radius=4cm, y radius=0.5cm, start angle=70, end angle=-250];
    \draw(Q+) arc[x radius=4cm, y radius=0.5cm, start angle=70, end angle=-250];
    \draw[dashed] (P+) to[out=210, in=0] (Q-);
    \draw[dashed] (Q+) to[out=150, in=0] (P-);
    \fill (-4,0) circle(2pt) coordinate(A1);
    \fill (4,0) circle(2pt) coordinate(A2);
    \fill (0,-0.5) circle(2pt) coordinate(A3);
    \fill (-4*0.707,-0.5*0.707) circle(2pt) coordinate(A4);
    \fill (4*0.707,-0.5*0.707) circle(2pt) coordinate(A5);
    \begin{scope}[yshift=-2cm]
    \fill (-4,0) circle(2pt) coordinate(B1);
    \fill (4,0) circle(2pt) coordinate(B2);
    \fill (0,-0.5) circle(2pt) coordinate(B3);
    \fill (-4*0.707,-0.5*0.707) circle(2pt) coordinate(B4);
    \fill (4*0.707,-0.5*0.707) circle(2pt) coordinate(B5);
    \end{scope}
    \foreach \i in {1,2,3,4,5}
    \draw[thick](A\i) -- (B\i);
    \draw(-4*0.5,-0.5*0.8660-2) node[below right]{$2$};
    \draw(4*0.5,-0.5*0.8660-2) node[below left]{$1$};
    \draw(-4*0.5,0.5*0.8660) node[above right]{$1$};
    \draw(-4,-1) node[left]{$(1\ 2)$};
    \draw[->] (0,-3.1) -- (0,-4.5);
    \draw (0,-6) ellipse (4 and 1);
    \node [fill, circle, inner sep=1.5pt] at (-4,-6) {};
    \node [fill, circle, inner sep=1.5pt] at (-2.85,-6.7) {};
    \node [fill, circle, inner sep=1.5pt] at (0,-7) {};
    \node [fill, circle, inner sep=1.5pt] at (2.85,-6.7) {};
    \node [fill, circle, inner sep=1.5pt] at (4,-6) {};
    \node at (-0.3,-3.8) {$\pi$};
    \node at (3.5,1) {$\bTri(\Sigma)$};
    \node at (3.5,-4.7) {$\Tri(\Sigma)$};
    \node (v1) at (-3.6,-2.45) {1};
    \node at (3.5,-2.5) {2};
    \node at (-3.5,-0.5) {2};
    
    \draw (-5.05,-5.1) coordinate (v2) -- (-5.75,-5.65) -- (-5.4,-6.4) coordinate (v7) -- (-4.65,-6.4) coordinate (v8) {} -- (-4.4,-5.65) -- (v2);
    \draw  (v2) edge (v7);
    \draw  (v2) edge (v8);
    \draw (5.1,-5.1) coordinate (v3) -- (4.4,-5.65) -- (4.75,-6.4) coordinate (v17) {} -- (5.5,-6.4) -- (5.75,-5.65) coordinate (v18) {} -- (v3);
    \draw  (v17) edge (v18);
    \draw  (v3) edge (v17);
    \draw (-3.15,-7.05) coordinate (v6) -- (-3.85,-7.6) coordinate (v10) {} -- (-3.5,-8.35) -- (-2.8,-8.35) coordinate (v9) {} -- (-2.5,-7.6) -- (v6);
    \draw  (v6) edge (v9);
    \draw  (v9) edge (v10);
    \draw (3.15,-7.05) coordinate (v4) -- (2.45,-7.6) coordinate (v14) {} -- (2.8,-8.35) coordinate (v16) {} -- (3.55,-8.35) -- (3.8,-7.6) coordinate (v15) {} -- (v4);
    \draw  (v14) edge (v15);
    \draw  (v16) edge (v15);
    \draw (0,-7.5) coordinate (v5) -- (-0.7,-8.05) coordinate (v11) {} -- (-0.35,-8.8) -- (0.4,-8.8) coordinate (v12) {} -- (0.65,-8.05) coordinate (v13) {} -- (v5);
    \draw  (v11) edge (v12);
    \draw  (v11) edge (v13);
    \draw (-5.15,-1.5) coordinate (v2) -- (-5.85,-2.05) -- (-5.5,-2.8) coordinate (v7) -- (-4.75,-2.8) coordinate (v8) {} -- (-4.5,-2.05) -- (v2);
    \draw  (v2) edge (v7);
    \draw  (v2) edge (v8);
    \draw (-5.15,1) coordinate (v2) -- (-5.85,0.45) -- (-5.5,-0.3) coordinate (v7) -- (-4.75,-0.3) coordinate (v8) {} -- (-4.5,0.45) -- (v2);
    \draw  (v2) edge (v7);
    \draw  (v2) edge (v8);
    \draw (5.1,-1.5) coordinate (v3) -- (4.4,-2.05) -- (4.75,-2.8) coordinate (v17) {} -- (5.5,-2.8) -- (5.75,-2.05) coordinate (v18) {} -- (v3);
    \draw  (v17) edge (v18);
    \draw  (v3) edge (v17);
    \draw (-3.1,-2.8) coordinate (v6) -- (-3.8,-3.35) coordinate (v10) {} -- (-3.45,-4.1) -- (-2.75,-4.1) coordinate (v9) {} -- (-2.45,-3.35) -- (v6);
    \draw  (v6) edge (v9);
    \draw  (v9) edge (v10);
    
    \node at (-5.45,0.4) {\scriptsize 2};
    \node at (-4.85,0.4) {\scriptsize 1};
    \node at (-5.45,-2.15) {\scriptsize 1};
    \node at (-4.8,-2.15) {\scriptsize 2};
    \node at (-3.35,-3.5) {\scriptsize 1};
    \node at (-2.75,-3.45) {\scriptsize 2};
    \node at (4.8,-2.05) {\scriptsize 2};
    \node at (5.3,-2.2) {\scriptsize 1};
    
    \node at (-6.7,0.25) {$(\tri, \ell) =$};
    \node at (-5,-3.6) {$\phi^{-1}(\tri, \ell) =$};
    
    \draw [ultra thick, red, -{stealth}](-4,0) .. controls (-4,-0.2) and (-3.35,-0.3) .. (-2.85,-0.35) .. controls (-2.85,-0.85) and (-2.85,-1.35) .. (-2.85,-2.35);
    \draw [ultra thick, red, -{stealth}](-4,-6) .. controls (-4,-6.3) and (-3.4,-6.55) .. (-2.85,-6.7);
    \node [red] at (-2.6,-1) {$\gamma$};
    \node [red] at (-3.35,-6.15) {$\pi(\gamma)$};
    \end{tikzpicture}
    \caption{The graph the labeled triangulations of the pentagon.}
    \label{fig:pent}
\end{figure}
\end{ex}

\begin{ex}\label{ex:Dehn_path}
Let $\Sigma$ be an annulus with one special point on each boundary component. Let $C$ be a simple closed curve whose homotopy class generates $\pi_1(\Sigma) \cong \bZ$, and $T_C \in MC(\Sigma)$ the (right-hand) Dehn twist along $C$. Choose a labeled triangulation $(\tri,\ell)$ as shown in the top left of \cref{fig:example_Dehn}. Here
$I(\Sigma)=\{1,2\}$. 
Then one can see from the figure that the path 
\begin{align*}
    \gamma_C: (\tri,\ell) \overbar{1} (\tri',\ell') \overbar{(1\ 2)} T_C^{-1}(\tri,\ell)
\end{align*}
represents $T_C$.

\begin{figure}[h]
    \centering
    \begin{tikzpicture}[scale=0.85]
    \draw  (0,0) node (v1) {} ellipse (0.5 and 0.5);
    \draw  (v1) ellipse (2 and 2);
    \draw [red,thick] (0,0) circle(1cm);
    \node [fill, circle, inner sep=1.3pt] at (0,0.5) {};
    \node [fill, circle, inner sep=1.3pt] at (0,2) {};
    \draw [blue] (0,2) -- (0,0.5) {};
    \draw [blue](0,0.5) .. controls (1.85,2) and (2.15,-1.6) .. (0,-1.6) .. controls (-2.15,-1.6) and (-1.8,1) .. (0,2);
    \node [blue] at (-0.2,1.45) {$2$};
    \node [blue] at (1.55,0.5) {$1$};
    \node [blue] at (-0.5,-0.4) {$3$}; 
    \node [blue] at (-2.25,0) {$4$};
    \node [red] at (0,-1.2) {$C$};
    \draw [squigarrow, thick] (6.5,0) -- (3.5,0) node[midway,above]{$T_C$};
    \node at (-2.5,-1.5) {$(\tri,\ell)$};
    \begin{scope}[xshift=10cm]
    \draw  (0,0) node (v1) {} ellipse (0.5 and 0.5);
    \draw  (v1) ellipse (2 and 2);
    \draw [red,thick] (0,0) circle(1cm);
    \node [fill, circle, inner sep=1.3pt] at (0,0.5) {};
    \node [fill, circle, inner sep=1.3pt] at (0,2) {};
    \draw [blue] (0,2) -- (0,0.5) {};
    \draw [blue](0,0.5) .. controls (-1.85,2) and (-2.15,-1.6) .. (0,-1.6) .. controls (2.15,-1.6) and (1.8,1) .. (0,2); 
    \node [blue] at (0.2,1.45) {$1$};
    \node [blue] at (-1.55,0.5) {$2$};
    \node [blue] at (-0.5,-0.4) {$3$}; 
    \node [blue] at (-2.25,0) {$4$};
    \node [red] at (0,-1.2) {$C$};
    \node at (2.6,-1.6) {$T_C^{-1}(\tri,\ell)$};
    \end{scope}
    %
    \begin{scope}[xshift=5cm,yshift=-4cm]
    \draw  (0,0) node (v1) {} ellipse (0.5 and 0.5);
    \draw  (v1) ellipse (2 and 2);
    \draw [red,thick] (0,0) circle(1cm);
    \node [fill, circle, inner sep=1.3pt] at (0,0.5) {};
    \node [fill, circle, inner sep=1.3pt] at (0,2) {};
    \draw [blue] (0,2) -- (0,0.5) {};
    \draw [blue](0,0.5) .. controls (-1.85,2) and (-2.15,-1.6) .. (0,-1.6) .. controls (2.15,-1.6) and (1.8,1) .. (0,2); 
    \node [blue] at (0.2,1.45) {$2$};
    \node [blue] at (-1.55,0.5) {$1$};
    \node [blue] at (-0.5,-0.4) {$3$}; 
    \node [blue] at (-2.25,0) {$4$};
    \node [red] at (0,-1.2) {$C$};
    \node at (2.5,-1.5) {$(\tri',\ell')$};
    \end{scope}
    \draw [thick] (1.75,-1.75) -- (2.85,-2.85) node[midway,below left]{$1$};
	\draw [thick] (7,-2.85) -- (8.25,-1.75) node[midway,below right]{$(1\ 2)$};   
    \end{tikzpicture}
    \caption{The mutation loop given by the Dehn twist $T_C$.\CORRECTED}
    \label{fig:example_Dehn}
\end{figure}

\end{ex}

\subsubsection{Coordinate expression of the action of mapping classes}
As we mentioned in \cref{subsec:NT_classif}, the mapping class group acts on the space $\eML(F)$ for a hyperbolic surface $F$ of type $\Sigma$.
By using representation paths, we describe this action as follows.
Let $\gamma: (\tri, \ell) \to \phi^{-1}(\tri, \ell)$ be a representation path of a mapping class $\phi \in MC(\Sigma)$.
Then, we have the following diagram:
\begin{equation}
\begin{tikzcd}
    \eML(F) \ar[r, equal] \ar[d, "\boldsymbol{\sfx}^{(\tri,\ell)}"] & \eML(F) \ar[r, "\phi"] \ar[d, "\boldsymbol{\sfx}^{\phi^{-1}(\tri, \ell)}"] & \eML(F) \ar[d, "\boldsymbol{\sfx}^{(\tri,\ell)}"]\\
    \bR^{I} \ar[r, "\mu_{\gamma}"] & \bR^{I} \ar[r, "\phi_\ast"] & \bR^{I}
\end{tikzcd}
\end{equation}
Here, 
$\mu_\gamma$ denotes the composite of the coordinate changes and transpositions along $\gamma$, and $\phi_\ast$ denotes the identification $\sfx^{\phi^{-1}(\tri)}_{\phi^{-1}(\alpha)} \mapsto \sfx^{\tri}_\alpha$ for all $\alpha \in \tri$. Therefore, the composite $\phi_{(\tri,\ell)}:=\phi_\ast \circ \mu_\gamma$ gives the \emph{coordinate expression} of the $\phi$-action in the coordinate system $\bsfx^{(\tri, \ell)}$, which is a PL map. 

\begin{ex}\label{ex:Dehn_path_2}
In the previous example, we get the coordinate expression
\begin{align*}
    (T_C)_{(\tri,\ell)}(\sfx_1,\sfx_2) = (\sfx_2 + \max\{0,\sfx_1\}, -\sfx_1)
\end{align*}
of the Dehn twist $T_C$ in the coordinate system $\bsfx^{(\tri,\ell)}=(\sfx_1,\sfx_2)$. 
\end{ex}

\subsection{Sign stability}
For a real number $a \in \bR$, let $\sgn(a)$ denote its sign:
\[
\sgn(a):=
\begin{cases}
    + & \mbox{ if } a>0,\\
    0 & \mbox{ if } a=0,\\
    - & \mbox{ if } a<0.
\end{cases}
\]
We say that $\sgn(a)$ is \emph{strict} if $\sgn(a)\neq 0$. 

In this subsection, we fix a marked surface $\Sigma$ and a marked hyperbolic surface $(F,f)$ with geodesic boundary of type $\Sigma$. 
The following expression is useful in the sequel:

\begin{lem}[{\cite[Lemma 3.1]{IK19}}]\label{l:x-cluster signed}
For ideal triangulations $\tri$ and $\tri'$ of $\Sigma$ such that $\tri' \setminus \{\kappa'\} = \tri \setminus \{\kappa\}$, the coordinate transformation $\boldsymbol{\sfx}^{\tri'} \circ (\boldsymbol{\sfx}^\tri)^{-1}$ of the space $\eML(F)$ can be rewritten as
\begin{align}\label{eq:sign x-cluster}
    \sfx^{\tri'}_\alpha =
\begin{cases}
    -\sfx^\tri_\kappa & \mbox{if $\alpha=\kappa'$}, \\
    \sfx^\tri_\alpha + [\sgn(\sfx^{\tri}_\kappa) b^{\tri}_{\kappa  \alpha}]_+\sfx^{\tri}_\kappa & \mbox{if $\alpha \neq \kappa'$}.
\end{cases}
\end{align}
\end{lem}

\begin{dfn}[sign of a path]\label{d:sign}
Given a path $\gamma: (\tri, \ell) \to (\tri', \ell')$ in $\bTri(\Sigma)$ with $\pi(\gamma): \tri_0 \overbarnear{\kappa_0} \tri_1 \overbarnear{\kappa_1} \cdots \overbarnear{\kappa_{h-1}} \tri_h$ and a measured geodesic lamination $L \in \eML(F)$, the \emph{sign} $\boldsymbol{\epsilon}_\gamma(L)$ of $\gamma$ at $L$ is the sequence
\begin{align*}
    \bep_\gamma(L)=(\sgn(\sfx^{\tri_0}_{\kappa_0}(L)), \dots, \sgn(\sfx^{\tri_{h-1}}_{\kappa_{h-1}}(L))).
\end{align*}
\end{dfn}

The sign $\bep_\gamma(L)$ tells us the domain of linearity of the PL isomorphism $\mu_\gamma: \bR^{\tri} \to \bR^{\phi^{-1}(\tri)}$.
For $\bep \in \{+, -\}^h$, we define
\begin{align}\label{eq:sign_cone}
    \cC^{\bep}_\gamma := \overline{\{ L \in \eML(F) \mid \bep_\gamma(L) = \bep \}}.
\end{align}
This subspace is a cone (\emph{i.e.}, an $\bR_{>0}$-invariant convex subset) in $\eML(F)$.

\begin{lem}[{\cite[Lemma 3.5]{IK19}}]
If $\cC^{\bep}_\gamma \neq \emptyset$, then $\bsfx^{(\tri,\ell)}(\cC^{\bep}_\gamma)$ is a domain of linearity of the piecewise linear isomorphism $\mu_{\gamma}: \bR^{I} \to \bR^{I}$.
\end{lem}

We denote by $E^{\bep}_\gamma$ the presentation matrix of the linear extension of the restriction $\mu_\gamma|_{\bsfx^{(\tri, \ell)}(\cC^{\bep}_\gamma)}$.

\begin{dfn}[sign stability]\label{d:sign stability}
Let $\phi \in MC(\Sigma)$ be a mapping class, and let $\gamma: (\tri, \ell) \to \phi^{-1}(\tri, \ell)$ be its representation path such that $\pi(\gamma)$ has length $h$.
Let $\Omega \subset \eML(F)$ be a subset which is invariant under the rescaling action of $\bR_{> 0}$, which we call a \emph{domain of stability}. 

Then we say that $\gamma$ is \emph{sign-stable} on $\Omega$ if there exists a sequence $\bep^\stab \in \{+,-\}^h$ of strict signs such that for each $w \in \Omega \setminus \{0\}$, there exists an integer $n_0 \in \mathbb{N}$ such that   \[\boldsymbol{\epsilon}_\gamma(\phi^n(w)) = \boldsymbol{\epsilon}^\stab \]
for all $n \geq n_0$. 
We call $\bep^\stab=\boldsymbol{\epsilon}_{\gamma,\Omega}^\stab$ the \emph{stable sign} of $\gamma$ on $\Omega$.
Also, we write $E^{(\tri, \ell)}_{\phi, \Omega} := E^{\bep^\stab}_\gamma$ for the presentation matrix and call it \emph{stable presentation matrix}.
\end{dfn}

We note that the stable presentation matrix $E^{(\tri,\ell)}_{\phi, \Omega}$ depends only on the mutation loop $\phi$ and the initial labeled triangulation $(\tri,\ell)$, since it is a presentation matrix of the restriction of the action $\phi: \eML(F) \to \eML(F)$ to one of its domain of linearity with respect to the coordinate system $\bsfx^{(\tri,\ell)}$ (\emph{cf}. \cite[Corollary 3.7]{IK19}).

\subsection{Basic sign stability and the cluster stretch factor}
Let $\cC^+_\tri \subset \eML(F)$ be the cone consisting of the measured geodesic laminations $(G, \mu)$ with the support $G = \bigcup f(\tri)$. 
It is characterized in the coordinate system $\bsfx^\tri$ by
\begin{align*}
    \cC^+_\tri = \{ L \in \eML(F) \mid  \sfx^\tri_\alpha(L) \geq 0 \text{ for all } \alpha \in \tri \}.
\end{align*}
Define another cone $\cC^-_\tri \subset \eML(F)$ by replacing all $\sfx^\tri_\alpha$ with $-\sfx^\tri_\alpha$ in the coordinate description above. 
Every path $\gamma$ starting from $(\tri, \ell)$ has a constant sign in each of the interiosr of the cones $\cC^+_{\tri}$ and $\cC^-_{\tri}$, respectively, which has the special meaning in the theory of cluster algebra \cite[Lemma 3.12, Corollary 3.13]{IK19}. 
In this sense, the sign stability on the set
\begin{align}\label{eq:Omega^can}
    \Omega^{\mathrm{can}}_{\tri}:=\interior\cC^+_{\tri} \cup \interior\cC^-_{\tri}
\end{align}
is most fundamental, and it turns out that it is sufficient for the computation of the algebraic entropy of cluster transformations (\cref{subsec:entropy}).
\begin{dfn}\label{d:cluster stretch factor}
A representation path $\gamma: (\tri, \ell) \to \phi^{-1}(\tri, \ell)$ of a mapping class $\phi \in MC(\Sigma)$ is said to be \emph{basic sign-stable} if it is sign-stable on $\Omega^\mathrm{can}_{\tri}$.
In this case, the spectral radius
$\lambda_{\phi}^{\tri} \geq 1$ of the stable presentation matrix $E_{\phi,\Omega_\tri^{\mathrm{can}}}^{(\tri,\ell)}$ is called
the \emph{cluster stretch factor} of $\phi$.
\end{dfn}

We remark that the cluster stretch factor $\lambda_\phi^{\tri}$ does not depend on the initial triangulation $\tri$ under a certain condition \cite[Remark 3.16]{IK19}. It is also the Perron--Frobenius eigenvalue of the stable presentation matrix \cite[Theorem 3.12]{IK19}.

\begin{ex}\label{ex:Dehn_SS}
We continue to study the Dehn twist $\phi := T_C$ along $C$ in \cref{ex:Dehn_path,ex:Dehn_path_2}.
Since the path $\gamma_C$ contains a single flip, the possible sign sequences are $\bep=(+)$ or $(-)$. 
We have
\begin{align*}
    E_{\gamma_C}^{(+)}=\begin{pmatrix}
    0 & 1 \\
    -1   & 0
    \end{pmatrix} \quad\mbox{and} \quad
    E_{\gamma_C}^{(-)}=\begin{pmatrix}
    2 & 1 \\
    -1   & 0
    \end{pmatrix}.
\end{align*}
Then the path $\gamma_C$ is sign-stable on the entire space $\Omega:=\eML(F)$ with the stable sign $\boldsymbol{\epsilon}_{\gamma_C}^\stab=(-)$. Indeed, one can verify that the cone
\begin{align*}
    \cC = \{ \sfx_{2} \geq 0,\ \sfx_{1} + \sfx_{2} \geq 0 \} 
\end{align*}
is $\phi$-stable (\emph{i.e.}, $\phi(\cC) \subset \cC$), and the orbit $(\phi^n(L))_{n \geq 0}$ of any $L \in \eML(F)$ eventually enters this cone.
In particular, the ray $\bR_{>0} \cdot L_\phi$ with $(\sfx_{1}(L_{\phi}), \sfx_{2}(L_{\phi})) = (-1, 1)$ is fixed by $\phi$.
Moreover, the corresponding eigenvalue $1$ is the cluster stretch factor of $\phi$.
\end{ex}

\begin{figure}[h]
    \centering
    \begin{tikzpicture}[scale=1.2]
    \coordinate (v1) at (0,0) {};
    \coordinate (v3) at (0,3) {};
    \coordinate (v4) at (0,-3) {};
    \coordinate (v6) at (-3,3) {} {};
    \filldraw [draw=yellow!35, fill=yellow!35] (3,0) node[right]{$\sfx^\tri_{\ell(1)}$} -- (3,3) -- (-3,3) -- (0,0) -- (3,0);
    \draw [dashed] (v3) edge (v1) node[above] {$\sfx^\tri_{\ell(2)}$};
    \draw [dashed] (v1) edge (v4);
    \draw (-3,0) -- (3,0);
    \node[color=orange] at (1.7,1.3) {$\cC$};
    \coordinate (v8) at (-3,3) {} {} {};
    \draw[red, thick]  (v1) edge (v8);
    \node[red] at (-2,1.5) {$L_{\phi}$};
    \end{tikzpicture}
    \caption{An invariant cone (yellow region) and the fixed ray (red ray) of the Dehn twist $\phi = T_C$ along $C$.}
    \label{fig:cones_kroc}
\end{figure}

Further examples of sign-stable representation paths with cluster stretch factor equal to $1$ or larger than $1$ are provided in \cref{ex:Dehn_SS} and \cref{ex:4_sph}. 

\subsection{Uniform sign stability}\label{subsec:uniform_stability}
Let us consider the $\bR_{>0}$-invariant subset $\eML_\bQ(F) \subset \eML(F)$ consistsing of the measured geodesic laminations having only the non-compact part $B$ and the closed geodesics part $S$ in terms of the splitting given in \cref{l:structure_of_lamination}.

\begin{dfn}[Uniform sign stability]\label{d:uniform stability}
A mapping class $\phi \in MC(\Sigma)$ is said to be \emph{uniformly sign-stable} if any representation path $\gamma: (\tri, \ell) \to \phi^{-1}(\tri, \ell)$ of $\phi$ is sign-stable on $\eML_\bQ(F)$. 
\end{dfn}
Note that $\eML_\bQ(\Sigma)$ contains $\cC^+_\tri$ for all $\tri \in \Tri(\Sigma)$.
Thus, $\Omega^{\mathrm{can}}_{\tri} \subset \eML_\bQ(\Sigma)$,
so a uniformly sign-stable mutation loop has a well-defined cluster stretch factor $\lambda_\phi^\tri$.


\smallskip

\subsection{Relation to the cluster-pseudo-Anosov property}\label{subsec:cluster-pA}

Let us consider the subspaces 
\begin{align*}
    |\fF^\pm_\Sigma|:= \bigcup_{\tri \in \Tri(\Sigma)} \cC_\tri^\pm \subset \eML(F)
\end{align*}
Obviously, $|\fF^\pm_\Sigma|$ are contained in $\eML_\bQ(F)$. It is known that $|\fF^+_\Sigma|=|\fF^-_\Sigma|$ unless $\Sigma$ is a once-punctured surface, and that $|\fF^+_\Sigma| \cup |\fF^-_\Sigma| \subset \eML(F)$ is dense \cite{Yurikusa}. 

Recall that a mapping class $\phi \in MC(\Sigma)$ is said to be \emph{cluster-pA} if no non-trivial power of $\phi$ fixes any point in $|\fF_\Sigma^+|$. 
This definition is equivalent to \cite[Definition 2.1]{Ish19}, thanks to \cite[Theorem 2.13]{GHKK}. 

\begin{prop}\label{prop:uniform_SS_cluster_pA}
A uniformly sign-stable mapping class is cluster-pA.
\end{prop}

\begin{proof}
First, we note that the collection
\begin{align*}
    \fF_\Sigma^\pm:=\{ \text{the faces of } \cC_{\tri}^\pm \mid \tri \in \Tri(\Sigma) \}
\end{align*}
of cones in $\eML(F)$ forms a simplicial fan (in each coordinate system $\bsfx^\tri$ for $\tri \in \Tri(\Sigma)$) by \cite[Theorem 2.13]{GHKK}, called the \emph{cluster complex}.
Moreover, the support of these fans are nothing but the subspaces $|\fF^\pm_\Sigma|$, respectively.

Suppose that for a mapping class $\phi$, there exists an integer $m > 0$ and a measured geodesic lamination $L \in |\fF_\Sigma^+|$ such that $\phi^m(L)=L$.
Take an ideal triangulation $\tri_0$ of $\Sigma$ such that $L \in \cC^+_{\tri_0}$. 
Since the action of a mapping class on the cluster complex is simplicial, $\phi^m$ permutes the faces of the cones $\cC^+_{\tri_0}$ which contain $L$. Then there exists an integer $r > 0$ such that the power $\phi^{mr}$ fixes the face $F_\kappa$ defined by $\sfx^{\tri_0}_\kappa = 0$ of $\cC^+_{\tri_0}$ for some $\kappa \in \tri$.

Now we choose a representation path $\gamma$ of $\phi$ starting $(\tri_0, \ell_0)$ so that the first two edges of $\pi(\gamma)$ are of the form $\tri_0 \overbarnear{\kappa} \tri_1 \overbarnear{\kappa} \tri_0$.
Then the sign $\boldsymbol{\epsilon}_\gamma(\phi^n(L))$ must contain a zero entry for all $n \in mr\bZ_{>0}$ from the definition of $F_\kappa$.
In particular the sign never stabilizes to a strict one, and hence $\gamma$ is not sign-stable on $|\fF_\Sigma^+|$.
Thus $\phi$ is not uniformly sign-stable.  
\end{proof}

\section{Sign stability of Dehn twists}\label{sec:Dehn twist}

Here we discuss the sign stability of Dehn twists. 
Let $C$ be a simple closed curve on a marked surface $\Sigma$ satisfying the following condition:
\begin{align}\label{eq:condition_Dehn}
    \mbox{There exists at least one marked point on each connected component of $\Sigma \setminus C$.}
\end{align}
On each side of $C$, choose a loop based at a marked point which is freely homotopic (forgetting the base-point) to the curve $C$ so that these two loops form a tubular neighborhood $\cN(C)$ of $C$. See \cref{fig:tubular_nbd}. 
\begin{figure}[h]
    \centering
    \begin{tikzpicture}
    \draw (3.6,-1) coordinate (v2) arc [start angle=-90, end angle=90, x radius=.45cm, y radius=1.5cm];
    \draw [dashed] (3.6,2) coordinate (v1) arc [start angle=90, end angle=270, x radius=.45cm, y radius=1.5cm];
    \draw (-0.5,2.4) .. controls (1.5,1.9) and (5.5,1.9) .. (7.5,2.4);
    \draw (-0.5,-1.4) .. controls (1.5,-0.9) and (5.5,-0.9) .. (7.5,-1.4);
    \node [fill, circle, inner sep=1.3] at (0,0.5) {};
    \node [fill, circle, inner sep=1.3] at (7,0.5) {};
    \draw [blue] (7,0.5) .. controls (5.5,0.5) and (5.05,1) .. (5,1.5) .. controls (4.95,1.95) and (4.8,2) .. (4.6,2.05);
    \draw [blue] (7,0.5) .. controls (5.5,0.5) and (5.05,0) .. (5,-0.5) .. controls (4.95,-0.95) and (4.8,-1) .. (4.6,-1.05);
    \draw [blue] (0,0.5) .. controls (1.5,0.5) and (2.5,1) .. (2.7,1.5) .. controls (2.85,1.9) and (2.6,2.05) .. (2.5,2.05);
    \draw [blue] (0,0.5) .. controls (1.5,0.5) and (2.55,0) .. (2.75,-0.5) .. controls (2.9,-0.9) and (2.65,-1.05) .. (2.5,-1.05);
    \draw [blue, dashed] (2.5,2.05) arc [start angle=90, end angle=270, x radius=.45cm, y radius=1.55cm];
    \draw [blue, dashed] (4.6,2.05) arc [start angle=90, end angle=270, x radius=.45cm, y radius=1.55cm];
    \node at (3.6,-1.35) {$C$};
    \end{tikzpicture}
    \caption{A tubular neighborhood $\cN(C)$ with marked points.}
    \label{fig:tubular_nbd}
\end{figure}
Then $\cN(C)$ can be regarded as a marked surface: an annulus with one marked point on each of its boundary component. Take a labeled triangulation of $\cN(C)$ as shown in the top-left of \cref{fig:example_Dehn}, 
and extend it arbitrarily to a labeled triangulation $(\tri_C,\ell_C)$ of $\Sigma$.
Then as we have seen in \cref{ex:Dehn_path}, the path $\gamma_C: (\tri_C,\ell_C) \overbar{1} (\tri',\ell') \overbar{(1\ 2)} T_C^{-1}(\tri_C,\ell_C)$ represents the mutation loop $T_C \in  MC(\Sigma)$. 

Fix a marked hyperbolic surface $(F,f)$ with geodesic boundary of type $\Sigma$.
Let us consider the $\bR_{>0}$-invariant set $\Omega_C \subset \eML(F)$ consisting of the measured geodesic laminations whose supports intersect with the closed geodesic $G_C$ homotopic to $C$.
It contains the subset $\Omega^{\mathrm{can}}_{\tri_C}$. Let us write $\sfx_i:=\sfx_i^{(\tri_C,\ell_C)}$ for $i \in I$. 


\begin{prop}[Sign stability of Dehn twists]\label{lem:Dehn_SS}
Let $T_C$ be the Dehn twist along a simple closed curve $C$ satisfying \eqref{eq:condition_Dehn}. Then the representation path $\gamma_C: (\tri_C,\ell_C) \xrightarrow{1} (\tri',\ell') \xrightarrow{(1\ 2)} T_C^{-1}(\tri_C,\ell_C)$ is sign-stable on the $\bR_{>0}$-invariant set $\Omega_{(\tri_C,\ell_C)}^{(1,2)}$. 
Its cluster stretch factor is $1$.
\end{prop}

\begin{proof}

First, we note that the subset $\Omega_C$ is characterized in the coordinate system $\bsfx^{(\tri_C,\ell_C)}$ as 
\begin{align*}
    \Omega_C = \{ L \in \eML(F) \mid (\sfx_{1}(L), \sfx_{2}(L)) \neq (0,0)\}.
\end{align*}
We have investigated the action of $T_C$ in the direction $(\sfx_1, \sfx_2)$ in \cref{ex:Dehn_SS}.
Since the sign of the path $\gamma$ only concerns with these two coordinates, the first assertion follows. For the second statement, notice that the stable presentation matrix $E_{T_C}^{(\tri_C,\ell_C)}$ is block-decomposed as follows:
\begin{align}\label{eq:E_matrix_Dehn_twist}
    E_{T_C}^{(\tri_C,\ell_C)}=E_{\gamma_C}^{(-)}=
    \left(
    \begin{array}{c|c}
    \begin{array}{cc}
         2 & 1 \\
         -1  & 0 
    \end{array}
    & \text{\huge$0$} \\ \hline
    \text{\huge$\ast$} & 
    \begin{array}{ccc}
        1 &      & \\
          &\ddots& \\
          &      & 1
    \end{array}
    \end{array}
    \right)
\end{align}
Here the first two directions correspond to $\{1,2\}$. Then we get $\lambda_\phi=\rho(E_{T_C}^{(\tri_C,\ell_C)}) =1$ as desired.
\end{proof}


Note that the limit $L_{T_C} \in \eML(F)$ given by $\sfx_1(L_{T_C}) = -\sfx_2(L_{T_C}) = -1$ and $\sfx_j(L_{T_C}) = 0$ for $j \notin \{1,2\}$ is identified with measured geodesic lamination $(G_C,\delta_C)$.
Here, $\delta_C := \delta_{G_C}$ (recall the notation in \cref{ex:MGL}).
Therefore, $L_{T_C} \in \ML_0(F)$.

\begin{rem}
In fact, the $\bR_{>0}$-invariant set $\Omega_C$ is described using the ``geometric intersection number" $\mathcal{I}_C$ for measured geodesic laminations (\emph{cf}. \cite[Corollary 12.2]{FG03}):
\begin{align*}
    \Omega_C = \{ L \in \eML(F) \mid \mathcal{I}_C(L) \neq 0\} \sqcup \bR_{>0}(G_C, \delta_C).
\end{align*}
The map $\mathcal{I}_C$ is PL in the shear coordinates. In particular, 
\begin{align*}
    \mathcal{I}_C\circ (\sfx^{(\tri_C,\ell_C)})^{-1} = \max\left\{ \frac{\sfx_1 + \sfx_2}{2},\frac{-\sfx_1 + \sfx_2}{2}, -\frac{\sfx_1 + \sfx_2}{2}\right\}.
\end{align*}
Indeed, it can be obtained as the tropical limit of the trace function along $C$ on the enhanced \Teich\ space \cite[Theorem 12.2]{FG03}. 
\end{rem}

\begin{ex}\label{ex:Dehn_SS_2}
Let $\Sigma$ be a once-punctured surface of genus 2. Let $C \subset \Sigma$ be the closed curve shown in \cref{fig:genus_2}, and consider the Dehn twist $T_C$ along $C$. 
A labeled triangulation $(\tri,\ell)$ of $\Sigma$ is shown in the left of \cref{fig:Dehn_genus_2}. Note that the arcs $3$ and $7$ bound a tubular neighborhood $\cN(C)$ of $C$.
\begin{figure}[h]
    \centering
    \begin{tikzpicture}
    \draw (-0.05,-0.1) .. controls (-0.2,-0.35) and (0.1,-0.55) .. (0.5,-0.55) .. controls (0.9,-0.55) and (1.2,-0.35) .. (1.05,-0.1);
    \draw (0,-0.4) .. controls (0.1,-0.05) and (0.9,-0.05) .. (1,-0.4);
    \draw (4.45,-0.05) .. controls (4.3,-0.3) and (4.6,-0.5) .. (5,-0.5) .. controls (5.4,-0.5) and (5.7,-0.3) .. (5.55,-0.05);
    \draw (4.5,-0.35) .. controls (4.6,0) and (5.4,0) .. (5.5,-0.35);
    \draw  (2.75,-0.4) ellipse (4 and 2.2);
    \node [fill, circle, inner sep=1.5pt] (p) at (3.15,-2) {};
    \draw [blue] (2.75, -2.6) arc [start angle=-90, end angle=90, x radius=.6cm, y radius=2.2cm];
    \draw [blue, dashed] (2.75,1.8) arc [start angle=90, end angle=270, x radius=.6cm, y radius=2.2cm];
    \draw [blue] (p) .. controls (0.5,-2) and (-1.5,-0.8) .. (-0.2,0.45) .. controls (1.15,1.6) and (2.55,-0.5) .. (p);
    \draw [blue] (p) .. controls (5.4,-1.75) and (6.9,-0.6) .. (5.85,0.45) .. controls (4.9,1.4) and (3.75,0.1) .. (p);
    \draw [blue] (p) .. controls (2.5,-1) and (1.5,-0.55) .. (0.5,-0.55);
    \draw [blue, dashed] (0.5,-0.55) .. controls (-0.15,-0.55) and (-0.4,-1.65) .. (0,-2);
    \draw [blue] (0,-2) .. controls (0.45,-2.2) and (1.5,-2.5) .. (p);
    \draw [blue] (p) .. controls (4,-1) and (4.45,-0.5) .. (5,-0.5);
    \draw [blue, dashed] (5,-0.5) .. controls (5.6,-0.5) and (5.85,-1.8) .. (5.35,-2.05);
    \draw [blue] (5.35,-2.05) .. controls (4.8,-2.25) and (4,-2.35) .. (p);
    \draw [red] (-1.25,-0.4) .. controls (-1.25,0.05) and (-0.15,0.05) .. (0,-0.4);
    \draw [red, dashed] (0,-0.4) .. controls (0,-0.9) and (-1.25,-0.9) .. (-1.25,-0.4);
    \node [blue] at (3.35,1.15) {1};
    \node [blue] at (1.05,0.95) {2};
    \node [blue] at (1.35,-0.9) {3};
    \node [blue] at (4.55,0.85) {4};
    \node [blue] at (4.35,-1) {5};
    \node [red] at (-1.5,0) {$C$};
    \end{tikzpicture}
    \caption{Once-punctured surface of genus 2. Here only some of the arcs in the labeled triangulation $(\tri,\ell)$ are shown.}   \label{fig:genus_2}
\end{figure}

\begin{figure}[h]
    \centering
\begin{tikzpicture}[xscale=-1]
\draw [blue] (6,3.5) coordinate (v1) -- (4,2.5) coordinate (v2) -- (3,0.5) coordinate (v4) -- (4,-1.5) coordinate (v6) -- (6,-2.5) coordinate (v3) -- (8,-1.5) -- (9,0.5) coordinate (v7) -- (8,2.5) coordinate (v5) -- (v1) -- (v3);
\draw [blue] (v1) -- (v6);
\draw [blue] (v5) -- (v3) -- (v7);
\draw [red] (3.5,-0.5) .. controls (4,-0.15) and (5.3,2.3) .. (5,3);
\draw [blue] (v4) -- (v1);
\node [blue] at (6.2,1.15) {1};
\node [blue] at (4.85,3.15) {2};
\node [blue] at (3.35,-0.65) {2};
\node [blue] at (3.35,1.65) {3};
\node [blue] at (4.85,-2.15) {3};
\node [blue] at (7.15,3.15) {4};
\node [blue] at (8.65,-0.65) {4};
\node [blue] at (8.65,1.65) {5};
\node [blue] at (7.15,-2.15) {5};
\node [blue] at (4.1,1.25) {6};
\node [blue] at (5.05,0.35) {7};
\node [blue] at (7,0.5) {8};
\node [blue] at (7.85,-0.4) {9};
\node [fill, circle, inner sep=1.3pt] at (6,3.5) {};
\node [fill, circle, inner sep=1.3pt] at (4,2.5) {};
\node [fill, circle, inner sep=1.3pt] at (3,0.5) {};
\node [fill, circle, inner sep=1.3pt] at (4,-1.5) {};
\node [fill, circle, inner sep=1.3pt] at (6,-2.5) {};
\node [fill, circle, inner sep=1.3pt] at (8,-1.5) {};
\node [fill, circle, inner sep=1.3pt] at (9,0.5) {};
\node [fill, circle, inner sep=1.3pt] at (8,2.5) {};
\node [red] at (4,0.5) {$C$};

\draw [squigarrow, thick](0.5,0.5) -- (2.5,0.5);
\node at (1.5,0.9) {$T_C$};

\draw [blue] (-3,3.5) coordinate (v1) -- (-5,2.5) coordinate (v2) -- (-6,0.5) coordinate (v4) -- (-5,-1.5) coordinate (v6) -- (-3,-2.5) coordinate (v3) -- (-1,-1.5) -- (0,0.5) coordinate (v7) -- (-1,2.5) coordinate (v5) -- (v1) -- (v3);
\draw [blue] (v1) -- (v6);
\draw [blue] (v5) -- (v3) -- (v7);
\draw [red] (-5.5,-0.5) .. controls (-5,-0.15) and (-3.7,2.3) .. (-4,3);
\node [blue] at (-2.8,1.15) {1};
\node [blue] at (-4.15,3.15) {6};
\node [blue] at (-5.65,-0.65) {6};
\node [blue] at (-5.65,1.65) {3};
\node [blue] at (-4.15,-2.15) {3};
\node [blue] at (-1.85,3.15) {4};
\node [blue] at (-0.35,-0.65) {4};
\node [blue] at (-0.35,1.65) {5};
\node [blue] at (-1.85,-2.15) {5};
\node [blue] at (-5.15,0.85) {2};
\node [blue] at (-3.95,0.35) {7};
\node [blue] at (-2,0.5) {8};
\node [blue] at (-1.15,-0.4) {9};
\node [fill, circle, inner sep=1.3pt] at (-3,3.5) {};
\node [fill, circle, inner sep=1.3pt] (v8) at (-5,2.5) {};
\node [fill, circle, inner sep=1.3pt] at (-6,0.5) {};
\node [fill, circle, inner sep=1.3pt] (v9) at (-5,-1.5) {};
\node [fill, circle, inner sep=1.3pt] at (-3,-2.5) {};
\node [fill, circle, inner sep=1.3pt] at (-1,-1.5) {};
\node [fill, circle, inner sep=1.3pt] at (0,0.5) {};
\node [fill, circle, inner sep=1.3pt] at (-1,2.5) {};
\draw [blue](v8) -- (v9);
\node [red] at (-4.35,2.05) {$C$};
\node [orange] at (6.75,2) {$\beta_2$};
\draw [orange] (7,3) .. controls (6.9,1.85) and (7.5,0.25) .. (8.5,-0.5);
\draw [green] (8.5,1.5) .. controls (7.5,0.75) and (6.9,-0.85) .. (7,-2);
\node [green] at (7.4,-0.7) {$\beta_1$};
\end{tikzpicture}
    \caption{Dehn twist along $C$.\CORRECTED}
    \label{fig:Dehn_genus_2}
\end{figure}

Then the following path $\gamma$ is a representation path for $T_C$:
\begin{align*}
\gamma: (\tri, \ell) \overbar{6} (\tri', \ell') \overbar{(2\ 6)} T_C^{-1}(\tri, \ell).
\end{align*}
By \cref{lem:Dehn_SS}, this path is sign-stable on $\Omega_C$ with stable sign $(-)$ and the stretch factor $1$. 
\end{ex}

\paragraph{\textbf{Dependence of the sign stability on representation paths}}
As an aside, we give a counterexample to the following conjecture in \cite{IK19} which asserts that the sign stability is an intrinsic property of a mutation loop:

\begin{conj}[{\cite[Conjecture 1.3]{IK19}}]
Let $\gamma_i: (\tri_i, \ell_i) \to \phi^{-1}(\tri_i, \ell_i)$ $(i=1,2)$ be two paths which represent the same mapping class $\phi \in MC(\Sigma)$.
Then $\gamma_1$ is sign-stable if and only if $\gamma_2$ is.
\end{conj}
Here is a counterexample.
Let us consider the Dehn twist $T_C$ as in \cref{ex:Dehn_SS_2}.
It has the following representation path:
\begin{align*}
    \gamma': (\tri, \ell) \xrightarrow{6, 9, 9} (\tri', \ell') \overbar{(2\ 6)} T_C^{-1}(\tri, \ell).
\end{align*}
While the path $\gamma$ is sign-stable on $\Omega_C$, the path $\gamma'$ is not.
Indeed, its sign at $C \in \Omega_C$ is $(+,0,0)$ and thus not strict.
This is a quite general phenomenon: it can occur when the domain $\Omega = \Omega_C$ of stability contains a non-$\X$-filling fixed point (recall \cref{def:X-filling}).
What is worse, it has different strict signs at $L_1 := (G_C \sqcup G_{\beta_1}, \delta_C + \delta_{\beta_1})$ and $L_2 := (G_C \sqcup G_{\delta_2}, \delta_C + \delta_{\beta_2})$, where $\beta_i$ are closed curves shown in \cref{fig:Dehn_genus_2} and $G_{\beta_i}$ are the closed geodesics which are homotopic to $\beta_i$, respectively.
Observe that $L_1, L_2 \in \Omega_C$, and the signs of $\gamma'$ at these points are $(-, -, -)$ and $(-, +, +)$,\CORRECTED\ respectively.

\section{Sign stability of pseudo-Anosov mapping classes: empty boundary case}\label{sec:pA_SS_NS}

Let $X$ be a compact Hausdorff space and $f:X \to X$ be a homeomorphism.
We say that the discrete dynamical system $(X,f)$ has \emph{North-South dynamics} (\emph{NS dynamics} for short) if there exist two distinct fixed points $x^+_f, x^-_f \in X$ of $f$ such that for any point $x \in X$, the followings hold:
\begin{itemize}
    \item if $x \neq x^-_f$, then $\lim_{n\to \infty} f^n(x) = x^+_f$;
    \item if $x \neq x^+_f$, then $\lim_{n\to \infty} f^{-n}(x) = x^-_f$.
\end{itemize}
The point $x^+_f$ (resp. $x^-_f$) is called the attracting (resp. repelling) point of $f$. 


In this section, we assume that $\Sigma$ is a punctured surface (\emph{i.e.}, it has no boundary components) and fix a marked hyperbolic surface $(F,f)$ of type $\Sigma$.
For an $\bR_{>0}$-invariant subset $\Omega \subset \eML(F)$, let
\begin{align*}
    \bP \Omega := (\Omega \setminus \{0\}) / \bR_{>0}.
\end{align*}
A point of the quotient space $\bP\Omega$ is sometimes identified with a ray in $\Omega$.
Our aim in this section is to prove the following:

\begin{thm}\label{t:pA_SS_NS}
Let $\Sigma$ be a punctured surface.
For a mapping class $\phi \in MC(\Sigma)$, the following conditions are equivalent:
\begin{enumerate}
    \item $\phi$ is generic pseudo-Anosov.
    \item $\phi$ is uniformly sign-stable.
    \item $\phi$ has NS dynamics on $\bP \eML(F)$ and the attracting and repelling points are $\cX$-filling.
\end{enumerate}
In this case, the cluster stretch factor $\lambda_\phi^\tri$ coincides with the pA stretch factor of $\phi$ for any ideal triangulation $\tri$.
\end{thm}

Recall that $\phi$ is generic pA if one of (or equivalently, both of) its pA pair is generic.

\subsection{Sign stability and the pseudo-Anosov property}\label{sec:(1)<=>(2)}
We prove the equivalence of (1) and (2) in \cref{t:pA_SS_NS}.
First we recall the following classical theorem due to W. Thurston:

\begin{thm}[\cite{Th88}]\label{thm:NS on U}
Let $\Sigma$ be a punctured surface and $\phi \in MC(\Sigma)$ a mapping class.
Then the mapping class $\phi$ is pA if and only if the action of $\phi$ on $\bP \eML_0(F)$ has NS dynamics.
Moreover, the attracting and repelling fixed points are given by the pA pair of $\phi$.
\end{thm}

We call a mapping class \emph{pure} if it fixes each puncture. Let $PMC(\Sigma)$ denote the normal subgroup consisting of the pure mapping classes, which fits into the following exact sequence:
\begin{align*}
    1 \to PMC(\Sigma) \to MC(\Sigma) \xrightarrow{\sigma_\Sigma} \fS_P \to 1.
\end{align*}
The following lemma is a direct consequence of the Ivanov's work \cite{Iva}.

\begin{lem}[Proved in \cref{subsec:proof_1}]\label{lem:NS on X for pure}

If a pA mapping class $\phi$ is pure, then the action of $\phi$ on $\bP \eML(F)$ has NS dynamics, which is an extension of that on $\bP \ML_0(F)$.
\end{lem}

\begin{proof}[Proof of $(1) \Longleftrightarrow (2)$ in \cref{t:pA_SS_NS}]
Let $\phi \in MC(\Sigma)$ be a generic pA mapping class and take a representation path $\gamma: (\tri, \ell) \to \phi^{-1}(\tri, \ell)$ in $\bTri(\Sigma)$.
By \cref{thm:NS on U} the mapping class $\phi$ has NS dynamics on $\bP \ML_0(F)$, whose attracting and repelling points are written as $[L_\phi^+]$ and $[L_\phi^-]$, respectively.
By \cref{p:shear_non-vanishing} (1), the sign $\boldsymbol{\epsilon}_\gamma(L_\phi^+) =: \boldsymbol{\epsilon}^+_\gamma$ of the point $L_\phi^+$ is strict.
Therefore, the cone
\[\cC^{\boldsymbol{\epsilon}^+_\gamma}_{\tri} := \oline{\{ L \in \eML(F) \mid \boldsymbol{\epsilon}_\gamma(L) = \boldsymbol{\epsilon}^+_\gamma \}}\]
has the maximal dimension and $L^+_\phi \in \interior \cC^{\boldsymbol{\epsilon}^+_\gamma}_{\tri}$. Moreover, since $[L^+_\phi]$ is a fixed point of $\phi$, 
\vspace{-2mm}
\begin{align}\label{eq:sign_power_phi}
\boldsymbol{\epsilon}_{\gamma^r}(L^+_\phi) = (\overbrace{\boldsymbol{\epsilon}^+_\gamma, \dots, \boldsymbol{\epsilon}^+_\gamma}^{r}) =: \boldsymbol{\epsilon}_\gamma^{+,r} 
\end{align}
for any $r \in \bZ_{>0}$.
Here, $\gamma^r$ denotes the path from $(\tri, \ell)$, whose labeles is the $r$ times iteration of the labeles of $\gamma$.
Hence $L^+_\phi \in \interior \cC^{\boldsymbol{\epsilon}_\gamma^{+,r}}_{\tri} \subset \interior \cC^{\boldsymbol{\epsilon}^+_\gamma}_{\tri}$.

Take $r \in \bZ_{>0}$ so that the mapping class $\phi^r$ is pure. 
It has NS dynamics on $\bP\eML(F)$ and its attracting and repelling points are $[L_\phi^+]$ and $[L_\phi^-]$ respectively by \cref{lem:NS on X for pure}.
Hence for any $L \in \eML(F) \setminus [L^-_\phi]$, there exists a positive integer $m_0 \in \bZ$ such that $\phi^{mr}(\hL) \in \interior \cC^{\boldsymbol{\epsilon}^{+,r}_\gamma}_{\tri}$
for all integer $m \geq m_0$.
By \eqref{eq:sign_power_phi}, it implies that $\phi^n(L) \in \interior \cC^{\boldsymbol{\epsilon}^+_\gamma}_{\tri}$ for all $n \geq m_0r$. 
It is nothing but the sign stability of $\phi$ with the stable sign $\boldsymbol{\epsilon}^{+}_\gamma$.

Conversely, assume that $\phi$ is not generic pA.
Then by Nielsen--Thurston classification (\cref{t:NTclassification}), it is either periodic, reducible, or non-generic pA.
In the periodic case, $\phi^r$ is the identity for some integer $r >0$ so cannot be uniformly sign-stable. 

In the reducible case, $\phi$ fixes a lamination whose support consists of closed curves $L \in \ML_0(F)$.
Then \cref{p:shear_non-vanishing} (3) tells us that there exists an ideal triangulation $\tri$ and an ideal arc $\alpha \in \tri$ such that $\sfx^{\tri}_{\alpha}(L) =0$.
Then, for a representation path $\gamma: (\tri, \ell) \to \phi^{-1}(\tri, \ell)$, the path 
\begin{align*}
    \gamma': (\tri,\ell) \overbar{\ell^{-1}(\alpha)} (\tri',\ell) \overbar{\ell^{-1}(\alpha)} (\tri,\ell) \xrightarrow{\ \gamma\ } \phi^{-1}(\tri, \ell)
\end{align*}
is clearly a representation path again.
Then $\boldsymbol{\epsilon}_{\gamma'}(L)$ is not a strict sign at the first two steps, and hence $\phi$ cannot be uniformly sign-stable. 
\end{proof}

\subsection{NS dynamics on the space of \texorpdfstring{$\cX$}{X}-laminations and the pseudo-Anosovness}\label{sec:(1)<=>(3)}
Our aim here is to prove the equivalence of (1) and (3) in \cref{t:pA_SS_NS}.
From \cref{thm:NS on U}, 
the implication (3) $\Too$ (1) is clear since the action of a mapping class on $\eML(F)$ restricts to that on $\ML_0(F)$.
It remains to prove the implication (1) $\Too$ (3). Although this statement seems to be well-known to specialists, we give a proof here based on the sign stability of a pA mapping class (\emph{i.e.}, we are actually going to show $(1)\,\&\,(2) \Too (3)$). 

First, we discuss a suitable block decomposition of the presentation matrix of a mapping class $\phi$. In order to obtain a natural one from the geometry, the following observations are useful.  

\smallskip 

\paragraph{\textbf{(1) Tangent action of $MC(\Sigma)$ on $\eML(F)$}}
The shear coordinates endows the space $\eML(F)$ with a \emph{tangential structure} in the sense of \cite{Bon}. In particular, the tangent space $T_L \eML(F)$ is defined to be the cone of one-sided directional derivatives for each $L \in \eML(F)$. 
Since a mapping class $\phi \in MC(\Sigma)$ acts on $\eML(F)$ preserving this tangential structure, it induces the tangent map
\begin{align*}
    (d\phi)_{L}: T_{L} \eML(F) \to T_{\phi(L)} \eML(F)
\end{align*}
for any $L \in \eML(F)$. 

Given a labeled triangulation $(\tri,\ell)$ of $\Sigma$, the shear coordinate system $\boldsymbol{\sfx}^{(\tri,\ell)}$ induce a tangential isomorphism
\begin{align*}
    T_L \eML(F) \xrightarrow{\sim} \bR^{I}.
\end{align*}
Via this isomorphism, we can import the linear structure of $\bR^{I}$ into $T_L \eML(F)$. The resulting vector space will be denoted by $T^{(\tri,\ell)}_L \eML(F)$. 
Moreover if $L$ is chosen so that the coordinate expression of $\phi$ with respect to the labeled triangulation $(\tri, \ell)$ (which is a PL map with finitely many half-planes of discontinuity) is differentiable at $L$, then the tangential map $(d\phi)_L$ preserves this linear structure. Thus, we obtain a linear map
\begin{align*}
    (d\phi)^{(\tri,\ell)}_L: T^{(\tri,\ell)}_L \eML(F) \to T^{(\tri,\ell)}_{\phi(L)} \eML(F).
\end{align*}
If $\gamma: (\tri, \ell) \to \phi^{-1}(\tri, \ell)$ is a representation path of $\phi$ and $L \in \interior \cC^{\bep}_\gamma$ with a strict sign $\bep$, then $E_\gamma^{\bep}$ gives the presentation matrix of $(d\phi)_L^{(\tri,\ell)}$. 

\smallskip

\paragraph{\textbf{(2) Block decomposition of the tangent action}}
Now we fix $L_0 \in \interior \cC^{\bep}_\gamma \cap \ML_0(F)$ with a strict sign $\bep$. 
Since the map $\theta^\mathrm{ML}: \eML(F) \to \bR^P$ is a linear map in any shear coordinate system by \cref{lem:ML_Casimir}, 
we have an exact sequence
\begin{align*}
    0 \to T_{L_0}^{(\tri,\ell)} \ML_0(F) \to T_{L_0}^{(\tri,\ell)} \eML(F) \xrightarrow{(d \theta^\mathrm{ML})_{L_0}} \bR^P \to 0
\end{align*}
of vector spaces by differentiating the exact sequence
\eqref{eq:lam_exact_seq}, 
and a similar sequence for the tangent spaces at $\phi (L_0)$. 

Now we assume $\bep_\gamma(\phi(L_0))$ is also strict.
From the $MC(\Sigma)$-equivariance of \eqref{eq:lam_exact_seq}, we see that these sequences are $(d\phi)_{L_0}^{(\tri,\ell)}$-equivariant, where the action on $\bR^P$ is given by $\sigma_\Sigma(\phi)$.
Therefore, by choosing linear sections $s_{L_0}: \bR^P \to T^{(\tri,\ell)}_{L_0} \eML(F)$ and $s_{\phi(L_0)}: \bR^P \to T^{(\tri,\ell)}_{\phi(L_0)} \eML(F)$, we get a block-decomposition of the form
\begin{align}\label{eq:block-decomp}
    (d\phi)_{L_0}^{(\tri,\ell)} = 
    \begin{pmatrix}
    (d\phi^u)_{L_0}^{(\tri,\ell)} & \ast \\
    0 & \sigma_\Sigma(\phi)
    \end{pmatrix}
\end{align}
with respect to the direct sum decompositions $T_{L_0}^{(\tri,\ell)}\eML(F) = T_{L_0}^{(\tri,\ell)} \ML_0(F) \oplus s_{L_0}(\bR^P)$ and $T_{\phi(L_0)}^{(\tri,\ell)} \eML(F) = T_{\phi(L_0)}^{(\tri,\ell)} \ML_0(F) \oplus s_{\phi(L_0)}(\bR^P)$.
Here $\phi^u$ is the restriction of the action of $\phi$ to $\ML_0(F)$. 

\bigskip

With these preparations, let us consider a generic pA mapping class $\phi \in MC(\Sigma)$. 
From the implication (1) $\Too$ (2) in \cref{t:pA_SS_NS}, any representation path $\gamma: (\tri, \ell) \to \phi^{-1}(\tri, \ell)$ of $\phi$ is sign-stable on $\eML_\bQ(F)$. 
Then for each point $L \in \eML(F)$, there exists a positive integer $n_0$ such that $\phi^n(L) \in \interior \cC_\tri^{\bep^+_\gamma}$ for all $n \geq n_0$.
Thus we have 
\begin{align*}
    \bsfx^{(\tri,\ell)}(\phi^n(\phi^{n_0}(L))) = (E_\gamma^{\epsilon_\gamma^+})^n\, \bsfx^{(\tri,\ell)}(\phi^{n_0}(L))
\end{align*}
for all $n \geq 0$.
In other words, the action of $\phi$ is a linear dynamical system in the stable range $n \geq n_0$. 
By \cref{t:pA_SS_NS} we know that the linear action of $\phi^u$ on the cone $\cC^{\bep^+_\gamma}_\tri \cap \ML_0(F)$ has North dynamics with attracting point $L^+_\phi$.
Since $L^+_\phi$ satisfies the assumptions of the above, we can take the decomposition \eqref{eq:block-decomp} for this point.
Interpreting these in terms of the corresponding presentation matrix, our assertion is reduced to the following lemma, which is just a matter of linear algebra:

\begin{lem}[Proved in \cref{subsec:proof_2}]\label{lem:asympt to U}
Let $N, N', h$ be positive integers satisfying $N = N' +h$ and
let $E \in \mathrm{GL}(N; \bR)$ be a matrix of the form
\[
E = 
\begin{pmatrix}
A & B\\
0 & D
\end{pmatrix}, \quad A \in \mathrm{GL}(N'; \bR),\quad B \in \mathrm{Mat}(h \times N';\bR), \quad D \in \mathrm{GL}(h; \bR).
\]
Assume that
\begin{itemize}
    \item[(a)] the matrix $D$ has finite order $s$;
    \item[(b)] there exists a cone $\cC \subset \bR^N$ which is invariant under the action of $E$; 
    \item[(c)] there exists a unique attracting point $[u_+] \in \bP(\cC \cap \bR^{N'})$ for the action of $A$ on $\bP(\cC \cap \bR^{N'})$, where $u_+ \in \bR^{N'}$ is a unit eigenvector of $A$ with eigenvalue $\lambda>1$. 
\end{itemize}
Then for any point $[u] \in \bP\cC$, we have $E^n([u]) \to [u_+]$ as $n \to \infty$. 
\end{lem}

\begin{proof}[Proof of $(1) \Too (3)$ in \cref{t:pA_SS_NS}]
Let $N:=|I(\Sigma)|$, $N':=\dim \ML_0(F) = |I(\Sigma)| - |P|$, $\cC:=\cC^{\bep^+_\gamma}_\tri$, and $E$ the block-decomposed matrix \eqref{eq:block-decomp} (with a basis compatible with the decomposition).
The assumptions are satisfied by \cref{thm:NS on U}, where $[u_+]$ is the coordinate vector of $[(G_+,\mu_+)]$ and the corresponding eigenvalue is given by the pA stretch factor $\lambda >1$.
Hence $\phi$ has a unique attracting point in $\bP \eML(F)$. 

By considering the action of $\phi^{-1}$ near the class $[(G_-,\mu_-)]$, the same argument can be applied and we see that $\phi$ has a unique repelling point.
From the definition of the genericity of pA and \cref{p:shear_non-vanishing} (1), its pA pair consists of $\cX$-filling measured geodesic laminations.
Thus the assertion is proved.
\end{proof}

\begin{figure}[h]
    \centering
    \begin{tikzpicture}[auto]
    \draw[very thick, red] (-1.3,0) .. controls (1,0) and (1,-0.08) .. (1.5,0.5);
    \draw[very thick, red] (-1.3,0) .. controls (1,0) and (1,0.08) .. (1.5,-0.5);
    \node at (0.5,0.25) {$v$};
    \node at (1.75,0.75) {$e_1$};
    \node at (1.75,-0.75) {$e_2$};
    \node at (-1,0.25) {$e_0$};
    \node at (5.5,0) {$\nu(e_0) = \nu(e_1) + \nu(e_2)$};
    \end{tikzpicture}
    \caption{Switch condition.}
    \label{fig:switch_cond}
\end{figure}

\begin{proof}[Proof of coincidence of the two stretch factors in \cref{t:pA_SS_NS}]
We verify the coincidence of the cluster stretch factor $\lambda_\phi^\tri$ (\cref{d:cluster stretch factor}) and the pA stretch factor $\lambda_\phi$ of the generic pA mapping class $\phi$.
By the block decomposition \eqref{eq:block-decomp}, the pA stretch factor $\lambda_\phi$ is also an eigenvalue the stable presentation matrix of $\phi$ since the coordinate vector of $(G_+, \mu_+)$ is its eigenvector.
Thus, we have $\lambda_\phi \leq \lambda_\phi^\tri$ since the cluster stretch factor $\lambda_\phi^\tri$ is a spectral radius of the stable presentation matrix.

We are going to show the reversed inequality based on the fact that the latter is equal to the spectral radius of the \emph{transition matrix} of the \emph{invariant train track} $\tau_\phi$ of $\phi$ (see \cite{BH95}). 

Recall that a train track $\tau$ is a certain trivalent graph on $\Sigma$ equipped with a smoothing structure on each vertex, see \cref{fig:switch_cond}. For a detail, we refer the reader to \cite{PH}, the appendix of \cite{IK2}, or \cite{Kan23}.
Let $V(\tau) \subset \bR^{\{\text{\scriptsize edges of $\tau$}\}}_{\geq 0}$ denote the subset defined by the \emph{switch conditions} shown in \cref{fig:switch_cond} at each vertex, called the cone of measures on $\tau$. It is known that $V(\tau)$ is naturally identified with a subcone of $\ML_0(F)$.

A train track $\tau$ is said to be \emph{$\phi$-invariant} if $\phi(\tau)$ can be collapsed to $\tau$ preserving the smoothing structure at vertices. It is known that any pseudo-Anosov mapping class $\phi$ has an invariant train track $\tau_\phi$ such that $L_\phi^+ \in V(\tau_\phi)$ and $L_\phi^+ \in V(\phi(\tau_\phi))
$. 

Since $\phi$ is generic, the invariant track $\tau_\phi$ is also \emph{complete}, namely, there is no train track $\tau'$ containing $\tau_\phi$ as a proper subgraph such that $V(\tau) \cap \bR_{>0}^{\text{edges of $\tau$}} \neq \emptyset$. 
Indeed, we cannot add a new edge to $\tau_\phi$ so that the result is still a train track, since the complementary regions of $\tau_\phi$ are either a disk with three cusps or a once-punctured disk with one cusp by the genericity of $\phi$ (\emph{cf}. the construction of the Markov partition from the invariant track \cite{BH95}). As a result, the cone $V(\tau_\phi)$ is full-dimensional in $\ML_0(F)$.

The transition matrix is a certain linear isomorphism $T_\phi: \bR^{\{\text{\scriptsize edges of $\phi(\tau_\phi)$}\}} \to \bR^{\{\text{\scriptsize edges of $\tau_\phi$}\}}$ describing the collapsing 
such that $T_\phi(V(\phi(\tau_\phi))) \subset V(\tau_\phi)$.

Consider the exact sequence
\begin{align*}
    0 \to T_{L_\phi^+} V(\phi(\tau_\phi)) \to T_{L_\phi^+} \bR^{\{\mbox{\scriptsize edges of $\tau_\phi$}\}} \to T_{L_\phi^+} \bR^{\{\mbox{\scriptsize edges of $\tau_\phi$}\}} / T_{L_\phi^+} V(\tau_\phi) \to 0.
\end{align*}
Taking a section of $T_{L_\phi^+} \bR^{\{\text{\scriptsize edges of $\tau_\phi$}\}} / T_{L_\phi^+} V(\tau_\phi)$, we obtain a block decompoistion of $T_\phi$.
The submatrix of $T_\phi$ in the direction of $T_{L_\phi^+} V(\phi(\tau_\phi)) \cong T_{L_\phi^+} \ML_0(F)$ coincides with $(d\phi^u)^{(\tri,\ell)}_{L_\phi^+}$ by taking a suitable basis \cite[Section 3.4]{Kan23}.
Thus, we have $\lambda_\phi^\tri \leq \lambda_\phi$.
\end{proof}



The proof of \cref{t:pA_SS_NS} is completed.

\begin{ex}\label{ex:4_sph}
Here is an example of a generic pA mapping class.
Consider a fourth-punctured sphere ($g=0$, $h=4$) with a labeled ideal triangulation $(\tri,\ell)$ shown in \cref{fig:4_sph}.
\begin{figure}[h]
    \centering
    \begin{tikzpicture}
    \begin{scope}[scale=0.9]
    \draw (-3.5,0) ellipse (2.5 and 2.5);
    \draw[blue] (-1,0) arc [start angle=0, end angle=-180, x radius=2.5cm, y radius=1cm];
    \draw[blue, dashed] (-1,0) arc [start angle=0, end angle=180, x radius=2.5cm, y radius=1cm];
    \draw[blue] (-3.5,2.5) arc [start angle=90, end angle=270, x radius=1cm, y radius=2.5cm];
    \draw[blue, dashed] (-3.5,2.5) arc [start angle=90, end angle=-90, x radius=1cm, y radius=2.5cm];
    \node [fill, circle, inner sep=1.5pt] at (-4.4,-0.95) {};
    \node [fill, circle, inner sep=1.5pt, gray] at (-2.55,0.95) {};
    \node [fill, circle, inner sep=1.5pt, gray] at (-5.5,0.6) {};
    \node [fill, circle, inner sep=1.5pt] at (-1.5,-0.6) {};
    \node [blue] at (-4.4,-1.9) {\small 1};
    \node [blue] at (-5.3,-0.45) {\small 2};
    \node [blue!50] at (-3.5,1.2) {\small 3};
    \node [blue] at (-4.4,1.9) {\small 4};
    \node [blue] at (-3.35,-1.25) {\small 5};
    \node [blue!50] at (-1.4,0.85) {\small 6};
    \node at (-4.1,-0.6) {$p$};
    \end{scope}
    \begin{scope}[scale=1.1]
    \draw (0.05,0.05) -- (0.95,0.05);
    \draw (0.05,-0.05) -- (0.95,-0.05);
    \draw [blue](4.5,1.5) coordinate (v1) {} -- (4.5,-1.5) coordinate (v2) {} -- (3,0) -- (v1) -- (6,0) -- (v2);
    \draw [blue](4.5,-1.5) .. controls (3.5,-2) and (2,-1) .. (2,0) .. controls (2,1) and (3.5,2) .. (4.5,1.5);
    \node [fill, circle, inner sep=1.5pt] at (4.5,1.5) {};
    \node [fill, circle, inner sep=1.5pt] at (6,0) {};
    \node [fill, circle, inner sep=1.5pt] at (4.5,-1.5) {};
    \node [fill, circle, inner sep=1.5pt] at (3,0) {};
    \node [blue] at (2.45,1.25) {\small 1};
    \node [blue] at (3.6,-0.9) {\small 2};
    \node [blue] at (3.6,0.9) {\small 3};
    \node [blue] at (4.3,0) {\small 4};
    \node [blue] at (5.4,-0.9) {\small 5};
    \node [blue] at (5.4,0.9) {\small 6};
    \node at (4.85,-1.65) {$p$};
    \end{scope}
    \end{tikzpicture}
    \caption{A fourth-punctured sphere with a labeled triangulation. A planar picture is shown in the right, regarding $\bP^2=\bR^2 \cup \{\infty\}.$}
    \label{fig:4_sph}
\end{figure}
Consider the mapping class $\phi:=h_2 \circ h_5^{-1}$, where $h_5^{-1}$ is the left\CORRECTED\ half-twist along the arc $\ell(5)$ and $h_2$ is the right\CORRECTED\ half-twist along $\ell(2)$. 
It is an example of a pA mapping class, which realizes the smallest stretch factor among those in $MC(\Sigma)$. See \cite[Section 15.1]{FM}.
We show representation paths of $h_2$ and $h_5^{-1}$ in \cref{fig:pA_4_sph}.
\begin{figure}[h]
    \centering
    \hspace{-1.1cm}
    \begin{tikzpicture}[scale=.91]
\draw [blue](4.5,1.5) coordinate (v1) -- (4.5,-1.5) coordinate (v2) -- (3,0) -- (v1) -- (6,0) -- (v2);
\draw [blue](4.5,-1.5) .. controls (3.5,-2) and (2,-1) .. (2,0) .. controls (2,1) and (3.5,2) .. (4.5,1.5);
\node [fill, circle, inner sep=1.5pt] at (4.5,1.5) {};
\node [fill, circle, inner sep=1.5pt] at (6,0) {};
\node [fill, circle, inner sep=1.5pt] at (4.5,-1.5) {};
\node [fill, circle, inner sep=1.5pt] at (3,0) {};
\node [blue] at (2.45,1.25) {\scriptsize 1};
\node [blue] at (3.6,-0.9) {\scriptsize 2};
\node [blue] at (3.6,0.9) {\scriptsize 3};
\node [blue] at (4.3,0) {\scriptsize 4};
\node [blue] at (5.4,-0.9) {\scriptsize 5};
\node [blue] at (5.4,0.9) {\scriptsize 6};

\node [fill, circle, inner sep=1.5pt] (v3) at (11,1.5) {};
\node [fill, circle, inner sep=1.5pt] (v6) at (12.5,0) {};
\node [fill, circle, inner sep=1.5pt] (v5) at (11,-1.5) {};
\node [fill, circle, inner sep=1.5pt] (v4) at (9.5,0) {};
\node [blue] at (9.95,0.8) {\scriptsize 1};
\node [blue] at (10.3,-0.5) {\scriptsize 2};
\node [blue] at (10.75,0) {\scriptsize 3};
\node [blue] at (12,-0.95) {\scriptsize 4};
\node [blue] at (12.5,-2) {\scriptsize 5};
\node [blue] at (11.9,0.9) {\scriptsize 6};
\draw [blue](9.5,0) .. controls (9,-1) and (9.5,-2.5) .. (11,-2.5) .. controls (12.5,-2.5) and (13,-1) .. (12.5,0);
\draw [blue](v3) -- (v4) -- (v5) -- (v3) -- (v6);
\draw [blue](9.5,0) .. controls (10.5,-4) and (13.1,-1.6) .. (11,1.5);

\node [fill, circle, inner sep=1.5pt] (v3) at (8,-4) {};
\node [fill, circle, inner sep=1.5pt] (v6) at (9.5,-5.5) {};
\node [fill, circle, inner sep=1.5pt] (v5) at (8,-7) {};
\node [fill, circle, inner sep=1.5pt] (v4) at (6.5,-5.5) {};
\node [blue] at (6.95,-4.7) {\scriptsize 3};
\node [blue] at (7.3,-6) {\scriptsize 2};
\node [blue] at (7.75,-5.5) {\scriptsize 4};
\node [blue] at (9,-6.45) {\scriptsize 5};
\node [blue] at (9.5,-7.5) {\scriptsize 1};
\node [blue] at (8.9,-4.6) {\scriptsize 6};
\draw [blue](6.5,-5.5) .. controls (6,-6.5) and (6.5,-8) .. (8,-8) .. controls (9.5,-8) and (10,-6.5) .. (9.5,-5.5);
\draw [blue](v3) -- (v4) -- (v5) -- (v3) -- (v6);
\draw [blue](6.5,-5.5) .. controls (7.5,-9.5) and (10.1,-7.1) .. (8,-4);

\node [fill, circle, inner sep=1.5pt] (v8) at (-2.5,1.5) {};
\node [fill, circle, inner sep=1.5pt] (v10) at (-1,0) {};
\node [fill, circle, inner sep=1.5pt] (v9) at (-2.5,-1.5) {};
\node [fill, circle, inner sep=1.5pt] (v7) at (-4,0) {};
\node [blue] at (-1,0.6) {\scriptsize 1};
\node [blue] at (-3.75,-1.5) {\scriptsize 2};
\node [blue] at (-3.4,0.9) {\scriptsize 3};
\node [blue] at (-3,-1) {\scriptsize 4};
\node [blue] at (-1.6,-0.9) {\scriptsize 5};
\node [blue] at (-2.25,0.25) {\scriptsize 6};
\draw [blue](v7) -- (v8) -- (v9) -- (v10);
\draw [blue](-4,0) .. controls (-4.4,-1) and (-4,-2.4) .. (-2.5,-2.4) .. controls (-1,-2.4) and (-0.6,-1) .. (-1,0);
\draw [blue](-2.5,1.5) .. controls (-4.5,-2) and (-2,-3.5) .. (-1,0);
\draw [blue](-2.5,1.5) .. controls (-5.5,1.5) and (-5,-2.75) .. (-2.5,-2.75) .. controls (-1.1,-2.75) and (-0.45,-1.95) .. (-0.45,-1) .. controls (-0.45,0.5) and (-1.5,1.5) .. (-2.5,-1.5);

\node [fill, circle, inner sep=1.5pt] (v8) at (0.75,-4) {};
\node [fill, circle, inner sep=1.5pt] (v10) at (2.25,-5.5) {};
\node [fill, circle, inner sep=1.5pt] (v9) at (0.75,-7) {};
\node [fill, circle, inner sep=1.5pt] (v7) at (-0.75,-5.5) {};
\node [blue] at (1.7,-4.7) {\scriptsize 6};
\node [blue] at (-0.5,-7) {\scriptsize 1};
\node [blue] at (-0.15,-4.6) {\scriptsize 3};
\node [blue] at (0.25,-6.5) {\scriptsize 2};
\node [blue] at (1.65,-6.4) {\scriptsize 5};
\node [blue] at (1,-5.5) {\scriptsize 4};
\draw [blue](v7) -- (v8) -- (v9) -- (v10);
\draw [blue](-0.75,-5.5) .. controls (-1.15,-6.5) and (-0.75,-7.9) .. (0.75,-7.9) .. controls (2.25,-7.9) and (2.65,-6.5) .. (2.25,-5.5);
\draw [blue](0.75,-4) node (v11) {} .. controls (-1.25,-7.5) and (1.25,-9) .. (2.25,-5.5);
\draw [blue](v11);
\draw [blue](0.75,-4) -- (2.25,-5.5);

\draw [->, thick](5.25,-2.5) -- (6.75,-4);
\draw [->, thick](3.75,-2.5) -- (2.25,-4);
\draw [->, thick](-0.75,-4) -- (-1.75,-3);
\draw [->, thick](9.25,-4) -- (10.25,-3);
\node at (2.25,-3) {$(1, 2)$};
\node at (6.75,-3) {$(1, 5)$};
\node at (-2.25,-3.65) {$(1\ 2\ 4\ 6)$};
\node at (10.75,-3.65) {$(1\ 5\ 4\ 3)$};
\draw [squigarrow, thick](8.5,0) -- (7,0);
\draw [squigarrow, thick](0,0) -- (1.5,0);
\node at (0.75,0.3) {$h_5^{-1}$};
\node at (7.75,0.3) {$h_2$};
\end{tikzpicture}
\vspace{-2cm}
    \caption{The representation paths of $h_2$ and $h_5^{-1}$.}
    \label{fig:pA_4_sph}
\end{figure}
The following path is their concatenation:
\begin{align*}
    \gamma: (\tri, \ell) \xrightarrow{1,5} (\tri', \ell') \xrightarrow{(1\ 5\ 4\ 3)} h_2^{-1}(\tri, \ell) \xrightarrow{1,2} (\tri'', \ell'') \xrightarrow{(1\ 2\ 4\ 6)} \phi^{-1}(\tri, \ell).
\end{align*}


We can compute the coordinates of the pA pair $L^\pm_\phi$ as follows:
\begin{enumerate}
    \item Compute the domains of linearity of the PL map $\mu_\gamma: \bR^I \to \bR^I$ and the corresponding presentation matrices.
    \item Compute the eigenvectors of these matrices and find the ones which are contained in the corresponding domains of linearity. (Indeed, \cref{t:pA_SS_NS} (3) guarantees that there exist precisely two such eigenvectors.)
\end{enumerate}
The point $L^+_\phi$ (resp. $L^-_\phi$) corresponds to an eigenvector whose eigenvalue is lager (resp. smaller) than 1.
In this case, there are 16 domains of linearity of $\mu_\gamma$.
By computing the eigenvectors of the corresponding presentation matrices, one obtains the coordinates of $L^\pm_\phi$ as \footnote{In this case, the attracting/repelling points $L_\phi^\pm$ are antipodal to each other although it is a very special situation.}
\[ \bsfx^{(\tri,\ell)}(L_\phi^\pm) = \pm\left(
\frac{-1+\sqrt{5}}{2}, -1, 1, \frac{1-\sqrt{5}}{2}, 1, -1
\right). \]\CORRECTED


\cref{t:pA_SS_NS} tells us that the signs at every orbit in $\eML(F) \setminus \bR_{\geq 0} L_\phi^-$ converge to the stable sign $\boldsymbol{\epsilon}_{\gamma}^\stab:=\boldsymbol{\epsilon}_{\gamma}(L^+_\phi) = (-,+,+,-)$\CORRECTED. For example, 
the convergent behavior of signs at the orbit of the measured geodesic lamination $l^\pm_\tri \in \cC^\pm_\tri \subset \eML(F)$ such that $\bsfx^{(\tri, \ell)}(l^\pm_\tri) = (\pm 1, \dots, \pm 1)$
is shown in \cref{tab:signs_4_sph}.

\begin{table}[h]
    \centering
    \caption{Orbit of signs for $\gamma$.\CORRECTED}
    \vspace{-8pt}
    \begin{tabular}{c||c|c}
    \hline
    $n$ & $\boldsymbol{\epsilon}_\gamma(\phi^n(l^+_\tri))$ & $\boldsymbol{\epsilon}_\gamma(\phi^n(l^-_\tri))$ \\
    \hline
    $1$ & $(+,+,+,+)$ & $(-,-,-,-)$\\
    $2$ & $(+,+,+,-)$ & $(-,+,-,-)$\\
    $3$ & $(-,+,+,-)$ & $(-,+,+,-)$\\
    $4$ & $(-,+,+,-)$ & $(-,+,+,-)$\\
    $5$ & $(-,+,+,-)$ & $(-,+,+,-)$\\
    $\vdots$ & $\vdots$ & $\vdots$
    \end{tabular}
    \label{tab:signs_4_sph}
\end{table}

\end{ex}

\section{Topological and algebraic entropies}\label{subsec:entropy}
In this section, we compare the topological entropy of a mapping class and the algebraic entropies of its natural actions on certain cluster varieties. 
First of all, let us recall definitions and fundamental results on these notions of entropies.

\subsection{Topological entropy}
Given an open cover $\cU$ of a compact topological space $X$, 
let $N(\cU)$ denote the minimal cardinality of its finite subcover.
For open covers $\cU_1, \dots, \cU_n$ of $X$, we define their common refinement by
\[
\bigvee_{i=1}^n \cU_i := \{ U_1 \cap \cdots \cap U_n \mid U_i \in \cU_i,\  i=1,\dots,n\}.
\]
For a continuous map $f:X \to X$ and an open cover $\cU$ of $X$, we define another open cover by
\[
f^{-1}(\cU) := \{ f^{-1}(U) \mid U \in \cU \}.
\]
\begin{defi}[{\cite{AKM65}}]
The topological entropy of a continuous map $f:X \to X$ with respect to an open cover $\cU$ of $X$ is defined to be
\[ \cE^{\mathrm{top}}(f, \cU) := \lim_{n\to \infty} \log N\bigg( \bigvee_{i=0}^{n-1} f^{-i}(\cU)\bigg). \]
Then the topological entropy of $f$ is defined to be
\[ \cE^{\mathrm{top}}_f := \sup_{\cU} \cE^{\rm top} (f, \cU), \]
where $\cU$ runs over the all open covers of $X$.
\end{defi}

For a mapping class $\phi \in MC(\Sigma)$, its topological entropy is the infimum of the topological entropies of its representatives:
\[ \cE^{\mathrm{top}}_\phi := \inf_{f \in \phi}\, \cE^{\mathrm{top}}_f. \]

\begin{thm}\label{thm:top_ent_Dehn_0}
Let $\Sigma$ be a marked surface and let $C \subset \Sigma \setminus P$
be an essential simple closed curve.
Then we have
\[ \cE^{\mathrm{top}}_{T_C} = 0. \]
\end{thm}

\begin{proof}
Although this theorem seems to be well-known, we give an easy proof here since we could not find a reference. 
First, recall the following theorem giving the estimate of the topological entropy from above:

\begin{thm}[\cite{Ito70}]\label{thm:top_ent_above}
Let $M$ be a compact $n$-dimensional Riemannian manifold and $f:M \to M$ be a $C^1$-diffeomorphism.
Then,
\[
\cE^{\mathrm{top}}_f \leq n \log \sup_{x \in M} \big\| df_x \big\|.
\]
Here $\|\cdot \|$ denotes the operator norm with respect to the norms on the tangent spaces $T_x M$ and $T_{f(x)}M$ given by the Riemannian metric.
\end{thm}
In order to use this, we fix an auxiliary Riemannian metric on $\Sigma$.
Take a tubular neighborhood $\cN(C)$ of $C$ and an isometry
\[ \alpha: A \to \cN(C) \subset \Sigma \]
with an annulus $A := \bR/2\pi\bZ \times [0, 1]$ equipped with the standard Euclidean metric.
Let $T: A \to A$ be an orientation-preserving diffeomorphism
defined by $(\theta,t) \mapsto (\theta + b(t), t)$, where $b:[0,1] \to [0,2\pi]$ is a smooth monotonically increasing\CORRECTED\ function such that $b'(0)=b'(1)=0$. Then its tangent map is of the form
\[
dT_{(\theta,t)} =
\begin{pmatrix}
1 & *\\
0 & 1
\end{pmatrix}
\]
for any $(\theta,t) \in A$.
Then the identity-extension of the diffeomorphism $\alpha \circ T \circ \alpha^{-1}$ on $\cN(C)$ represents the Dehn twist $T_C$ along $C$.
Then from \cref{thm:top_ent_above} we get
\begin{align*}
    \cE^{\mathrm{top}}_{T_C} 
    &\leq \cE^{\mathrm{top}}_{\alpha \circ T \circ \alpha^{-1}}\\
    &\leq 2 \log \sup_{x \in \Sigma} \big\| d(\alpha \circ T \circ \alpha^{-1})_x \big\|\\
    &= 2 \log \sup_{x \in \cN(C)} \big\| d(\alpha \circ T \circ \alpha^{-1})_x \big\|\\
    &= 2 \log \sup_{y \in A} \big\| dT_y \big\| = 0
\end{align*}
as desired.
\end{proof}
For the topological entropy of a pA mapping class, we have the following classical result:

\begin{thm}[Thurston, \cite{FLP}]\label{thm:top_entropy}
Let $\Sigma$ be a punctured surface and $\phi \in MC(\Sigma)$ a pA mapping class with pA stretch factor $\lambda >1$. 
Then its topological entropy is given by
\[ \cE^{\mathrm{top}}_\phi = \log \lambda. \]
\end{thm}


\subsection{Algebraic entropy}\label{sec:alg_entropy}
Let us consider the Fock--Goncharov's moduli spaces $\A_{SL_2,\Sigma}$ and $\X_{PGL_2,\Sigma}$ over $\bC$ \cite{FG03}. They parametrize $SL_2(\bC)$- and $PGL_2(\bC)$-local systems on $\Sigma$ with certain decoration data at punctures (flat sections of some associated bundles), and hence give two extensions of the $SL_2(\bC)$-character variety. See \cite{FG03} for details. 
The mapping class group $MC(\Sigma)$ naturally acts on these moduli spaces by the pull-back of the local systems and the decoration data. We are going to compute the \emph{algebraic entropy} \cite{BV99} of this action. 

For a rational function $f(u_1,\dots,u_N)$ over $\bC$ on $N$ variables, write it as $f(u)=g(u)/h(u)$ for two polynomials $g$ and $h$ without common factors. Then the \emph{degree} of $f$, denoted by $\deg f$, is defined to be the maximum of the degrees of the constituent polynomials $g$ and $h$. 
For a homomorphism $\varphi^*:\bC(u_1,\dots,u_N) \to \bC(u_1,\dots,u_N)$ between the field of rational functions on $N$ variables, let $\varphi_i:=\varphi^*(u_i)$ for $i=1,\dots,N$. Since $\bC(u_1,\dots,u_N)$ is the field of rational functions on the algebraic torus $(\bC^\ast)^N$ equipped with coordinate functions $u_1,\dots,u_N$, the homomorphism $\varphi^*$ can be regarded as the pull-back action induced by a rational map $\varphi:(\bC^\ast)^N \to (\bC^\ast)^N$ between algebraic tori. 
We define the degree of $\varphi$ to be the maximum of the degrees  $\deg\varphi_1,\dots,\deg\varphi_N$ of components and denote it by $\deg (\varphi)$. 
\begin{dfn}[\cite{BV99}]\label{d:entropy}
The \emph{algebraic entropy} $\cE^{\mathrm{alg}}_\varphi$ of a rational map $\varphi:(\bC^\ast)^N \to (\bC^\ast)^N$ is defined to be
\begin{align*}
    \cE^{\mathrm{alg}}_\varphi:= \limsup_{n \to \infty} \frac{1}{n}\log(\deg(\varphi^n)).
\end{align*}
\end{dfn}
The algebraic entropy is invariant under a conjugation by a birational map $f$: $\cE^{\mathrm{alg}}_{f^{-1}\varphi f}=\cE^{\mathrm{alg}}_\varphi$.

Each of the moduli spaces $\A_{SL_2,\Sigma}$ and $\X_{PGL_2,\Sigma}$ has a natural \emph{cluster structure}. For example, $\X_{PGL_2,\Sigma}$ admits a birational map $\mathbf{X}_\tri: \X_{PGL_2,\Sigma} \to (\bC^\ast)^\tri$ (\emph{cluster charts}) given by the cross ratio parameters associated with ideal triangulations $\tri$ of $\Sigma$, transition maps among them being expressed as the so-called \emph{cluster transformations}. Indeed, the shear coordinate system $\boldsymbol{\sfx}_\tri$ on $\eML(F)$ (\cref{thm:lam_shear}) is the tropical analogue of $\mathbf{X}_\tri$. The moduli space $\A_{SL_2,\Sigma}$ has a similar structure. 

Given a mapping class $\phi \in MC(\Sigma)$ and an ideal triangulation $\tri$, its coordinate expression $\phi_\tri^x:=\mathbf{X}_\tri^{-1} \circ \phi \circ \mathbf{X}_\tri^{-1}:(\bC^\ast)^\tri \to (\bC^\ast)^\tri$ is a birational map. Its algebraic entropy $\cE_\phi^x:=\cE^{\mathrm{alg}}_{\phi^x_{\tri}}$ is independent of the choice of $\tri$. Let $\cE_\phi^a$ denote the algebraic entropy similarly defined for the action on $\A_{SL_2,\Sigma}$. 

The existence of these cluster structures tells us that these moduli spaces $\cA_{SL_2, \Sigma}$ and $\cX_{PGL_2, \Sigma}$ are birationally equibalent to the cluster $\cA$- and $\cX$-varieties of a mutation class $\bs_\Sigma$ associated with $\Sigma$, respectively.
Also, we can identify $\eML(F)$ and $\cX_{\bs_\Sigma}(\bR^\trop)$ \cite[Appendix A]{AIK24}. 



Moreover, there is a suitable identification between a subgraph of (labeled) exchange graphs of $\bs_\Sigma$ and the graph of (labeled) ideal triangulations.
Thus, we can regard mapping classes as mutation loops of $\bs_\Sigma$, and apply the results on algebraic entropy of \cite{IK19}.
First, we prove the palindromicity property for the mutation loops coming from mapping classes.

\begin{defi}\label{def:palindromicity}
We say that an invertible matrix $M$ satisfies the \emph{palindromicity property} if the characteristic polynomials of $M$ and $(M^{-1})^{\tr}$ are the same up to the overall sign.
\end{defi}

We conjectured that the presentation matrix $E_\gamma^{\bep_\gamma(w)}$ of a representation path $\gamma$ of a mutation loop at any point $w \in \cX_\bs(\bR^\trop)$ satisfies the palindromicity property \cite[Conjecture 3.13]{IK19}.
It still remains open and there is no counter-example so far.
The following is a partial affirmative result of this conjecture:

\begin{prop}\label{prop:spec_same_surf}
Let $\Sigma$ be a punctured surface.
Then the presentation matrix $E_\gamma^{\bep_\gamma(L_0)}$ of any representation path $\gamma$ of any mapping class $\phi$ at $L_0 \in \ML_0(F)$ satisfies the palindromicity property if the sign $\bep_\gamma(L_0)$ is strict.
\end{prop}

\begin{proof}
Let $\MF(\Sigma)$ denote the space of measured foliations on $\Sigma$. 
There is a homeomorphism $\ML_0(F) \xrightarrow{\sim} \MF(\Sigma)$
by collapsing the complements of a geodesic lamination (see, for example, \cite{CB}).
The space $\dMF(\Sigma) := \MF(\Sigma) \times \bR^P$ has global coordinates $\bsfa^\tri: \dMF(\Sigma) \to \bR^\tri$ associated with ideal triangulations $\tri$ (\cite{PP93}) such that the coordinate changes  $\bsfa^{\tri'} \circ (\bsfa^{\tri})^{-1}$ are given by tropical cluster $\cA$-transformations. The composite
\begin{align*}
    p: \dMF(\Sigma) \to \MF(\Sigma) \cong \ML_0(F) \subset \eML(F)
\end{align*}
is called the \emph{ensemble map}, where the first map is the projection to the first factor. It satisfies the relation
\begin{align*}
    p^\ast \sfx_\kappa^\tri = \sum_{\alpha \in \tri} \ve_{\kappa\alpha}^\tri \sfa_\alpha^\tri
\end{align*}
for all $\kappa \in \tri$. 
Let $\gamma:(\tri, \ell) \to \phi^{-1}(\tri, \ell)$ be a representation path of a mapping class $\phi$. 
For $\cF \in \dMF(\Sigma)$, we have
\begin{align*}
    (d\phi)_{\cF}^{(\tri,\ell)} = ((E_\gamma^{\bep_\gamma(L_0)})^{-1})^\tr
\end{align*}
by \cite[Remark 3.11]{IK19}, whenever $L_0:=p(\cF)$ has the strict sign.
The product structure $\dMF(\Sigma) = \MF(\Sigma) \times \bR^P$ is $MC(\Sigma)$-invariant, where the action on $\bR^P$ is via $\sigma_\Sigma$ (\cref{prop:lamination_split_action}). This gives the block decomposition
\begin{align*}
    (d\phi)_{\cF}^{(\tri,\ell)}= \begin{pmatrix}
        (d \phi^u)^{(\tri,\ell)}_{L_0} & 0\\
        0 & \sigma_\Sigma(\phi)
    \end{pmatrix}.
\end{align*}
By comparing with the decomposition \eqref{eq:block-decomp} of $E_\gamma^{\bep_\gamma(L_0)}$, we obtain the desired statement.
\end{proof}

Thanks to this proposition, we can compute the algebraic entropies $\cE^a_\phi$ and $\cE^x_\phi$ of a sign-stable mapping class $\phi$ by applying \cite[Corollary 1.2]{IK19}.

\subsection{Comparison of topological and algebraic entropies}

We obtain the following comparisons:

\begin{cor}\label{cor:cluster_Dehn_entropy}
For a simple closed curve $C$ on a marked surface $\Sigma$ satisfying the condition \eqref{eq:condition_Dehn}, we have
\[ \cE^a_{T_C} = \cE^x_{T_C} = \cE^{\mathrm{top}}_{T_C} = 0. \]
\end{cor}

\begin{proof}
For simplicity, first consider the case of a punctured surface. 
Consider the representation path $\gamma_C$ studied in \cref{lem:Dehn_SS}. Since its stable sign $\boldsymbol{\epsilon}_\gamma^\stab=(+)$ is the sign at the point $L_{T_C} \in \ML_0(F)$, we get $\cE^a_{T_C} = \cE^x_{T_C}=0$ by 
\cite[Corollary 1.2]{IK19} and \cref{prop:spec_same_surf}. Combining with \cref{thm:top_ent_Dehn_0}, we get the desired assertion. 

As another proof applicable to the general case, we can verify the vanishing of the upper bound $\log R_\phi$ appeared in \cite[Theorem 1.1]{IK19} by a direct computation using \eqref{eq:E_matrix_Dehn_twist}. 
\end{proof}

\begin{cor}\label{cor:pA_entropy}
For a punctured surface $\Sigma$ and a generic pA mapping class $\phi \in MC(\Sigma)$, we have
\[ \cE^a_{\phi} = \cE^x_{\phi} = \cE^{\mathrm{top}}_{\phi} = \log \lambda_\phi. \]
Here $\lambda_\phi >1$ is the cluster stretch factor (=pA stretch factor). 
\end{cor}

\begin{proof}
This is a direct consequence of \cref{t:pA_SS_NS}, 
\cite[Corollary 1.2]{IK19} and \cref{prop:spec_same_surf}. Indeed, just note that the representative of the attracting point $L_\phi^+$ giving the stable sign for any representation path is contained in $\ML_0(F)$. 
\end{proof}

\section{Proofs of lemmas in \cref{sec:pA_SS_NS}}

\subsection{Proof of \cref{lem:NS on X for pure}}\label{subsec:proof_1}
Let $D\Sigma$ be the closed surface obtained as the double of $\Sigma$. It is obtained by gluing two copies of the bordered surface $\Sigma^\circ$, which is obtained from $\Sigma$ by removing the small open disks $D_p$ for all $p \in P$,
with opposite orientations along the corresponding boundary components.
The surface $D\Sigma$ is equipped with an orientation-reversing involution $\iota$ which interchanges the two copies of $\Sigma^\circ$.
A hyperbolic structure $F$ as in \cref{sec:geodesic_laminations} is canonically extended to an $\iota$-invariant hyperbolic structure $DF$ on $D\Sigma$.

For each sign $\sigma = (\sigma_p)_{p \in P} \in \{+,-\}^P$,
let us consider the subset
\[ \eML^{(\sigma)}(F)
:= \{ (G,\mu) \in \eML(F) \mid \sigma_p \sigma_G(p) \geq 0 \mbox{ for } p \in P \}, \]
which is preserved by the action of $PMC(\Sigma)$.

Fix $\sigma \in \{+,-\}^P$.
Let us consider the embedding 
\begin{align*}
    D: \eML^{(\sigma)}(F) \hookrightarrow \ML_0(DF)
\end{align*}
defined as follows. For each measured geodesic lamination in $\eML^{(\sigma)}(F)$, take its $\iota$-invariant extension to the double $DF$. Then on the both sides of the boundary component $\partial_p$ for $p \in P$ in $DF$, spiralling leaves come in pair related by the involution $\iota$. We first straighten these leaves so that they hit the boundary component $\partial_p$ transversely, and then glue each $\iota$-related pair of leaves together to get a measured geodesic lamination on the double $DF$. By \cref{l:determined by signature}, the resulting map $D$ is injective. 
This construction is $PMC(\Sigma)$-equivariant, where the action on $\ML_0(DF)$ is through the group embedding 
\begin{align*}
    D:MC(\Sigma) \xrightarrow{\sim} MC(D\Sigma)^\iota \subset MC(D\Sigma).
\end{align*}
Note that for each pA mapping class $\phi \in PMC(\Sigma)$, the mapping class $D\phi$ is \lq\lq pure" in the sense of \cite{Iva}.  
Let $L_\phi^+, L_\phi^- \in \ML_0(F)$ be some representatives of the attracting and repelling points of the NS dynamics of $\phi \in PMC(\Sigma)$ on $\bP \ML_0(F)$, respectively.
Then \cite[Theorem A.1]{Iva} (and a more detailed description given in \cite[Lemma A.4]{Iva}) tells us that
\begin{align*}
[D(\phi^n(L))] = [(D\phi)^n(DL)] \xrightarrow{n \to \infty} [D(L_\phi^+)]
\end{align*}
for any $L \in \eML^{(\sigma)}(F) \setminus \bR_{\geq 0}L_\phi^-$
Since $D$ is an embedding and $\sigma$ was arbitrary, the proof is completed.

\subsection{Proof of \cref{lem:asympt to U}}\label{subsec:proof_2}
We have
\begin{align*}
    E^n(w)= \bigg(A^n(u) + \sum_{k=0}^{n-1} A^{n-k-1}BD^k(v),\  D^n(v)\bigg)
\end{align*}
for $w=(u,v) \in \bR^N=\bR^{N'} \times \bR^h$ and $n \geq 0$. Fix any $\varepsilon>0$. 

For $u \in \cC \cap \bR^{N'}$, let $\mathbf{a}(u):= \lim_{n \to \infty} \| A^n(u)/\lambda^n\|$. From the assumption (c) it converges, and there exists an integer $n(u)$ such that 
\begin{align*}
    \left\| \frac{A^n(u)}{\lambda^n} - \mathbf{a}(u)u_+\right\| <\varepsilon
\end{align*}
for all $n \geq n(u)$. 
For any $w=(u,v) \in \cC \subset \bR^N=\bR^{N'} \times \bR^h$, 
we have
\begin{align*}
    &\left\|\frac{E^n(w)}{\lambda^n}-\left( \mathbf{a}(u)+\sum_{k=0}^{n-1}
    \frac{\mathbf{a}(BD^k(v))}{\lambda^{k+1}}\right) u_+ \right\|  \\
    &\leq \left\| \frac{A^n(u)}{\lambda^n} - \mathbf{a}(u)u_+\right\| + \sum_{k=0}^{n-1}\frac{1}{\lambda^{k+1}}\left\| \frac{A^{n-k-1}(BD^k(v))}{\lambda^{n-k-1}}-\mathbf{a}(BD^k(v)) u_+ \right\| + \left\| \frac{D^n(v)}{\lambda^n}\right\|. 
\end{align*}
The first term is $\varepsilon$-small for all $n \geq n(u)$. The third term is $\varepsilon$-small for all $n \geq n_1(v)$, where $n_1(v)$ is chosen so that $\lambda^{-n_1(v)}\max_{k=0,\dots,s-1}\|D^k(v)\| <\varepsilon$. 
In order to estimate the second term, let $n_2(v):= \max_{k=0,\dots,s-1} n(BD^k(v))$ and 
\begin{align*}
    M:= \max_{m=0,\dots,n_2(v)-1}\max_{k=0,\dots,s-1} \left\| \frac{A^{m}(BD^k(v))}{\lambda^{m}}-\mathbf{a}(BD^k(v)) u_+ \right\|.
\end{align*}
For any $n > n_2(v)$, we have
\begin{align*}
    &\sum_{k=0}^{n-1}\frac{1}{\lambda^{k+1}}\left\| \frac{A^{n-k-1}(BD^k(v))}{\lambda^{n-k-1}}-\mathbf{a}(BD^k(v)) u_+ \right\| \\
    &\leq \sum_{k=0}^{n-n_2(v)-1}\frac{1}{\lambda^{k+1}}\varepsilon + \sum_{k=n-n_2(v)}^{n-1}\frac{1}{\lambda^{k+1}}M 
    \leq \frac{1}{\lambda-1}\varepsilon + \frac{M}{\lambda^n} \sum_{k=1}^{n_2(v)}\lambda^{k-1}.
\end{align*}
It is bounded from above by $(1/(\lambda-1)+1)\varepsilon$ if moreover $n \geq n_3(v)$, where $n_3(v)$ is chosen so that $\lambda^{-n_3(v)}M\sum_{k=1}^{n_2(v)}\lambda^{k-1} < \varepsilon$. Summarizing, we get
\begin{align*}
    &\left\|\frac{E^n(w)}{\lambda^n}-\left( \mathbf{a}(u)+\sum_{k=0}^{n-1}
    \frac{\mathbf{a}(BD^k(v))}{\lambda^{k+1}}\right) u_+ \right\|  < \left( 3 + \frac{1}{\lambda-1}\right)\varepsilon
\end{align*}
for all $n > \max\{n(u),n_1(v),n_2(v),n_3(v)\}$. 

Observe that the infinite sum $\sum_{k=0}^\infty \mathbf{a}(BD^k(v))/\lambda^{k+1}$ converges since $D$ has finite order, and thus we get
\begin{align*}
    \left\|\frac{E^n(w)}{\lambda^n}-\left( \mathbf{a}(u)+\sum_{k=0}^{\infty}
    \frac{\mathbf{a}(BD^k(v))}{\lambda^{k+1}}\right) u_+ \right\| <\left( 4 + \frac{1}{\lambda-1}\right)\varepsilon
\end{align*}
for sufficiently large $n$. Hence $[E^n(w)] \to [u_+]$ as desired.
\appendix
\section{Signed mapping classes}\label{sec:Appendix}

\subsection{Cluster modular groups}\label{sec:cluster_modular_group}
First, we confirm the definition of the cluster modular group in our context.

Let $\bs$ be a mutation class with index set $\widetilde{I} := I \sqcup I_\f$.
Here the seed mutations are allowed only at the indices in $I$.
The labeled exchange graph $\bExch_\bs$ is the graph whose vertices are the labeled seeds in $\bs$, and two vertices are connected by an edge with label $\pi$ if the corresponding seeds are related by a seed mutation $\mu_\pi$ for $\pi \in I$ or a relabeling by a transposition $\pi = (i\ j) \in \fS_{I} \times \fS_{I_\f}$.

Let $\mathrm{Mat}_\bs$ denote the mutation class of matrices underlying $\bs$.
Then, we have a map $B^\bullet: V(\bExch_\bs) \to \mathrm{Mat}_\bs$ which sends a vertex $v \in \bExch_\bs$ to the exchange matrix $B^{(v)}$ of the seed associated to $v$.
Then, the \emph{cluster modular group} $\Gamma_\bs$ is the group consisting of graph automorphisms $\phi$ of $\bExch_\bs$ which preserve the fibers of $B^\bullet$ and the labels on the edges.
An element of $\Gamma_\bs$ is called a \emph{mutation loop}.

Some other basics of cluster algebras are summarized in \cite[Appendix A.1]{IK24_lam}. We also refer the reader to \cite[Appendix A]{AIK24} for the mutation class $\bs_\Sigma$ associated to a marked surface $\Sigma$.
In this paper, $\X_\bs$ stands for the cluster $\X$-variety without frozen coordinates. 
We set $\cU_\bs := \im p \subset \cX_\bs$.
Then, the following subsequence of \cite[(A.6)]{IK24_lam} for $\bA = \bR$ and $\bs = \bs_\Sigma$ is equivalent to \eqref{eq:lam_exact_seq}:
\begin{align*}
     1 \to \cU_\bs(\bR^T) \to \cX_\bs(\bR^T) \xrightarrow{\theta_\bs} H_{\cX}(\bR) \to 1.
\end{align*}


\subsection{Signed mapping classes}
As we mentioned in \cref{sec:alg_entropy}, the graph of (labeled) ideal triangulations is identified with the subgraph of (labeled) exchange graph of the mutation class $\bs_\Sigma$.
For this reason, the mapping class group $MC(\Sigma)$ is a proper subgroup of the cluster modular group $\Gamma_\Sigma:=\Gamma_{\bs_\Sigma}$ in general.

Except for some cases, we can recover the entire of the cluster modular group by extending the mapping class group as follows:
\begin{align*}
    MC^\pm(\Sigma) := \begin{cases}
    MC(\Sigma) & \mbox{if $\Sigma$ is a punctured surface with exactly one puncture}, \\    
    MC(\Sigma) \ltimes (\bZ/2)^P & \mbox{otherwise}.
    \end{cases}
\end{align*}
Here,
in the second case, 
the group structure is given by
\begin{align}\label{eq:sgn_mc_mult}
    (\phi_1,\varsigma_1)\cdot(\phi_2,\varsigma_2):=(\phi_1 \cdot \phi_2,\ (\phi_2^*\varsigma_1) \cdot \varsigma_2)
\end{align}
for $(\phi_i,\varsigma_i) \in MC(\Sigma)\times (\bZ/2)^P$, where $(\phi^* \varsigma)_p := \varsigma_{\phi(p)}$ and $\varsigma \cdot \varsigma' = (\varsigma_p \varsigma'_p)_p$.
This group $MC^\pm(\Sigma)$ is called \emph{signed mapping class group}.

Also, Fomin--Shapiro--Thurston generalize the ideal triangulations, called \emph{tagged triangulations}, to fill the exchange graph of $\bs_\Sigma$.
Actually, they give a bijection between the graph of tagged triangulations and the exchange graph \cite[Theorem 7.11]{FST}.
The signed mapping class group acts on the graph of tagged triangulations as follows:

A tagged triangulation is a suitable equivalence class of a pair of an ideal triangulation and a sign $\tau = (\tau_p)_{p \in P}$ for punctures \cite[Section 8]{BS15}.
For a tagged triangulation $(\tri, \tau)$ and a signed mapping class $(\phi, \varsigma)$, we define
\begin{align*}
    (\phi, \varsigma)^{-1} (\tri, \tau) := (\phi^{-1}(\tri),\ \varsigma \cdot \phi^* \tau).
\end{align*}

Therefore, we get the group homomorhpsim $MC^\pm(\Sigma) \to \Gamma_\Sigma$.

\begin{thm}
\label{t:cluster modular_surface case}
Let $\Sigma$ be a marked surface.
If $\Sigma$ is not a sphere with 4 punctures, then the homomorphism
$MC^\pm(\Sigma) \to \Gamma_\Sigma$
is an isomorphism.
Otherwise, the image is a subgroup of index $\geq 2$.
\end{thm}

\begin{proof}
Except for the following cases, the statement is proved in \cite{BS15}:
\begin{enumerate}[(a)]
    \item a sphere with 4 punctures;
    \item a sphere with 5 punctures;
    \item a once-punctued disk with 2 or 4 marked points on its boundary;
    \item a twice-punctued disk with 2 marked points on its boundary;
    \item an unpunctured disk with 4 marked points on its boundary;
    \item an annulus with once marked point on each boundary component;
    \item a once-punctured torus.
\end{enumerate}
In the case (a), there is a mutation loop which is not contained in $MC(\Sigma) \ltimes (\bZ/2)^h$ (see \cite[Graph 2]{Gu11}). 
Nevertheless, the injectivity of $MC^\pm(\Sigma) \to \Gamma_\Sigma$ still holds (\emph{e.g.}, \cite{ASS12}). The other cases are remedied as follows.

The case (b) was excluded only for the issue of existence of a non-degenerate potential of a quiver corresponding to an ideal triangulation, so it does not matter for our setting.
In the cases (c) and (d), we have a mutation loop which is not contained in $MC^\pm(\Sigma)$ when forgetting frozen indices.
However, in our setting, those are not mutation loops since they do not preserve the frozen part of the exchange matrix. 
In the cases (e) and (f), the injectivitiy of $MC^\pm(\Sigma) \to \Gamma_\Sigma$ does not hold when forgetting frozen indices. 
It is also not an issue when considering frozen indices, since the relevant mutation loops are distinguished by their action on frozen coordinates.

In the case (g), the quotient of the mapping class group by the \emph{hyperelliptic involution} is isomorphic to the automorphism group of the mapping class groupoid, which is nothing but the cluster modular group. See \cite[Chapter 4, Section 4 and Chapter 5, Section 1]{Penner}.
\end{proof}

The aim of this appendix is to extend our main result \cref{t:pA_SS_NS} for signed mapping classes:

\begin{thm}\label{t:sgn_pA_NS_SS}
Let $\Sigma$ be a punctured surface.
For a signed mapping class $(\phi, \varsigma) \in MC^\pm(\Sigma)$, the following conditions are equivalent:
\begin{enumerate}
    \item $\phi$ is generic pseudo-Anosov.
    \item $(\phi, \varsigma)$ is uniformly sign-stable.
    \item $(\phi, \varsigma)$ has NS dynamics on $\bP \eML(F)$ and the attracting and repelling points are $\cX$-filling.
\end{enumerate}
In this case, the cluster stretch factor of $(\phi, \varsigma)$ also coincides with the pA stretch factor of $\phi$.
\end{thm}

First, we summarize the arguments in \cref{sec:(1)<=>(2)}:

\begin{prop}\label{prop:mutation_loop_NS_SS}
Let $\phi \in MC(\Sigma)$ is a mapping class. 
\begin{itemize}
    \item[(i)] 
    Assume that there exists an integer $r >0$ such that the mapping class $\phi^r$ has NS dynamics on $\bP \eML(F)$, and the attracting/repelling points $[L_\phi^\pm]$ are $\X$-filling and fixed by $\phi$.
    Then $\phi$ is uniformly sign-stable.
    \item[(ii)] Conversely, if $\phi$ is uniformly sign-stable, 
    then any positive power $\phi^r$ never have a fixed point $[L] \in \bP\eML_\bQ(F)$ which is not $\X$-filling. In particular, the mapping class $\phi$ is cluster-pA.
\end{itemize}
\end{prop}
The statement (ii) is slightly stronger than \cref{prop:uniform_SS_cluster_pA}, since a point contained in a face of a cone of the fan $\mathfrak{F}_\Sigma^+$ is not $\X$-filling but the converse is not true. For instance, the simple closed curve $C$ in \cref{fig:Dehn_genus_2} is not $\X$-filling but does not belong to any face of a cone of the fan.

\begin{proof}[Proof of \cref{t:sgn_pA_NS_SS}]
(1) $\Longleftrightarrow$ (2): 
We can generalize \cref{prop:mutation_loop_NS_SS} for signed mapping classes since the action of $(\phi, \varsigma)$ on $\ML_0(F)$ is equal to the action of $\phi$, and there exist $s > 0$ such that $(\phi, \varsigma)^s = (\phi^s, \mathrm{id})$ by \eqref{eq:sgn_mc_mult}.
The equivalence is a direct consequence of this.

(1) $\Longleftrightarrow$ (3):
With the same notice, this equivalence can be proven by the same argument of \cref{sec:(1)<=>(3)}.
\end{proof}

Our results \cref{t:pA_SS_NS,t:sgn_pA_NS_SS} is proved by taking two steps $(1) \Longleftrightarrow (2)$ and $(1) \Longleftrightarrow (3)$.
Also note that under the condition $(1)$, the condition $(2)$ is equivalent to the following stronger condition:
\begin{itemize}
    \item[($2^\#$)] The signed mapping class $(\phi, \varsigma)$ is uniformly sign-stable and cluster stretch factor is greater than $1$.
\end{itemize}
We expect that there is a direct proof of $(2^\#) \Longleftrightarrow (3)$.
As a generalization, we conjecture the following:

\begin{conj}
Let $\bs$ be any mutation class and $\phi \in \Gamma_\bs$.
Then $\phi$ is uniformly sign-stable with cluster stretch factor greater than $1$ if and only if $\phi$ has NS dynamics on $\bP \cX_\bs(\bR^\trop)$ with $\cX$-filling attracting/repelling points.
\end{conj}
The stronger condition $(2^\#)$ excludes the following example
which is uniformly sign-stable by but it does not have NS dynamics.

\begin{ex}
Let us continue to deal the Dehn twist in \cref{ex:Dehn_path}.
As we see in \cref{ex:Dehn_SS}, this mapping class has a representation path which is sign-stable on the entire space $\eML(F)$.
Since $C$ is $\cX$-filling, the Dehn twist $T_C$ is uniformly sign-stable.
However, its cluster stretch factor is 1.
\end{ex}



\end{document}